\documentclass[11pt]{article}
\usepackage[american]{babel}
\usepackage{amsmath}
\usepackage{dsfont,mathtools,amssymb}
\usepackage{color}
\mathtoolsset{showonlyrefs,showmanualtags}

\usepackage[
backend=bibtex,
style=alphabetic,
sorting=nty,
maxbibnames=99
]{biblatex}

\usepackage{hyperref}
\usepackage[dvipsnames]{xcolor}
\hypersetup{
	colorlinks=true,
	pdfpagemode=UseNone,
    citecolor=OliveGreen,
    linkcolor=NavyBlue,
    urlcolor=Magenta,
	pdfstartview=FitW
}

\usepackage{amsthm}
\usepackage[margin = 1in]{geometry}
\usepackage{bbm}
\usepackage{bm}
\usepackage{verbatim}
\usepackage{graphicx}
\usepackage{color}
\usepackage{tikz}
\usetikzlibrary{shapes,arrows,fit}
\usepackage{subfigure}

\usepackage{thmtools}
\usepackage{thm-restate}

\newtheorem{theorem}{Theorem}[section]

\newtheorem{lemma}[theorem]{Lemma}

\numberwithin{equation}{section}
\numberwithin{figure}{section}

\theoremstyle{definition}
\newtheorem{definition}[theorem]{Definition}

\theoremstyle{remark}
\newtheorem{remark}[theorem]{Remark}

\newcommand{\Av}{\mathrm{Av}}

\renewcommand{\rho}{\varrho}
\renewcommand{\phi}{\varphi}

\newcommand{\R}{\mathbb{R}}
\newcommand{\B}{\mathcal{B}}

\newcommand{\integers}{\mathbb{Z}}
\newcommand{\IND}{{\bf 1}}
\newcommand{\scalar}[2]{\langle #1 , #2\rangle}
\newcommand{\tmix}{T_{\rm mix}}
\newcommand{\poly}{\mathrm{poly}}

\DeclareMathOperator{\var}{Var}
\DeclareMathOperator{\cov}{Cov}

\DeclareMathOperator{\dist}{dist}

\DeclareMathOperator{\diag}{diag}
\DeclareMathOperator{\supp}{supp}

\DeclareMathOperator{\Ent}{Ent}

\newcommand{\eps}{\epsilon}
\newcommand{\be}{\begin{equation}}

\newcommand{\cC}{\ensuremath{\mathcal C}}
\newcommand{\cD}{\ensuremath{\mathcal D}}

\newcommand{\cS}{\ensuremath{\mathcal S}}

\newcommand{\bbE}{{\ensuremath{\mathbb E}} }

\newcommand{\bbN}{{\ensuremath{\mathbb N}} }

\newcommand{\bbP}{{\ensuremath{\mathbb P}} }

\newcommand{\bbR}{{\ensuremath{\mathbb R}} }

\newcommand{\bbZ}{{\ensuremath{\mathbb Z}} }

\newcommand{\si}{\sigma}
\newcommand{\ent}{{\rm Ent} }

\newcommand{\tc}{\, |\, }

\let\a=\alpha \let\b=\beta   \let\d=\delta  
 \let\g=\eta     \let\k=\kappa  
\let\m=\mu            
  \let\s=\sigma \let\t=\tau   
  
\let\D=\Delta   \let\G=\Gamma  \let\L=\Lambda 
\let\O=\Omega

\def\({\left(}
\def\){\right)}
\newcommand{\hamming}[2]{d_{\mathrm{H}}\left( #1, #2 \right)}
\newcommand{\tv}[2]{d_{\mathrm{TV}}\left( #1, #2 \right)}

\newcommand{\sgn}{\mathrm{sgn}}

\newcommand{\VSpairs}{\mathcal{X}}
\newcommand{\vx}{x}
\newcommand{\vy}{y}
\newcommand{\spa}{a}
\newcommand{\spb}{a'}
\newcommand{\ALOi}{J}
\newcommand{\Pinning}{\mathcal{T}}

\newcommand{\loc}{\Phi}
\newcommand{\ctrct}{\kappa}

\newcommand{\select}{p}

\newcommand{\Cgbf}{C_{\textsc{gbf}}}
\newcommand{\Cubf}{C_{\textsc{ubf}}}
\newcommand{\Cellubf}{C_{\textsc{ubf}}}

\newcommand{\Joint}{\O_{\textsc{j}}}

\date{\today}

\title{On Mixing of Markov Chains: Coupling, Spectral Independence, and Entropy Factorization}

\author{
	Antonio Blanca\thanks{Department of Computer Science and Engineering, Penn State, University Park, PA 16801, USA.
		Email: {ablanca@cse.psu.edu}.
		Research supported in part by NSF grant CCF-1850443.}
	\and
	Pietro Caputo\thanks{Department of Mathematics, University of Roma Tre, Largo San Murialdo 1, 00146 Roma, Italy.
	Email: \{pietro.caputo,daniel.parisi@uniroma3.it\}}
\and
	Zongchen Chen\thanks{School of Computer Science, Georgia Institute of Technology, Atlanta, GA 30332, USA.
	Email: \{chenzongchen,vigoda\}@gatech.edu.
	Research supported in part by NSF grant CCF-2007022.}
	\and
	Daniel Parisi$^\dag$
\and
Daniel \v{S}tefankovi\v{c}\thanks{Department of Computer Science, University of Rochester, Rochester, NY 14627, USA.
	Email: stefanko@cs.rochester.edu.
	Research supported in part by NSF grant CCF-2007287.}
		\and
	Eric Vigoda$^\ddag$
}

\begin{document}

\maketitle

\begin{abstract}
For general spin systems, we prove that a contractive coupling for any local Markov chain
implies optimal bounds on the mixing time and the
modified log-Sobolev constant for a large class of Markov chains including the Glauber dynamics, arbitrary heat-bath block dynamics, and the
Swendsen-Wang dynamics.  This reveals a novel connection between probabilistic techniques for bounding the convergence to stationarity and
%the coupling method, which is
%a probabilistic technique for bounding the convergence rate, with
analytic tools for analyzing the decay of relative entropy.  As a %immediate
corollary of our general results, we obtain $O(n\log{n})$
mixing time and $\Omega(1/n)$ modified log-Sobolev constant of the Glauber dynamics
for sampling random $q$-colorings of an $n$-vertex graph with constant maximum degree $\Delta$ when $q > (11/6 - \eps_0)\Delta$
for some fixed $\eps_0>0$.  We also obtain $O(\log{n})$ mixing time and $\Omega(1)$ modified log-Sobolev constant of
the Swendsen-Wang dynamics for the ferromagnetic Ising model on an $n$-vertex graph of
constant maximum degree when the parameters of the system lie in the tree uniqueness region.
At the heart of our results are new techniques for establishing spectral independence of the spin system and block factorization of the relative entropy.  On one hand we prove
that a contractive coupling of a local Markov chain implies spectral independence of the Gibbs distribution.  On the other hand we show that
spectral independence implies factorization of entropy for arbitrary blocks, establishing optimal bounds
on the modified log-Sobolev constant of the corresponding block dynamics.
\end{abstract}

\thispagestyle{empty}

\newpage

\setcounter{page}{1}

\section{Introduction}

 Spectral independence is a powerful new approach for proving fast
 convergence of Markov chain Monte Carlo (MCMC) algorithms.
 The technique was introduced by Anari, Liu, and Oveis Gharan~\cite{ALO20}
 to establish rapid mixing of the Glauber dynamics by utilizing high-dimensional expanders.
For a spin system defined on a graph $G=(V,E)$, the Glauber dynamics
 is the simple single-site update Markov chain which updates the spin at a randomly
 chosen vertex in each step.   The mixing time is the number of steps to reach
 close to the stationary distribution.

 Our paper addresses two broad questions.
 First, what are the implications of spectral independence?  In particular, does it imply
 fast convergence for other Markov chains beyond the simple Glauber dynamics?
 We prove that it does: we
show that spectral independence implies optimal mixing time
bounds and modified log-Sobolev constants
for a broad class of chains, including all possible
heat-bath block dynamics and the Swendsen-Wang dynamics.
 Our proof utilizes recent work on entropy factorization~\cite{CP20}.

 Our second question is when does spectral independence hold, and how
 does it relate to traditional proof approaches, such as coupling techniques?
% Here again we prove a general result, showing that standard coupling
% proofs imply spectral independence.
Here again we prove a general result, showing that a contractive coupling
for any \emph{local} Markov chain implies
spectral independence. %\antonio{Adjusted this sentence.}.
 This immediately yields
 stronger than state of the art mixing time bounds for a variety of chains.  In addition, it provides
 an intriguing conceptual connection between the coupling method and modified log-Sobolev inequalities as we describe below.

There are two broad approaches for establishing fast convergence of
 MCMC algorithms: probabilistic or analytic techniques.
Probabilistic techniques primarily utilize the coupling method; a popular example is the
path coupling method which has become a fundamental tool in theoretical computer science~\cite{BubleyDyer}.
In contrast, analytic techniques establish decay to equilibrium by means of functional inequalities such as Poincar\'e or
log-Sobolev inequalities, which correspond to decay of variance and relative entropy respectively.
In particular, the so-called modified log-Sobolev inequality is often a powerful analytic tool in establishing
tight bounds on the mixing time, while the weaker Poincar\'e inequality provides control on the spectral gap; see, e.g.,  \cite{DiaSal96,Martinelli99,BT06}.  

These two approaches---probabilistic or analytic---appeared disparate. While coupling techniques have been used to prove  Poincar\'e inequalities, there are no clear relations  between the probabilistic approach and log-Sobolev inequalities.
We establish a strong connection by proving %, in a quite general context,
that coupling inequalities in the form of
bounds on the Ollivier-Ricci curvature of the Markov chain imply
entropy decay, and hence the associated %log-Sobolev and
modified log-Sobolev inequality holds; see
Section~\ref{sec:prelim} for definitions. In the context of spin systems on bounded-degree graphs, this settles a remarkable  %long-standing
conjecture of Peres and Tetali (see Conjecture 3.1 in~\cite{ELL} and Remark~\ref{Remark:PTconj}).

Our technical contributions apply in the general setting of $q$-state spin systems.  This is a convenient setting to capture a wide family of
distributions defined on graphs, including the equilibrium distribution of undirected graphical models.
We now introduce some relevant notation and refer to  Section~\ref{sec:prelim} for a formal definition of  general spin systems.
Let $G=(V,E)$ be a $n$-vertex graph and let $\Delta$ denote the maximum degree of $G$.
For integer $q\geq 2$ the state space of the model is the set $\Omega=\{\sigma\in [q]^V:\mu(\sigma)>0\}$ of
assignments with positive weight in the Gibbs distribution~$\mu$.

Canonical examples of a spin system
include the Ising model (with $q=2$ spin values) and the Potts model (with $q\geq 3$);
in these models, for an inverse temperature parameter
%The simplest example of a spin system
%is the Ising model (with $q=2$ spin values) and Potts model ($q\geq 3$);
%for inverse temperature
$\beta$, a configuration $\sigma\in\Omega$
has probability $\mu(\sigma) \propto \exp(\beta M(\sigma))$
where $M(\sigma)$ is the number of edges of $G$ which are monochromatic in $\sigma$.
The Ising/Potts model is ferromagnetic when $\beta>0$ and antiferromagnetic when $\beta<0$.
The hard-core model is another spin system defined on the set of independent sets of $G$ weighted by a parameter $\lambda>0$;
each independent set $\sigma$ has probability proportional to $\lambda^{|\sigma|}$ in the Gibbs distribution.
The $q$-colorings model, where the Gibbs distribution is uniform over the collection of proper vertex $q$-colorings of $G$, is also a classical spin system.
%
%The hard-core model is defined on the set of independent sets of $G$ weighted by a parameter $\lambda>0$
%where the independent set $\sigma$ has probability proportional to $\lambda^{|\sigma|}$.
%In the $k$-colorings model the Gibbs distribution is uniform over the collection of proper vertex $k$-colorings of $G$.

The Glauber dynamics is the simplest MCMC approach for sampling from the Gibbs distribution~$\mu$.
The transitions of the Markov chain $(X_t)$
update the spin at a randomly chosen vertex in each step.  From $X_t\in\Omega$, we choose
a random vertex $\vx$, set $X_{t+1}(\vy)=X_t(\vy)$ for all $\vy\neq \vx$, and
the spin $X_{t+1}(\vx)$ is chosen from the marginal distribution at $\vx$ conditional on the current spins on $N(\vx)$,
the neighborhood of $\vx$.   The mixing time is the number of steps, from the worst initial state, to get
close to the stationary distribution; see Section~\ref{sec:prelim}.

Our results apply more broadly to the general class of heat-bath block dynamics.
Let $\B=\{B_1,\dots,B_\ell\}$ be any collection of sets (or blocks) such that $V=\cup_i B_i$ and  let $\alpha=(\alpha_B)_{B\in\B}$ be a
probability distribution on $\B$.
A step of the heat-bath block dynamics operates by choosing a block $B$ with probability $\alpha_B$
and updating the configuration in $B$ with a sample from the Gibbs distribution conditional on the
configuration on $V\setminus B$.  Note that the Glauber dynamics corresponds to setting the blocks
to individual vertices with uniform weights, and for a bipartite graph the even-odd chain (also known as 
the alternating scan dynamics) corresponds to uniform weighting
for two blocks corresponding to the two parts.
%\antonio{I think it is better not to mention even-odd here, but if we do we should define it properly.}. 
By extending the weight to $\alpha_B=0$ if $B\notin\B$ we think of  $\alpha $ as a distribution over all subsets of $V$ and speak of the $\alpha$-weighted heat-bath block dynamics.

Given $\alpha$, define the minimum ``coverage probability'' of a vertex by
\begin{equation}\label{eqn:delta}
\delta = \delta(\alpha) = \min_{\vx\in V} \sum_{B:B\ni \vx} \alpha_B.
\end{equation}
%denote the minimum ``coverage probability'' of a vertex.
We say that the block dynamics have \emph{optimal mixing} when there exists a constant $C$ such that for all weights $\alpha$ the mixing time of the $\alpha$-weighted heat-bath block dynamics is at most $C\delta(\alpha)^{-1}\log{n}$. Similarly, we say that  the block dynamics have
\emph{optimal entropy decay} if the modified log-Sobolev constant of the $\alpha$-weighted heat-bath block dynamics is at least $\delta(\alpha)/C$.
Note that the constant $C$ may depend on the parameters defining the spin system and on the maximum degree $\Delta$, but it does not depend on $n$ and it is independent of the choice of weights $\alpha$. In this generality, these bounds are optimal up to the value of the constant $C$. Indeed,   for the Glauber dynamics we have $\delta(\alpha)=1/n$ and the mixing time matches the $\Omega(n\log{n})$ lower bound established by Hayes and Sinclair~\cite{HayesSinclair}
for bounded-degree graphs.
Moreover,  by restricting to test functions of a single spin it is not hard to check that the spectral gap of the $\a$-weighted block dynamics is always at most $\delta(\alpha)$, and therefore the lower bound $\delta(\alpha)/C$ on the modified log-Sobolev constant of the block dynamics is optimal up to the multiplicative constant $1/C$; see e.g.\ \cite{BT06} for standard relations between spectral gap and modified log-Sobolev constant.

\subsection{Applications}

We begin with a few examples of applications of our results.  We then delve into our general technical contributions in subsequent subsections.
We note that all of these applications follow immediately from previous coupling proofs together with our new technical contributions.

For $q$-colorings of graphs with maximum degree $\Delta$, Jerrum~\cite{Jerrum} proved that the Glauber dynamics
has $O(n\log{n})$ mixing time when $q>2\Delta$.
Jerrum's result was improved to $q>\frac{11}{6}\Delta$ in~\cite{Vig00} and further improved
to $q>(\frac{11}{6}-\eps_0)\Delta$ for some small $\eps_0\approx 10^{-5}>0$ by
Chen et al.~\cite{CDMPP19} by analyzing a Markov chain referred to as the flip dynamics; this implied
$O(n^2)$ mixing time of the Glauber dynamics.  We obtain $O(n\log{n})$ mixing time of the Glauber dynamics,
which is asymptotically optimal~\cite{HayesSinclair}, and also obtain optimal bounds on the log-Sobolev and
modified log-Sobolev constants.

\begin{theorem}\label{thm:colorings}
%Let $\eps_0\approx 10^{-5}>0$ be a fixed constant.
For $q$-colorings on an $n$-vertex graph of maximum degree $\Delta$, when $q>(\frac{11}{6}-\eps_0)\Delta$, where $\eps_0\approx 10^{-5}>0$ is a fixed constant,
the Glauber dynamics has mixing time $O(n\log{n})$ and log-Sobolev and modified log-Sobolev constants $\Omega(1/n)$.
More generally, under these assumptions all block dynamics %for arbitrary weighted blocks
have optimal mixing and optimal entropy decay. %modified log-Sobolev constant.
\end{theorem}

For the ferromagnetic Ising model, Mossel and Sly~\cite{MS} established optimal mixing time bounds
of $O(n\log{n})$ for the Glauber dynamics on any graph of maximum degree $\Delta$ in the tree uniqueness region; that is, for all $\beta< \beta_c(\Delta)$, where  $\beta_c(\Delta) := \ln(\frac{\Delta}{\Delta-2})$ is the threshold of the uniqueness/non-uniqueness phase transition
on the $\Delta$-regular tree. Our general results allow us to extend this to arbitrary heat-bath block dynamics and to the Swendsen-Wang dynamics~\cite{SW}.
The latter is a particularly interesting Markov chain which utilizes
the random-cluster representation of the ferromagnetic model to perform global updates in a single step; see
Section~\ref{sec:SW} for its definition.  This non-local nature
makes tight analysis of the Swendsen-Wang dynamics challenging.
In~\cite{BCV},
it was shown that the mixing time of
Swendsen-Wang dynamics on any graph of maximum degree $\Delta$ in the tree uniqueness region is $O(n).$
Our general results imply a bound of $O(\log{n})$ on the mixing time of the Swendsen-Wang dynamics and a bound of $\Omega(1)$
on the corresponding modified log-Sobolev constant
in the same tree uniqueness region. As shown in \cite{BCPSV20} for the special case of the $d$-dimensional integer lattice $\bbZ^d$, these estimates are optimal up to a multiplicative constant. Our results also yield new optimal bounds on the log-Sobolev and modified log-Sobolev constants for
the Glauber dynamics in the same setting.

%
%a sophisticated Markov chain of particular interest is the Swendsen-Wang dynamics~\cite{SW}.
%The Swendsen-Wang dynamics utilizes
%the random-cluster representation of the ferromagnetic model to perform global updates in a single step; see
%Section~\ref{sec:SW} for its definition.  This global nature
%makes tight analysis of the Swendsen-Wang dynamics challenging.  Mossel and Sly~\cite{MS} established optimal mixing time bounds
%of $O(n\log{n})$ for the Glauber dynamics on any graph of maximum degree $\Delta$ in the tree uniqueness region;
%we note that analogous bounds on the log-Sobolev and modified log-Sobolev constants were not
%known.
%
%Here we establish an optimal bound of $O(\log{n})$ on the mixing time of the Swendsen-Wang dynamics and an optimal bound of $\Omega(1)$
%on the modified log-Sobolev constant of the Swendsen-Wang dynamics (and also optimal bounds on the log-Sobolev and modified log-Sobolev for
%the Glauber dynamics) in the identical tree uniqueness region, that is for all $\beta< \beta_c(\Delta)$, where  $\beta_c(\Delta) := \ln(\frac{\Delta-2}{\Delta})$ is the threshold of the uniqueness/non-uniqueness phase transition
%on the $\Delta$-regular tree.

Moreover, we obtain improved results for the ferromagnetic Potts model.  Unlike the Ising model, for the ferromagnetic Potts model
known rapid mixing results for the Glauber dynamics do not reach the tree uniqueness threshold.  The best known results~\cite{Hayes,Ullrich,BGP}
imply that
%for the Glauber dynamics for general graphs are the following pair of works.
%Ullrich~\cite[Corollary 2.14]{Ullrich} (see \cite[Observation 11]{Hayes} for the particular case of the Ising model)
the Glauber dynamics mixes in $O(n\log{n})$ steps when $\beta<\beta_0$ where $\beta_0=\max\big\{ \frac{2}{\Delta},\frac{1}{\Delta}\ln(\frac{q-1}{\Delta}) \big\}$. In addition,~\cite{BGP} showed $\poly(n)$ mixing of the Glauber dynamics
%\eric{Changed to $\poly(n)$ mixing and added block size.}
for $\beta< \beta_1$ where $\beta_1=(1-o(1))\frac{\ln q}{\Delta-1}$, the $o(1)$ term tends to $0$ as $q\rightarrow\infty$; see Remark~\ref{rmk:BGP} for more details.
These results yield polynomial mixing time bounds for the Swendsen-Wang dynamics in the corresponding regimes of $\beta$.
Note the critical point for the uniqueness threshold on the tree was established by H\"{a}ggstr\"{o}m~\cite{Haggstrom}
and it behaves as $\beta_u = \frac{\ln q}{\Delta-1} + O(1)$; see \cite{BGP}.
In both regimes, we prove optimal bounds for the mixing time and (modified) log-Sobolev constant of the Glauber dynamics and
also for the Swendsen-Wang dynamics. %\antonio{Adjusted. Check.}

\begin{theorem}\label{thm:Ising-Potts-mixing}
For the \emph{ferromagnetic Ising model} with %inverse temperature
$\beta<\beta_c(\Delta)$ on any $n$-vertex
graph of maximum degree $\Delta\geq 3$,  all heat-bath block dynamics have optimal mixing and optimal entropy decay, and the Swendsen-Wang dynamics has optimal mixing time $O(\log n)$ and optimal modified log-Sobolev constant $\O(1)$. %times and modified log-Sobolev constants.
For the \emph{ferromagnetic Potts model} the same results hold when $\beta < \max\{\beta_0,\beta_1\}$.
\end{theorem}

\subsection{Spectral independence definitions}
\label{sec:intro-spectral-independence}

A central concept in our work is \emph{spectral independence}, which was introduced by
Anari, Liu and Oveis Gharan~\cite{ALO20} to establish polynomial mixing time bounds for the Glauber dynamics.
To formally define spectral independence it will be important to consider the effect of \emph{pinnings} which can
informally be viewed as boundary conditions.
For $U\subset V$, let $\Omega_U = \{\tau\in [q]^U: \exists \sigma\in\Omega, \sigma_U=\tau\}$ denote
the set of assignments to $U$ with valid extensions on the remaining vertices.
%Moreover, for $v\in V$, let $\Omega_U^v = \{s\in [q]: \exists \sigma\in\Omega_U, \sigma(v)=s\}$ denote
%the valid spin assignments for vertex $v$.
In particular, $\Omega_\vx$ denotes the set of all valid spin assignments for the vertex $\vx$ under $\mu$.
A {pinning} is a fixed assignment $\tau$ on some $U\subset V$ where $\tau\in\Omega_U$. We write $\mu^\t$ for the Gibbs measure $\mu(\cdot\tc\sigma_U=\t)$ obtained by conditioning on the given $\t$.  In the presence of a pinning $\tau$ on $U\subset V$, the definition of the
Glauber dynamics
%\antonio{We should clarify how the Glauber is defined with pinnings} 
remains the same with the assignment $\tau$ on $U$ fixed (see Remark~\ref{rmk:glauber-pinning} for a definition).
Let $\Pinning = \cup_{U \subset V} \Omega_U$ denote the collection of all pinnings,
and $\VSpairs = \{(\vx,\spa): \vx \in V, \spa \in \Omega_\vx\}$ for the set of all feasible vertex-spin pairs.

 The spectral independence approach considers
the following matrix which captures the pairwise influence of vertices.
For a pair of vertices $\vx,\vy$ and a pair of spins $\spa,\spa'$, it is the influence of the spin $\spa$ at $\vx$ on the marginal probability of $\spa'$ at $\vy$.

\begin{definition}[ALO influence matrix]\label{def:J}
The \emph{ALO influence matrix} $\ALOi \in \R^{\VSpairs \times \VSpairs}$ is defined by $\ALOi(\vx, \spa; \vx, \spa') = 0$ and
\[
\ALOi(\vx, \spa; \vy, \spa') = \mu(\si_\vy = \spa' \mid \si_\vx = \spa) - \mu(\si_\vy = \spa') \quad \text{~for~} x \neq y.
\]
Moreover, for a pinning $\t \in \Pinning$, $J^\t$ denotes the influence matrix with respect to the conditional measure $\mu^\t$.
%\eric{Should restrict to $\si_y\in\Omega^\tau_y,\si_x\in\Omega^\tau_x$, right?}
\end{definition}

Note that \cite{ALO20} defined the influence matrix only for $q=2$ in a slightly different form and the definition was later generalized to all $q\ge 2$ by two independent works \cite{CGSV21,FGYZ21} in different ways. 
In this paper we use the definition from \cite{CGSV21} which is more suitable for our applications in Section~\ref{sec:contraction-SI} for establishing spectral independence, though the definition from \cite{FGYZ21} could also work with some efforts. 
Since $J$ is self-adjoint the eigenvalues of $J$ are real; see Eq.~\eqref{self-adjoint} below for more details.
Let $\lambda_1(J)\geq 0$ denote its largest eigenvalue (the eigenvalue zero always exists since all row sums of $J$ vanish).

\begin{definition}[Spectral independence]
We say that a spin system is {\em $\eta$-spectrally independent} if for all pinnings $\t \in \Pinning$ we have
$\lambda_1(J^\t)  \leq \eta.$
\end{definition}

%\begin{remark}
%%The notion of spectral independence can be generalized to arbitrary vertex weights as considered in~\cite{CLV1}.
%One way of establishing spectral independence is by considering weighted sum of absolute influences under arbitrary vertex weights as is done in~\cite{CLV1}.
%More specifically, if there exist vertex weights $\rho:V\rightarrow \R_+$ and a constant $\eta$ such that,
%for all $x,a$, we have
%$$\sum_{y\in V,\si_y} \rho_v \cdot |J^\t(x, \si_x; y, \si_y)|\leq \eta$$
%%then $\eta$-spectral independence holds.
%then $\lambda_1(J^\t)  \leq \eta.$.
%\end{remark}

There is one additional property of the Gibbs distribution
that will be relevant to us; namely,
that the marginal probability for any vertex is lower bounded by a constant $b$.
This property is typically trivial to satisfy for some constant $b=b(\Delta)>0$. We write $\Omega^\tau_x$ for the set of spin values that are allowed at $x$ in the presence of the pinning $\t$.

\begin{definition}[Marginal boundedness]\label{def:margb}
We say that the spin system is {\em $b$-marginally bounded} if for all pinnings $\t$, all $\vx\in V$, all $\spa\in\Omega^\tau_x$ we have
$\mu^\tau(\sigma_\vx =\spa)\geq b$.
\end{definition}

\subsection{Consequences of spectral independence}
\label{sec:si-consequences}

The spectral independence approach has been quite powerful as it
led to rapid mixing results for the hard-core model in the tree uniqueness region~\cite{ALO20},
for any 2-spin antiferromagnetic spin system in the tree uniqueness region~\cite{CLV1},
and for colorings~\cite{CGSV21,FGYZ21} it matched the best known parameter bounds using other algorithmic approaches.
Moreover, recent work of Chen et al.~\cite{CLV20} shows that spectral independence implies optimal mixing of the Glauber dynamics
in all of these cases as stated in the following theorem.

\begin{theorem}[\cite{CLV20}]
\label{thm:main}
For an arbitrary spin system on a graph of maximum degree $\Delta$, if the system is $\eta$-spectrally independent and $b$-marginally bounded,
then there exists a constant $C=C(b,\eta,\Delta)>0$ such that the mixing time of the Glauber dynamics
%on an $n$-vertex graph of maximum degree $\Delta$
for the spin system
is at most $Cn\log{n}$ where $n$ is the number of vertices,
and the modified log-Sobolev constant of the Glauber dynamics is at least $1/(Cn)$.
Moreover, the constant $C$ satisfies $C=\left(\frac{\Delta}{b}\right)^{O(1+\frac{\eta}{b})}$.
\end{theorem}
The key step in the proof of Theorem \ref{thm:main} is the implication
\begin{equation}\label{eq:imply1}
{\rm Spectral\; Independence}
\;\;\Longrightarrow\;\; {\rm Approximate \; Tensorization \;of\;  Entropy.}
\end{equation}
Approximate tensorization of entropy says that there exists a constant $C\geq 1$, such that
for any function $f:\Omega\rightarrow\R_{+}$,
\begin{equation}\label{eq:approx-tensor}
\Ent(f) \leq C\sum_{\vx\in V} \mu[\Ent_\vx(f) ],
\end{equation}
where
$\mu[f] = \sum_{\sigma \in \Omega} \mu(\sigma) f(\sigma)$ and
$\Ent(f)=\mu[f\log(f/\mu [f])]$ denote the mean and entropy of $f$ with respect to the measure $\mu$. In particular, $\Ent(f)$ is the relative entropy of the probability measure $f\mu/\mu[f]$ with respect to $\mu$, while $\mu[\Ent_\vx f ] = \mu[f\log(f/\mu_\vx [f])]$ is the expected value according to $\mu$ of the conditional entropy $\t\mapsto \Ent(f | \t)$ for $\tau$ a spin configuration on $V\setminus\{\vx\}$. To make some intuitive sense of approximate tensorization, notice that
if $\mu$ is a product distribution over $V$ then~\eqref{eq:approx-tensor} holds with $C=1$. In general,
approximate tensorization is easily seen to imply the desired bounds on the modified log-Sobolev  constant and
the mixing time of the Glauber dynamics; see e.g.\ \cite{CMT14}. In the setting of spin systems on the lattice $\bbZ^d$, approximate tensorization estimates are known to hold under the so-called strong spatial mixing condition; this follows from the logarithmic Sobolev inequalities established in \cite{SZ92a,MO94,C01}.

We present an alternative proof of some of the key steps for the implication \eqref{eq:imply1} in Section~\ref{sec:spectral-ind}; see Theorem \ref{th:reform}.
The analogous result in~\cite{CLV20} is proved in the more general framework of simplicial complexes and generalizes the result of \cite{CGM19} for homogeneous strongly log-concave distributions; see also \cite{HS19} for related results.
Our proof is completely framed in the setting of spin systems and is devoid of any work on simplicial complexes.
This new approach may be conceptually simpler to some readers, and it enables us to present a self-contained proof of our main results.  As a byproduct we also obtain an incremental
improvement in the resulting mixing time bound improving the exponent in the constant $C$ from
$O(1+ \eta/b^2)$ (see Theorem 1.9 in~\cite{CLV20}) to $O(1+ \eta/b)$ as stated in Theorem~\ref{thm:main}.

One of our  main results in this paper is the following substantial extension of %the implication
\eqref{eq:imply1}:
\begin{equation}\label{eq:imply2}
{\rm Spectral\; Independence}
\;\;\Longrightarrow\;\; {\rm General \; Block \;Factorization  \;of\;  Entropy.}
\end{equation}  Caputo and Parisi~\cite{CP20} introduced the notion of {\em general block factorization} of entropy which generalizes
approximate tensorization, and is useful for analyzing more general classes of Markov chains.
Let $\alpha=(\alpha_B)_{B\subset V}$ be an arbitrary probability distribution over subsets of $V$, and set $\delta(\alpha) = \min_{\vx\in V} \sum_{B:B\ni \vx} \alpha_B$ as in \eqref{eqn:delta}.
%\antonio{Shouldn't we make $\alpha$ a probability distribution?}
General block factorization
of entropy holds with constant $C$ if for all weights $\alpha$, for all $f:\Omega\rightarrow\R_{+}$:
\begin{equation}
\label{eqn:block-factorization}
  \delta(\alpha)\Ent f \leq C\sum_{B\subset V}\alpha_B\,\mu[\Ent_B{f}],
\end{equation}
where $\mu[\Ent_B f ] = \mu[f\log(f/\mu_B f)]$ is the expected value  of the conditional entropy $\t\mapsto \Ent(f | \t)$ for $\tau$ a spin configuration on $V\setminus B$. 
%We may assume without loss of generality that $\a$ is a probability over subsets of $V$. 
Entropy tensorization~\eqref{eq:approx-tensor} is the special case when
$\alpha_B = 1/n$ for every block of size $1$ and $\alpha_B = 0$ for larger blocks. The choice of the constant $\delta(\alpha)$ in this inequality is motivated by the fact that when $\mu$ is a product measure then \eqref{eqn:block-factorization} holds with $C=1$, in which case it is known as the Shearer inequality; see \cite{CMT14}. The block factorization of entropy is a  statement concerning the equilibrium distribution $\mu$ which has deep consequences for several natural sampling algorithms.
In particular, it implies optimal mixing and optimal entropy decay for arbitrary block dynamics and constitutes a key concept in the proof of  Theorem \ref{thm:colorings} and Theorem \ref{thm:Ising-Potts-mixing}. The precise formulation of \eqref{eq:imply2} and its corollaries is as follows.

\begin{theorem}\label{thm:block-factorization}
For an arbitrary
 spin system on a graph of maximum degree $\Delta$, if the system is $\eta$-spectrally independent and $b$-marginally bounded,
then general block factorization of entropy  \eqref{eqn:block-factorization} holds with constant $C=C(b,\eta,\Delta)$.
Moreover, all heat-bath block dynamics have optimal mixing and optimal entropy decay. The constant $C$ satisfies $C=\left(\frac{2}{b}\right)^{O\(\Delta(1+\frac{\eta}{b})\)}$.
\end{theorem}

Recall, for the Glauber dynamics $\delta(\alpha)=1/n$ and hence we recover Theorem~\ref{thm:main} as a special case of
 the above result.  As another example, for a bipartite graph Theorem~\ref{thm:block-factorization}
 implies $O(\log{n})$ mixing time of the even-odd dynamics. %\antonio{The even-odd dynamics has not been defined}

When the spin system satisfies \eqref{eqn:block-factorization} with $\alpha$ the uniform distribution over all subsets of a given size $\ell$ we refer to this as $\ell$-\emph{uniform block factorization of entropy} or $\ell$-UBF for short.   In~\cite{CLV20}, an important step in the proof of Theorem~\ref{thm:main} is establishing $\ell$-UBF with $\ell\sim \theta n$ for some $\theta\in(0,1)$.  To prove Theorem~\ref{thm:block-factorization} for arbitrary blocks we establish that $\ell$-UBF implies general block factorization of entropy, see Theorem~\ref{thm:UBF-GBF} for a detailed statement and Figure~\ref{fig:map} for a high-level overview.
%\eric{I modified Pietro's paragraph a bit about UBF.}

Recent work of Blanca et al.~\cite{BCPSV20} utilizes block factorization of entropy into the even and odd sublattices of $\mathbb Z^d$
to
obtain tight mixing time bounds for the Swendsen-Wang dynamics on boxes of $\integers^d$ in the high-temperature region.
Following the approach presented in~\cite{BCPSV20} and using our general result in Theorem \ref{thm:block-factorization}, here we prove optimal mixing time of the Swendsen-Wang dynamics
when spectral independence holds on arbitrary bounded-degree graphs. This can be formalized in the following statement, which is a key ingredient in the proof of Theorem \ref{thm:Ising-Potts-mixing}.
\begin{theorem}
	\label{thm:sw:general}
For the ferromagnetic Ising and Potts models on a graph of maximum degree $\Delta$, if the system is $\eta$-spectrally independent and $b$-marginally bounded,
then there exists a constant $C=C(b,\eta,\Delta)$ such that
the mixing time of the Swendsen-Wang dynamics is at most $ C\log{n}$ and the modified log-Sobolev constant
is at least $C^{-1}$. The constant $C$ satisfies $C=\left(\frac{2}{b}\right)^{O\(\Delta(1+\frac{\eta}{b})\)}$.
\end{theorem}
We turn to a further interesting consequence of spectral independence:
\begin{equation}\label{eq:imply3}
{\rm Spectral\; Independence}
\;\;\Longrightarrow\;\; {\rm Approximate \; Subadditivity  \;of\;  Entropy.}
\end{equation}
We say that the approximate subadditivity of entropy holds with constant $C$ if
\begin{equation}
\label{eqn:sub-add}
 \sum_{x\in V} \Ent (f_x)
 \leq C\Ent (f),
\end{equation}
where, for any nonnegative function $f$, the functions $f_x$ are defined by $f_x(a) =\mu(f\tc\si_x=a)$. Notice that when $\mu(f)=1$ then $\nu=f\mu$ is a probability measure and, if $\mu_x$ denotes the marginal of $\mu$ on $x$, then $f_x\mu_x$ gives the marginal of $\nu$ on $x$.
The inequality \eqref{eqn:sub-add} is known to be equivalent to a Brascamp-Lieb type inequality  for the measure $\mu$ \cite{CLL,CCE}. In particular, it implies that for any collection of functions $\varphi_\vx:[q]\mapsto \bbR$, $\vx\in V$, one has
\begin{equation}
\label{eqn:sub-add2}
 \mu\left(\prod_{\vx\in V}\varphi_\vx(\si_\vx)\right) \leq
 \prod_{\vx\in V}\mu\left( |\varphi_\vx(\si_\vx)|^C\right)^{1/{C}},
\end{equation}
where $C$ is the same constant as above.
For a general discussion of subadditivity of entropy, Brascamp-Lieb type inequalities, and their applications, see for instance \cite{BCELM}
and the references therein. In Theorem \ref{th:reform} %Section~\ref{sec:spectral-ind}
below we shall see that for
an arbitrary spin system on a graph of maximum degree $\Delta$, if the system is $\eta$-spectrally independent and $b$-marginally bounded,
then \eqref{eqn:sub-add} holds with $C=O(1+ \eta/b)$. The question of the validity of such inequalities in the context of high temperature spin systems was raised in \cite{CMT14} but as far as we know
there are no prior results in this direction.

\subsection{Establishing spectral independence}

The above results show the power of spectral independence as it implies optimal mixing time bounds for a wide
variety of Markov chains.   We next address when spectral independence holds and how it relates to classical
conditions that imply fast mixing.
The next series of results prove in a general context that when there exists a contractive coupling then spectral independence holds.

Let $d$ denote an arbitrary metric on $\Omega$.
A simple example is
the \emph{Hamming metric}, which for configurations $\sigma,\tau \in \Omega$ is defined to be $\hamming{\sigma}{\tau} = | \{x \in V: \sigma_x \neq \tau_x\} |$.
There are two types of more general metrics that we will consider: those within a constant factor of the Hamming metric and
vertex-weighted Hamming metric for arbitrary weights.
For $\gamma \ge 1$, a metric $d$ on $\Omega$ is said to be $\gamma$-equivalent to the Hamming metric (or $\gamma$-equivalent for simplicity)
%$\gamma$-(Hamming-)equivalent\zongchen{``$\gamma$-Hamming-equivalent'' $\leftarrow$ ``$\gamma$-equivalent to the Hamming metric''}
if for all $\sigma,\tau \in \Omega$,
\[
\frac{1}{\gamma} \hamming{\sigma}{\tau} \le d(\sigma, \tau) \le \gamma \hamming{\sigma}{\tau};
\]
that is, a $\gamma$-equivalent metric is an arbitrary metric where every distance is within a factor $\gamma$ of the Hamming distance.
In contrast, we can generalize the Hamming distance by considering arbitrary weights for the vertices.
Let $w: V \to \R_+$ be an arbitrary positive weight function.
The \emph{$w$-weighted Hamming metric} between two configurations $\sigma,\tau \in \Omega$ is defined to be
\[
d_w(\sigma, \tau) = \sum_{x \in V} w(\vx) \mathbf{1}\{\si_\vx \neq \tau_\vx\}.
\]
In particular, if $w_\vx = 1$ for all $\vx$ then $d_w$ is just the usual Hamming metric.  Note there are no constraints on the weights
except that they are positive; in particular, the weights can be a function of $n$.

We will often consider a class $\mathcal{P} = \{P^\tau: \tau \in \Pinning\}$ of Markov chains associated with $\mu$, where each $P^\tau$ is a Markov chain with stationary distribution $\mu^\tau$ and $\tau \in \Pinning$ is a pinning; for example, $\mathcal{P}$ can be the family of Glauber dynamics for all $\mu^\tau$'s.
In coupling proofs, the goal is to design a coupling so that for an arbitrary pair of states the chains contract with respect to
some distance metric after the coupled transition.
Roughly speaking, for $\ctrct\in(0,1)$,
we say that $\mu$ is $\ctrct$-contractive with respect to (w.r.t.) a collection $\mathcal{P}$ of Markov chains and a metric $d$ if one step of every chain $P^\tau$ contracts the distance by a factor $\ctrct$ in expectation.
This is formalized in the following definition.

\begin{definition}[$\ctrct$-Contraction]
%Let $P$ denote an ergodic Markov chain on $\Omega$ with stationary distribution $\mu$.
%For pinning $\t$, let $P^\t$ denote the dynamics with the fixed configuration $\t$.
Let $\mathcal{P}$ denote a collection of Markov chains associated with $\mu$ and let $d$ be a metric on $\Omega$.
For $\ctrct\in(0,1)$ we say that $\mu$ is $\ctrct$-contractive w.r.t.\ $\mathcal{P}$ and $d$ %for the dynamics $P$
if for all $\t \in \Pinning$, all $X_0,Y_0 \in \Omega^\tau$, there exists a coupling $(X_0,Y_0)\rightarrow (X_1,Y_1)$ for $P^\tau$ such that:
$$ \bbE[d(X_1,Y_1)|X_0,Y_0] \leq \ctrct d(X_0,Y_0).$$
\end{definition}

The following result shows that spectral independence holds if the Glauber dynamics has a contractive coupling.

\begin{restatable}{theorem}{glauber}
\label{thm:glauber}
\
	\begin{enumerate}
		\setlength{\itemsep}{.7pt}
		\item[(1)] If $\mu$ is $\ctrct$-contractive w.r.t.\ the Glauber dynamics and an arbitrary $w$-weighted Hamming metric,
		%For an arbitrary $w$-weighted Hamming metric if $\mu$ is $\a$-contractive for the Glauber dynamics for some $\a\in(0,1)$,
		then $\mu$ is spectrally independent with constant $\eta =\frac{2}{(1-\ctrct)n}$. In particular, if $\ctrct\leq 1-\frac{\eps}{n}$, then $\eta\leq \frac{2}{\eps}$.
		\item[(2)] If the metric in (1) is not a weighted Hamming metric but instead an arbitrary $\gamma$-equivalent metric, then $\eta =\frac{2\gamma^2}{(1-\ctrct)n}$.
		In particular, if $\ctrct\leq 1-\frac{\eps}{n}$, then $\eta\leq \frac{2\gamma^2}{\eps}$.
	\end{enumerate}
\end{restatable}

Note a $\ctrct$-contractive coupling for the Hamming distance immediately implies $O(n\log{n})$ mixing time of the Glauber dynamics (see, e.g.,~\cite{BubleyDyer, LevinPeresWilmer}). 
But the above theorem offers two additional features.  First, it allows arbitrary weights $w$ and the resulting 
bound on the mixing time is independent of these weights,  whereas a coupling argument, such as utilized
in path coupling~\cite{BubleyDyer}, yields a
mixing time bound which depends on the ratio of $\max_\vx w(\vx)/\min_\vx w(\vx)$.
Second, as discussed in the previous theorems, spectral independence (together with the
easily satisfied marginal boundedness) implies optimal bounds on the mixing time and entropy decay rate for %any associated modified log-Sobolev constants
%not only for the Glauber dynamics but also for 
arbitrary heat-bath block dynamics.

%We say a collection $\mathcal{P}$ of Markov chains associated with $\mu$ is \emph{$\loc$-local} if for any two adjacent pinning $\tau = (U,\tau_U)$ and $\tau' = (U\cup \vx, \tau_U \cup \tau_\vx)$ and every $\sigma \in \Omega^{\tau'}$, we have
%\[
%W_{1,d_{\mathrm{H}}} (P^\tau(\sigma,\cdot), P^{\tau'} (\sigma,\cdot)) \le \loc.
%\]

We can extend Theorem~\ref{thm:glauber} by replacing the Glauber dynamics with arbitrary Markov chains.
In particular, we consider a general class of Markov chains which we call the \emph{select-update dynamics}. 
In each step, the select-update dynamics picks a block $B \in \mathcal{B}$ randomly (with a distribution that may depend on the current configuration), and updates all vertices in $B$ using the current configuration (and the pinning if there is one).
Note that no assumptions are made on how to pick or update the blocks; the only requirement is that the dynamics converges to the correct stationary distribution.
If the chain selects a block $B$ from a fixed distribution over $\mathcal{B}$ and updates $B$ using the conditional marginal distribution on $B$ (under the pinning if applicable), then this is the standard heat-bath block dynamics that we introduced earlier; hence,  the select-update dynamics is much more general than the weighted heat-bath block dynamics.
Another example of the select-update dynamics is the flip dynamics for sampling random colorings of a graph; see Section~\ref{subsec:coloring}.%\antonio{Streamlined this paragraph}

We define $M = \max_{B \in \mathcal{B}} |B|$ to be the maximum block size
%and $D$ to be the maximum ``coverage probability'' of a vertex. %(in contrast to $\delta$ defined by \eqref{eqn:delta}).
and $D$ to be the maximum probability of a vertex being selected in any step of the chain.

\begin{restatable}{theorem}{general}
\label{thm:general}
If $\mu$ is $\ctrct$-contractive w.r.t.\ arbitrary select-update dynamics and an arbitrary $\gamma$-equivalent metric,
then $\mu$ is spectrally independent with constant $\eta =\frac{2\gamma^2DM}{1-\ctrct}$.
\end{restatable}

Theorem~\ref{thm:general} generalizes Theorem~\ref{thm:glauber}(2) since $M = 1$ and $D = 1/n$ for the Glauber dynamics.
If we further assume that the select-update dynamics updates each connected component of a block independently, then the maximum block size $M$ can be replaced by the maximum component size of a block; see Remark~\ref{rmk:block-general}.
See also Theorem~\ref{thm:general-general} for a stronger statement involving arbitrary Markov chains, where $DM$ is replaced by
the maximum expected distance of two chains when pinning a single vertex.
This more general statement potentially applies to chains with unbounded block sizes, including the Swendsen-Wang dynamics.
%\eric{Zongchen, please check, sentence added here.}
%\zongchen{This is correct. But SW is an extremal example of select-update dynamics; maybe Even/odd dynamics?}

It is worth remarking that, as a corollary of Theorem~\ref{thm:general}
we obtain that a coupling argument for the select-update dynamics where the maximum block size is constant (and $D/(1-\ctrct)=O(1)$)
 implies $O(n\log{n})$ mixing time of the Glauber dynamics, %as well as the associated bounds on the (modified)
%log-Sobolev constants 
together with the optimal mixing and optimal entropy decay for arbitrary heat-bath block dynamics.

Moreover, as a corollary of Theorem~\ref{thm:glauber} we obtain that the Dobrushin uniqueness condition implies spectral independence.
The Dobrushin uniqueness condition is a classical condition in statistical physics which considers the following dependency matrix.
\begin{definition}[Dobrushin uniqueness condition]
The \emph{Dobrushin dependency/influence matrix} $R \in \R^{V \times V}$ is defined by $R(x,x) = 0$ and
\[
R(x,y) =
\max \left\{ \tv{\mu_y(\cdot \mid \sigma)}{\mu_y(\cdot \mid \tau)}: (\si,\tau) \in \cS_{x,y} \right\} \quad \text{~for~} x \neq y
\]
where $\cS_{x,y}$ is the set of all pairs of configurations on $V \setminus \{y\}$ that can differ only at $x$.
The \emph{Dobrushin uniqueness condition} holds if the maximum column sum of $R$ is at most $1-\eps$ for some $\eps>0$.
\end{definition}
%\eric{Yes, you seem to be correct that it's column based on: Dobrushin conditions and Systematic Scan by Dyer, Goldberg, and Jerrum.  They say (page 3):\\
%As Weitz points out [21], Dobrushin and Shlosman stated their result in terms of the total influence on a site but they worked in a translation-invariant setting and what they used is that the total influence of a site is small. In fact, Dobrushin and Shlosman worked in a more general block-dynamics setting. This will be discussed below.}

The Dobrushin dependency matrix for the entry $R(x,y)$ considers the worst case pair of configurations on the entire
neighborhood of $y$ which differ at $x$.  If $x$ is not a neighbor of $y$ then $R(x,y)=0$.  Hence, the Dobrushin uniqueness
condition states that for all $y$, $\sum_{x\in N(y)} R(x,y) < 1$.  In contrast, the ALO influence matrix considers the influence of
a disagreement at $x$ on a vertex $y$ (which is not necessarily a neighbor) and no other vertices are fixed, although one needs to
consider all pinnings to establish spectral independence, so the notions are incomparable at first glance.

Using Theorem~\ref{thm:glauber} we prove that
the Dobrushin uniqueness condition implies spectral independence.
%Moreover, we can apply the more general Theorem~\ref{thm:general} to two generalizations of the Dobrushin uniqueness condition.
%Recall, the Dobrushin uniqueness condition states that the maximum column sum of the Dobrushin dependency matrix $R$ is $<1-\eps$ for some $\eps>0$.
Moreover, our result holds under generalizations of the Dobrushin uniqueness condition.
%The Dobrushin uniqueness condition was generalized by Hayes~\cite{Hayes} and further by Dyer et al.~\cite{DGJ09} to the following spectral condition:
Hayes~\cite{Hayes} generalized it to the following spectral condition:
if $\|R\|_2 \le 1-\eps$
for some $\eps>0$, then the mixing time of the Glauber dynamics is $O(n\log{n})$.
This was further generalized by Dyer et al.~\cite{DGJ09} to arbitrary matrix norms.
%We prove that this spectral condition of~\cite{Hayes,DGJ09} also
%implies spectral independence.
We prove spectral independence when the spectral radius $\rho(R) < 1$, which is the strongest statement of this type as the spectral radius is no larger than any matrix norm;
see Remark~\ref{rmk:dob} for a more detailed discussion.
%\zongchen{Modified.}

\begin{restatable}{theorem}{dep}
\label{thm:dep-inf}
If the Dobrushin dependency matrix $R$ satisfies $\rho(R) \le 1-\eps$ for some $\eps > 0$,
then $\mu$ is spectrally independent with constant $\eta =2/\eps$.
\end{restatable}

Previously, Marton~\cite{Marton15} (see also~\cite{GSS,SS}) showed that the spectral condition in Theorem~\ref{thm:dep-inf}
implies approximate tensorization of entropy and thus optimal bounds on the modified log-Sobolev constant for the Glauber dynamics. See also \cite{BB19} for related results with an alternative technique. However, the approach in these works does not imply block factorization of entropy as in our case.

\begin{remark}\label{Remark:PTconj}
Our definition of $\ctrct$-contraction is equivalent to the statement that the Markov chain has coarse Ollivier-Ricci curvature at least $1-\ctrct>0$ with respect to the metric $d$ \cite{Ollivier}. Combining Theorem \ref{thm:glauber} with Theorem \ref{thm:block-factorization} we obtain a proof of the following version of the Peres-Tetali conjecture:
 if the Glauber dynamics  has  Ollivier-Ricci curvature at least $\eps/n>0$ then the Glauber dynamics has a modified log-Sobolev constant at least $c/n$ and any $\a$-weighted heat-bath block dynamics has a modified log-Sobolev constant at least $c\,\delta(\a)$,  for some constant $c=c(\eps,b,\Delta)>0$, where $\delta(\a)$ is defined in \eqref{eqn:delta}.
Replacing Theorem \ref{thm:glauber} with its generalization Theorem \ref{thm:general} we obtain the same conclusion under the much milder assumption that there exists some $\ctrct$-contractive  select-update dynamics satisfying $DM/(1-\ctrct)=O(1)$.
The original  Peres-Tetali conjecture in the setting of random walks on graphs is that if there exists a graph metric $d$ such that  the random walk has Ollivier-Ricci curvature at least $\lambda>0$ with respect to $d$ then the random walk has modified log-Sobolev constant at least $c\lambda>0$, for some universal constant $c>0$; see Conjecture 3.1 in Eldan et al.~\cite{ELL}.
% The original conjecture is formulated for general Markov chains on a graph and with a linear dependency on $\kappa$ in the constant $c(\kappa)$.
%On the other hand, our results concern Glauber dynamics or any other weighted heat-bath block dynamics for spin systems, and our constant $c(\kappa)$ can be exponentially small in $1/\kappa$.
\end{remark}

The organization of the paper is demonstrated in Figure~\ref{fig:map}.
After giving preliminaries in Section~\ref{sec:prelim}, we deduce general block factorization of entropy from uniform block factorization in Section~\ref{sec:ubf:to:gbf}; the latter is a key intermediate step in the proof approach of \cite{CLV20}.
Later in Section~\ref{sec:spectral-ind}, we reformulate the result of \cite{CLV20} showing uniform block factorization given spectral independence; our new proof avoids abstract simplicial complexes and gives a slightly better constant.
In Section~\ref{sec:contraction-SI}, we establish spectral independence if the distribution admits a contractive Markov chain.
Finally, in Section~\ref{sec:SW} we show optimal mixing and optimal entropy decay of the Swendsen-Wang dynamics if $k$-partite factorization holds, which can be further deduced from spectral independence.

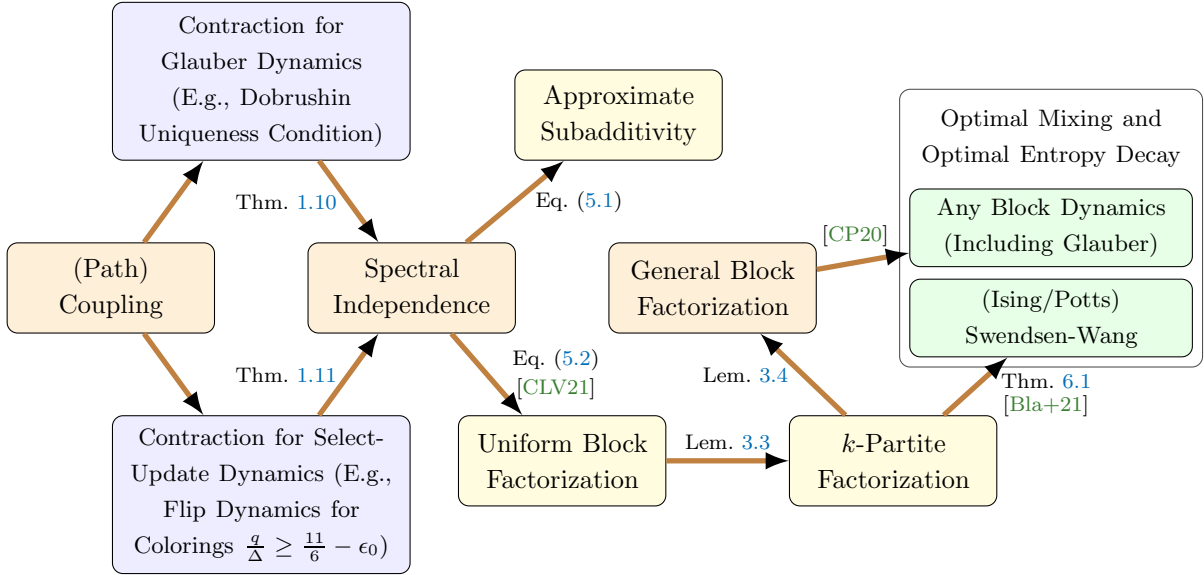
\begin{figure}\label{fig:map}
\centering
\begin{tikzpicture}
\tikzstyle{SI}   = [rounded corners, minimum height = 1.2cm, text width = 2.5cm, text centered, draw = black, fill = YellowOrange!15]
\tikzstyle{UBF}  = [rounded corners, minimum height = 1.2cm, text width = 2.5cm, text centered, draw = black, fill = yellow!15]

\tikzstyle{Ctct} = [rounded corners, minimum height = 2.1cm, text width = 3.7cm, text centered, draw = black, fill = blue!7]

\tikzstyle{SW}   = [rounded corners, minimum height = 1.0cm, text width = 3.5cm, text centered, draw = black, fill = green!10]
\tikzstyle{Opt}  = [minimum height = 1.0cm, text width = 3.5cm, text centered]

\tikzstyle{arrow} = [->, > = latex, line width = 2pt, draw = brown]

\node (SI) at (0,0) [SI] {\small Spectral Independence};
\node (GBF) at (4.0,0) [SI] {\small General Block Factorization};
\node (PC) at (-4.0, 0) [SI] {\small (Path) Coupling};

\node (UBF) at (2.0,-2.3) [UBF] {\small Uniform Block Factorization};
\node (kPF) at (6.4,-2.3) [UBF] {\small $k$-Partite Factorization};
\node (AS) at (2.75,2.3) [UBF] {\small Approximate Subadditivity};

\node (ct-gl) at (-2.0,2.75) [Ctct] {\footnotesize Contraction for Glauber Dynamics (E.g., Dobrushin Uniqueness Condition)};
\node (ct-su) at (-2.0,-2.75) [Ctct] {\footnotesize Contraction for Select-Update Dynamics (E.g., Flip Dynamics for Colorings $\frac{q}{\Delta} \ge \tfrac{11}{6} - \eps_0$)};

\node (Block) at (8.5,0.8) [SW] {\footnotesize Any Block Dynamics (Including Glauber)};
\node (SW) at (8.5,-0.4) [SW] {\footnotesize (Ising/Potts) Swendsen-Wang};
\node (Opt) at (8.5, 2.0) [Opt] {\footnotesize Optimal Mixing and Optimal Entropy Decay};
\node (bigbox) [rounded corners, draw = black!70, fit = (Opt)(Block)(SW)] {};

\draw [arrow] (PC) -- (ct-gl);
\draw [arrow] (PC) -- (ct-su);

\draw [arrow] (ct-gl) -- node[anchor=east] {\scriptsize Thm.\ \ref{thm:glauber}} (SI);
\draw [arrow] (ct-su) -- node[anchor=east] {\scriptsize Thm.\ \ref{thm:general}} (SI);

\draw [arrow] (SI) -- node[anchor=west, xshift = 3pt] {\scriptsize Eq.\ \eqref{eq:ubfsub11}} (AS);

\draw [arrow] (SI) -- node[anchor=west, xshift = 5pt] {\scriptsize \shortstack{Eq.\ \eqref{eq:ubfsub121}\\\cite{CLV20}}} (UBF);
\draw [arrow] (UBF) -- node[anchor=south] {\scriptsize Lem.\ \ref{lem:ubf:boosting}} (kPF);
\draw [arrow] (kPF) -- node[anchor=east, xshift = -1pt] {\scriptsize Lem.\ \ref{lem:kpart:gen}} (GBF);

\draw [arrow] (GBF) -- node[anchor=south, xshift = -3pt, yshift = 1pt] {\scriptsize \cite{CP20}} (Block);
\draw [arrow] (kPF) -- node[anchor=west, xshift = 5pt, yshift = -3.2pt] {\scriptsize \shortstack{Thm.\ \ref{th:main-intro2}\\\cite{BCPSV20}}} (SW);
\end{tikzpicture}
\caption{Organization of the paper}
\end{figure}

\section{Preliminaries}\label{sec:prelim}

\subsection{Spin systems}

We begin with the formal definition of general $q$-state spin systems.
Let $q \ge 2$ be an integer and write $[q] = \{1,\dots,q\}$.
Let $G = (V \cup \partial V,E \cup \partial E)$ be an undirected graph where $\partial V$ denotes the boundary set of the induced subgraph $G' = (V,E)$, and $\partial E$ consists of all edges between $V$ and $\partial V$. 
A $q$-spin system on $G$ with a boundary condition $\xi \in [q]^{\partial V}$ is parameterized by nonnegative symmetric matrices $A_{\vx \vy} \in \R_{+}^{q\times q}$, $\{\vx,\vy\} \in E \cup \partial E$, representing the nearest neighbor interactions,
and vectors $B_\vx \in \R_{+}^q$, $\vx\in V$, representing the external fields.
A configuration $\sigma \in [q]^V$ has weight:
\[  w(\sigma) = \prod_{\{\vx,\vy\} \in E}  A_{\vx\vy}(\sigma_\vx, \sigma_\vy) \prod_{\substack{\{\vx,\vy\} \in \partial E \\ \vx\in V,\, \vy \in \partial V}}  A_{\vx\vy}(\sigma_\vx, \xi_\vy) \prod_{\vx \in V} B_\vx(\sigma_\vx).
\]
Let $\Omega = \{\sigma\in [q]^V: w(\sigma)>0\}$ denote the collection of all feasible configurations and let $Z_G = \sum_{\sigma\in\Omega} w(\sigma)$ denote the partition function.
We assume that $\Omega \neq \emptyset$; i.e., the boundary condition $\xi$ is feasible.
Finally,
the \emph{Gibbs distribution} $\mu$ is given by, for $\sigma\in\Omega$,
\[
\mu(\sigma) = w(\sigma)/Z_G.
\]

Recall the notion of pinning from Section~\ref{sec:intro-spectral-independence} which
we briefly repeat here for convenience.
For $U\subset V$, we use the notation $\sigma_U=(\sigma_\vx)_{\vx\in U}$ and let $\Omega_U = \{\tau\in [q]^U: \exists \sigma\in\Omega, \sigma_U=\tau\}$ be the set of all possible pinnings on $U$. 
Note, for $\vx\in V$, $\Omega_\vx$ is the set of feasible spin assignments for vertex $\vx$.
Denote the collection of all pinnings by  $\Pinning = \cup_{U \subset V} \Omega_U$ and 
denote the set of all feasible vertex-spin pairs by $\VSpairs = \{(\vx,\spa): \vx \in V, \spa \in \Omega_\vx\}$.
For $\tau \in \Pinning$, let $\mu^\t$ denote the conditional Gibbs distribution $\mu(\cdot\tc\sigma_U=\t)$. We also write $\mu_\Lambda^\t=\mu^\t$ if $\t\in \Omega_{V\setminus\Lambda}$ and use the notation $\mu_\Lambda:\;\Omega_{V\setminus\Lambda}\ni \t\mapsto \mu_\Lambda^\t$ for the associated mapping.

For a pinning $\tau\in\Omega_U$ for $U \subset V$,
let $\Omega^\tau = \{\sigma\in\Omega: \sigma_U=\tau\}$ denote the corresponding state space; i.e., $\Omega^\tau$ is the support of $\mu^\tau$.
We also define $\Omega^\tau_W = \{\phi \in [q]^W: \exists \sigma \in \Omega^\tau, \sigma_W = \phi\}$ for $W \subset V \setminus U$ and $\VSpairs^\tau = \{(\vx,\spa): \vx \in V\setminus U, \spa \in \Omega^\tau_\vx\}$.
We say $\Omega^\tau$ is connected if
the graph on $\Omega^\tau$ with edges connecting pairs at Hamming distance 1 is connected.
The distribution $\mu$ over $\Omega$ is said to be \emph{totally-connected}
if for every $\tau \in \Pinning$, the set $\Omega^\tau$ is connected.   Throughout this paper,
we will assume the distribution $\mu$ is totally-connected as this is necessary for the Glauber dynamics to be ergodic for all conditional measures $\mu^\tau$.

We recall some classical examples of spin system.  The Ising/Potts model at inverse temperature $\beta\in\bbR$ corresponds to the interaction $A_{\vx\vy}(a,a')=\exp{(\beta\IND(a=a'))}$ and $B_\vx(a)=\exp{(h(a)) }$ where $h\in\bbR^q$ is a vector of external fields, with $q=2$ for the Ising model and $q\geq 3$ for the Potts model.  The hard-core (or independent sets) model with parameter $\lambda>0$ is obtained with $q=2$, $A_{\vx\vy}(a,a')=0$ if $a=a'=1$ and $A_{\vx\vy}(a,a')=1$ otherwise, and $B_\vx(a)=\lambda$ if $a=1$ and $B_\vx(a)=1$ if $a=2$. The $q$-colorings model corresponds to $A_{\vx\vy}(a,a')=\IND(a\neq a')$ and $B_\vx(a)=1$.
Note that the Ising/Potts models with any $\beta$ and $h$, as well as the hard-core model with any $\lambda>0$, and the $q$-colorings when $q\geq\Delta+2$ are totally-connected spin systems.

\subsection{Mixing time, entropy,  and log-Sobolev inequalities}

Let $P$ be the transition matrix of an ergodic Markov chain with finite state space $\Omega$ and stationary distribution $\mu$.
Let $P^t(X_0,\cdot)$ denote the distribution of the chain after $t$ steps
starting from the initial state $X_0 \in \Omega$. The {\it mixing time} $\tmix(P)$ of the chain is defined as
\[
\tmix(P) = \max\limits_{X_0 \in \Omega}\min \left\{ t \ge 0 : {\|{P}^t(X_0,\cdot)-\mu\|}_{\textsc{tv}} \le 1/4 \right\},
\]
where $\|\cdot\|_{\textsc{\tiny TV}}$ denotes total variation distance.
%The \textit{mixing time} of $M$ is defined as $\tmix(M) = \tmix(M,1/4)$.

In this paper, we rely on functional inequalities related to entropy to bound the mixing time.
For a function $f:\Omega \mapsto \bbR$,
let $\mu[f] = \sum_{\s \in \Omega}  \mu(\s) f(\s)$
and $\var_\mu(f) = \mu[f^2] - \mu[f]^2$
denote its mean and variance with respect to $\mu$.
Likewise, for $f:\Omega\rightarrow\R_{+}$, the {\em entropy} of $f$ with respect to $\mu$ is defined as
\begin{align}\label{eq:def-ent}
\ent( f) = \mu\left[f \cdot \log \left(\frac{f}{\mu[f]}\right)\right] = \mu[f \cdot \log f] - \mu[f] \cdot \log \mu[f].
\end{align}

When $f\geq 0$ is such that  $\mu[f] =1$, then
$\ent (f) = H(f \mu\,|\, \mu)$ equals the {\em relative entropy}, or Kullback-Leibler divergence, of the distribution $f \mu$ with respect to  $\mu$.

For real functions $f,g$ on $\Omega$,
the {\em Dirichlet form} associated to the pair $(P,\mu)$ is defined as
\begin{align}\label{direlenti}
\cD_P(f,g)=\scalar{f}{(1-P)g}_\mu,
%= \cD_P(f,f)=\scalar{(1-Q)f}{f}_\mu=\frac12\sum_{\sigma,\tau \in \Omega}\mu(\sigma)Q(\sigma,\tau)(f(\sigma)-f(\tau))^2
\end{align}
where $\scalar{f}{g}_\mu=\mu[fg]$ denotes the scalar product in $L^2(\mu)$.
When $P$ is reversible, i.e., $\mu(\sigma)P(\sigma,\tau) = \mu(\tau)P(\tau,\sigma)$, one has
\begin{align}\label{direlenti21}
\cD_P(f,g)=\frac12\sum_{\sigma,\tau \in \Omega}\mu(\sigma)P(\sigma,\tau)(f(\sigma)-f(\tau))(g(\sigma)-g(\tau)).
\end{align}

\begin{definition}\label{def:LSI}
	The pair $(P,\mu)$ satisfies the (standard) {\em log-Sobolev inequality} (LSI) with constant~$s$ if for all $f\geq 0$:
	\begin{align}\label{direlenti3}
	\cD_P(\sqrt f,\sqrt f)\geq s\, \ent (f).
	\end{align}
	It satisfies the {\em modified log-Sobolev inequality} (MLSI) with constant $\rho$ if for all $f\geq 0$:
	\begin{align}\label{direlenti4}
	\cD_P(f,\log f)\geq \rho\, \ent (f).
	\end{align}
	It satisfies the
	 (discrete time) {\em relative entropy decay with rate $\d > 0$} if for all distributions $\nu$:
	\begin{align}
		H(\nu P\tc\mu)\leq (1-\d) H(\nu \tc\mu).\label{relent2}
	\end{align}
	\end{definition}
	In this paper we focus on the entropy decay inequality \eqref{relent2} which may be seen as the discrete time analog of the modified log-Sobolev inequality.
	We recall some well known facts about its relation to the other two inequalities and its implications for mixing times. 
	\begin{lemma}\label{lem:mix}
	If $(P,\mu)$ satisfies the standard LSI with constant $s$ then it satisfies the MLSI with constant $\rho=2s$. If it satisfies the discrete time relative  entropy decay with rate $\d>0$, then it satisfies the MLSI with constant $\rho=\d$. Finally, if it satisfies the discrete time relative entropy decay with rate $\d>0$, then %its mixing time $\tmix(P)$ satisfies
	\begin{align}\label{reltmix}
	\tmix(P) \leq 1+ \d^{-1}[\log(8)+\log\log(1/\mu_*)]\,,
	\end{align}
	where $\mu_* = \min_{\si\in\Omega} \mu(\si)$.
\end{lemma}
We refer to e.g. \cite[Section 2]{BCPSV20} for a proof. Note that we have not assumed reversibility of $P$ in the above lemma.
If $(P,\mu)$ is reversible, then one can additionally show that the standard LSI with constant $s$ implies  the discrete time relative relative entropy decay with rate $\d=s$.
%
%	If a Markov chain with transition matrix $P$ and stationary distribution $\mu$ satisfies relative entropy decay with rate $\d>0$, then its mixing time $\tmix(P)$ satisfies
%	\begin{align}\label{reltmix}
%		\tmix(P) \leq 1+ \d^{-1}[\log(8)+\log\log(1/\mu_*)]\,,
%	\end{align}
%	where $\mu_* = \min_\si \mu(\si)$.
%
%In addition, the bound \eqref{relent2} is stronger than the MLSI in \eqref{direlenti4}.
%That is, if the entropy decay holds with rate $\d$ in discrete time then it holds with the same rate in continuous time and implies the MLSI with constant $\d$.
%Finally, we note that LSI implies the discrete time entropy decay.
%\begin{lemma}\label{lem:miclo}
%	If the pair $(P^*P,\mu)$ satisfies the standard LSI with constant $\a$, then  the discrete time entropy decay holds for $(P,\mu)$ with constant $\d=\a$.
%	In particular, if $P$ is reversible and $(P,\mu)$ satisfies the  LSI with constant $\a$, then for all $f\geq 0$:
%	\begin{align}\label{arelentii3}
%	\ent P f\leq (1-\a) \ent f.
%	\end{align}
%\end{lemma}
%

\subsection{Some basic properties of entropy}
To compute the relative entropy with respect to a pinned measure $\mu_\Lambda^\t$ it is convenient to use the notation
 \begin{align}\label{eq:def-ent-la}
\ent_\L( f) = \mu_\L\left[f  \log \left(f/\mu_\L[f]\right)\right],
\end{align}
with the understanding that if we evaluate the left hand side at a given pinning $\t$ on $\L^c=V\setminus \L$ we then evaluate the expectations in the right hand side  with respect to $\mu_\L^\t$. To emphasize the dependence on the pinning we sometimes write $\ent^\t_\L( f) $. The expectation $\mu[\Ent_\L f]$ is obtained by averaging with respect to $\mu$ over the pinning $\t$ on $\L^c$, and satisfies
\begin{align}\label{eq:def-ent-la1}
\mu[\ent_\L( f)] = \sum_{\t\in\O_{\L^c}}\mu(\si_{\L^c}=\t)\,\ent^\t_\L( f)=\mu\left[f  \log \left(f/\mu_\L[f]\right)\right].
\end{align}
The following lemma summarizes a standard decomposition of the relative entropy; see e.g.\  \cite[Lemma 3.1]{CP20} for a proof.
\begin{lemma}\label{lem:telescope} \
	For any $\L\subset V$, for any $f: \Omega \to \R_+$:
	\begin{align}\label{deco}
	\ent (f) = \mu\left[\ent_{\L} (f)\right]
	+ \ent\,(\mu_\L[f]).
	\end{align}
	More generally, for any %finite monotone set sequence
	$\L_0\subset \L_1\subset\cdots\subset \L_w\subset V$, for any $f: \Omega \to \R_+$:
	\begin{align}\label{tele}
	\sum_{i=1}^{w}\mu\left[\ent_{\L_{i}} (\mu_{\L_{i-1}}[f])\right]
	=\mu\left[\ent_{\L_w}(\mu_{\L_0}[f])\right].
	\end{align}
\end{lemma}

The following monotonicity property of the entropy functional %are especially useful and are
is an immediate consequence of the previous lemma.

\begin{lemma}\label{lem:mono}
	For all $A\subset B \subset V$,
	\begin{gather}
	\label{eq:prelim:b1}
	\mu[\ent_A(f)]\leq\mu[\ent_B(f)]\,.
%	\\
%	\label{eq:prelim:b2}
%	\ent(\mu_B f)\leq \ent(\mu_A f).
	\end{gather}
\end{lemma}	

Next, we recall the definition of general block factorization of entropy.
\begin{definition}\label{def:GBF}
The spin system is said to satisfy the {\em general
block factorization of entropy} with constant $C$ if for all $f\geq 0$, for all probability distribution $\a$ over subsets of $V$,
\begin{equation}
\label{eqn:block-factorization-def}
  \delta(\alpha)\Ent f \leq C\sum_{B\subset V}\alpha_B\,\mu[\Ent_B{f}],
\end{equation}
where $\d(\a)=\min_{x\in V} \sum_{B:\,B\ni x}\a_B$.
\end{definition}

We will often consider independent sets $\L$ of $V$, that is sets of vertices whose induced subgraph in $G$ has no edge; in those cases, $\mu_\L$ is a product measure $\mu_\L=\otimes_{x\in\L}\mu_x$ and the following lemma will be useful.
\begin{lemma}\label{lem:shearer}
	Fix $\L\subset V$ and %$\tau \in \Omega(V \setminus \Lambda),$ and
	suppose that $\mu_\L$ is a product measure on $\mu_\L=\otimes_{x\in\L}\mu_x$.
	Then, for any distribution $\a$ over the subsets of $\L$, and any $f: \Omega \to \R_+$:
	\begin{equation}\label{shearer}
	\d( \a)\,\ent_\L(f)\leq \sum_{B\subset \L}\a_B\, \mu_\L\!\left[\ent_B (f)\right],
	\end{equation}
	that is $\mu_\L$ satisfies the general block factorization of entropy with constant $C=1$.
\end{lemma}

The above statement is a consequence of the weighted Shearer inequality for the Shannon entropy; see Lemma 4.2 in~\cite{CP20}.
The following properties will also be used. 
\begin{lemma}\label{lem:prod:fact} \
Let $\L=A\cup B$	%Let $S_1, S_2, \dots , S_w\subset V$
and assume that $\mu_\L$ is a product $\mu_\L= \mu_A\otimes \mu_B$.
%\otimes_{i=1}^w \mu_{S_i}$. %be such that $\dist(S_i,S_j) \ge 2$ for every $i \neq j$.
	Then, for all %any $i > 1$ and  any function
	$f\geq 0$:
%	\begin{align}\label{prod}
%	\mu[\ent_{S_1 \cup \dots \cup S_i} (\mu_{_{S_1 \cup \dots \cup S_{i-1}}}(f))] = \mu[\ent_{S_i} (\mu_{_{S_1 \cup \dots \cup S_{i-1}}}(f))].\end{align}
	\begin{align}\label{prod}
	\ent_{\L} (\mu_B(f)) = \mu_\L[\ent_{A} (\mu_B(f))],
	\end{align}
	and for all $U\subset B$,
	\begin{align}\label{prodest}
	\mu_\L[\ent_{A} (\mu_B(f))]\leq \mu_\L[\ent_{A} (\mu_U(f))].
	\end{align}

\end{lemma}

\begin{proof}
From the decomposition in Lemma \ref{lem:telescope} it follows that
$$
\ent_{\L} (\mu_B(f))- \mu_\L[\ent_{A} (\mu_B(f))] = \ent_{\L} (\mu_A\mu_B(f)) = \ent_{\L} (\mu_\L(f)) =0.
$$
This proves \eqref{prod}. To prove \eqref{prodest} notice that	 by definition
\begin{align*}
	\mu_\L\left[\ent_A (\mu_B(f))\right]
	%&= \mu_\L\left[\mu_A \left[\mu_B(f)\log\left(\frac{\mu_B(f)}{\mu_A\mu_B(f))}\right)\right]\right]\\
	&= \mu_\L\left[\mu_B(f)\log\left(\frac{\mu_B(f)}{\mu_A\mu_B(f))}\right)\right].
%	 \tag{Def. of entropy}\\
%	& = \mu\left[\mu_{S_{<i}}(f)\log\left(\frac{\mu_{S_{<i}}(f)}{\mu_{V_j \cap S_i}(\mu_{S_{<i}}(f))}\right)\right] \tag{Law of total Exp.}\\
%	& = \mu\left[\mu_{S_{<i}}(\mu_{V_j\cap S_{<i}}(f))\log\left(\frac{\mu_{S_{<i}}(\mu_{V_j\cap S_{<i}}(f))}{\mu_{V_j \cap S_i}(\mu_{S_{<i}}(f))}\right)\right] \tag{Law of total Exp.}\\
%	& = \mu\left[\mu_{S_{<i}}(\mu_{V_j\cap S_{<i}}(f))\log\left(\frac{\mu_{S_{<i}}(\mu_{V_j\cap S_{<i}}(f))}{\mu_{S_{<i}}(\mu_{V_j \cap S_i}(\mu_{V_j\cap S_{<i}}(f)))}\right)\right] \tag{by~\eqref{eq27}}\\
%	& = \mu\left[\mu_{V_j\cap S_{<i}}(f)\log\left(\frac{\mu_{S_{<i}}(\mu_{V_j\cap S_{<i}}(f))}{\mu_{S_{<i}}(\mu_{V_j \cap S_i}(\mu_{V_j\cap S_{<i}}(f)))} \right)\right] \\
%	& = \mu\left[\mu_{V_j\cap S_{<i}}(f)\log\left(\frac{\mu_{S_{<i}}(\mu_{V_j\cap S_{<i}}(f))}{\mu_{V_j \cap S_i}(\mu_{S_{<i}}(\mu_{V_j\cap S_{<i}}(f)))} \right)\right]\tag{by~\eqref{eq27}}\\
%	& \leq \mu\left[\mu_{V_j \cap S_i}\left[\mu_{V_j\cap S_{<i}}(f)\log\left(\frac{\mu_{V_j\cap S_{<i}}(f)}{\mu_{V_j \cap S_i}(\mu_{V_j\cap S_{<i}}(f))}\right)\right]\right] \tag{Proposition \ref{lem:variationalprinciple}}\\
%	&=\mu\left[\ent_{V_j \cap S_i}(\mu_{V_j\cap S_{<i}}(f))\right] \tag{Def. of entropy}.
	\end{align*}
For any $U\subset B$, $\mu_B(f)=\mu_B\mu_U(f)$ and the product structure $\mu_\L=\mu_A\otimes\mu_B$ implies the commutation relation $\mu_A\mu_B \mu_U=\mu_B\mu_A \mu_U $. Therefore,
\begin{align*}
	\mu_\L\left[\ent_A (\mu_B(f))\right]
	&= \mu_\L\left[\mu_B\mu_U (f)\log\left(\frac{\mu_B\mu_U (f)}{\mu_B\mu_A\mu_U (f)}\right)\right] \\ &= \mu_\L\left[\mu_U (f)\log\left(\frac{\mu_B\mu_U (f)}{\mu_B\mu_A\mu_U (f))}\right)\right]\\ &=
	\mu_\L\left[\mu_A \left[\mu_U (f)\log\left(\frac{\mu_B\mu_U (f)}{\mu_A\mu_B\mu_U (f)}\right)\right]\right].
	\end{align*}
It remains to observe that
$$
\mu_A \left[\mu_U (f)\log\left(\frac{\mu_B\mu_U (f)}{\mu_A\mu_B\mu_U (f)}\right)\right]
%\leq \mu_A \left[\mu_U (f)\log\left(\frac{\mu_U (f)}{\mu_A\mu_U (f)}\right)\right] =
\leq \Ent_A(\mu_U (f)).
 $$
 The latter estimate follows from the well known variational principle
 \begin{align}\label{varprin}
\Ent_A(g)	=\sup\,\{\mu_A(gh)\,,\; \mu_A(e^h)\leq 1\}
\end{align}
valid for any $A$ and any function $g\geq0$; see, e.g.\ \cite[Proposition 2.2]{Ledoux}.
\end{proof}

\subsection{Implications of block factorization}
Fix a probability distribution $\a$ over subsets of $V$ and observe that the $\a$-weighted heat bath block dynamics defined in the introduction is the Markov chain with transition matrix $P_\a$ on $\Omega$ such that for any real function $f$
\begin{align}\label{alphaweighted}
	P_\a f=\sum_{B\subset V}\a_B \,\mu_B (f)\,.
	\end{align}
To clarify the above notation, if we evaluate the left hand side at a spin configuration $\sigma\in\Omega$ then each for each $B$ the term $\mu_B f$ in the right hand side is given by $\mu_B^\t f$
where $\t=\sigma_{V\setminus B}$. If $\a_B=n^{-1}\,\IND(|B|=1)$, then \eqref{alphaweighted} is the Glauber dynamics for $\mu$.

The $\a$-weighted heat bath block dynamics \eqref{alphaweighted} defines a reversible pair $(P_\a,\mu)$. Moreover, its Dirichlet form satisfies
\begin{align}\label{alphadir}
	\cD_\a(f,g)=\sum_{B\subset V}\a_B \,\mu[f(1-\mu_B)g] = \sum_{B\subset V}\a_B \,\mu\left[\cov_B(f,g)\right] \,,
	\end{align}
	where $\cov_B(f,g)=\mu_B\left[(f-\mu_B f)(g-\mu_B g)\right]$ denotes the covariance functional.

 \begin{lemma}
\label{lem:gbf}
If the spin system satisfies the general block factorization with constant $C$ then for all $\a$ the Markov chain $(P_\a,\mu)$ satisfies
%the Markov chain $P_\a$ in \eqref{alphaweighted} satisfies:
\begin{enumerate}
\item the modified log-Sobolev inequality with constant $\rho = \frac{\delta(\alpha)}{C}$;
\item the discrete time relative entropy decay with rate $\d = \frac{\delta(\alpha)}{C}$;
\item
 $T_{\mathrm{mix}}(P_\a) \le 1+ \frac{C}{\delta(\alpha)} [\log(8)+\log\log(1/\mu_*)]
,$ where $\mu_* = \min_{\si\in\Omega} \mu(\si)$.
\end{enumerate}
\end{lemma}
\begin{proof}
In view of Lemma \ref{reltmix} it is sufficient to prove item 2. We note that the relative entropy decay with rate $\d$ is equivalent to the entropy contraction
\begin{equation}
\label{eqn:contraction}
\Ent(P_\a f)\leq (1-\d) \Ent(f),
\end{equation}
for all $f\geq 0$. By convexity of $x\mapsto x\log x$ one has
\begin{align}
\label{eqn:contraction1}
\Ent(P_\a f)&=\mu[P_\a f\log(P_\a f)] - \mu[f]\log\mu[f] \nonumber \\
&\leq \sum_B \a_B\,\mu[\mu_B (f)\log(\mu_B (f))]- \mu[f]\log\mu[f] = \sum_B \a_B\Ent(\mu_B (f))].
\end{align}
From the decomposition in Lemma \ref{lem:telescope}  it follows that
\begin{align}
\label{eqn:contraction2}
\Ent(P_\a f)\leq \Ent(f) - \sum_B \a_B\mu[\Ent_B (f)].
\end{align}
By definition of block factorization we conclude
\begin{align}
\label{eqn:contraction3}
\Ent(P_\a f)\leq (1-\d(\a)/C)\Ent(f). &\qedhere
\end{align}
\end{proof}

\section{Uniform block factorization implies general block factorization}
\label{sec:ubf:to:gbf}
%
%Start with a strenghtening of Lemma \ref{lem:mono} \what
%\begin{lemma}
%\label{lem:mono2}
%Moreover, if $\mu_{\L\cup B}$ is a product measure $\mu_{\L\cup B}=\mu_\L\otimes\mu_B$, for some $\L,B\subset V$, then for all $A \subset B,$
%	\begin{gather}
%	\label{eq:prelim:b3}
%	\mu_{\L\cup B}\left[\Ent_\L(\mu_B f)\right]\leq \mu_{\L\cup B}\left[\Ent_\L(\mu_A f)\right].
%	\end{gather}
%\end{lemma}
%\begin{proof}
%\end{proof}

%Recall the notion of General Block Factorization (GBF) defined in~\eqref{eqn:block-factorization} in Section~\ref{sec:si-consequences}.
A key step in the proof of Theorem \ref{thm:block-factorization} is the proof that uniform block factorization (UBF) implies general block factorization (GBF).
%A special case
%of General Block Factorization (GBF)---as defined in~\eqref{eqn:block-factorization} in Section~\ref{sec:si-consequences}---is known as \emph{Uniform Block Factorization (UBF)}.
%UBF corresponds to the case when probability distribution $\alpha$ is the uniform distribution over all subsets of a fixed size~$\ell$.
%In this section, we show, for any spin system on any graph of bounded degree, that when UBF holds for $\ell = \lceil\theta n\rceil$ with  $\theta$ sufficiently small,
%then GBF holds.
%We will later prove that spectral independence implies UBF for any desired constant $\theta$ (see Theorem~\ref{th:reform}).
%These two results combined imply Theorem~\ref{thm:block-factorization} from the introduction.

%Hence, if one establishes factorization of entropy for uniform blocks of a certain size (even linear in $n$) then GBF holds and
%factorization of entropy holds for arbitrary blocks. This is very interesting in that it yields a surprising and powerful comparison result: at a high-level,
%it is saying if a form of entropy mixing holds for a particular block dynamics, then this entropy mixing also holds for arbitrary block dynamics.

We begin with the formal definition of UBF.
For a positive integer $\ell\leq n $, let $\binom{V}{\ell}$ denote the collection of all subsets of $V$ of size $\ell$.

\begin{definition}[Uniform Block Factorization (UBF)]\label{def:ubf}
	We say that the spins system  %a Gibbs distribution
	$\mu$ satisfies the \emph{$\ell$-uniform block factorization ($\ell$-UBF) of entropy} with constant $\Cellubf$ if for all $f: \Omega \to \R_{+}$
	\begin{equation}\label{eq:ubf}
	\frac{\ell}{n} \, \Ent (f) \le \Cellubf \cdot \frac{1}{\binom{n}{\ell}} \sum_{S\in \binom{V}{\ell}} \mu[\Ent_S(f)].
	\end{equation}
\end{definition}

In this section, we establish the following theorem.

\begin{theorem}\label{thm:UBF-GBF}
	For an arbitrary $b$-marginally bounded spin system on a graph of maximum degree~$\Delta$, if
	$\left \lceil{\theta n}\right \rceil $-UBF holds with constant $\Cellubf$ and $0<\theta \le \frac{b^{2(\Delta+1)}}{4e\Delta^2}$,
	%\eric{Do we need $\ell=\Theta(n)$, or just that $\ell\leq\theta n$?  I.e., does the theorem hold for $\ell=o(n)$ as well?}
	then GBF holds with constant $\Cgbf=\Cubf\times O\left( (\theta\,b^2)^{-1}\log (1/b) \Delta^3\right)$.
	%\eric{I don't think $k$ is defined yet?  Replace with $\Delta+1$?  Though the discussion then has to change a bit.}
\end{theorem}

As it will become apparent in our proof of this theorem, there is a trade-off between the upper bound for $\theta$ and the value we can deduce for $\Cgbf$; in particular, we could allow for UBF to hold for larger $\theta < 1$ (i.e., with a better dependence on $\Delta$) at the expense of an additional factor depending on $\Delta$ in $\Cgbf$. %Note that the theorem holds even when $\theta = o(1)$ as a function of $n$. 
%We note also that Theorem~\ref{thm:block-factorization} from the introduction follows immediately from this result and the fact that UBF is implied by spectral independence.

We will later show that spectral independence implies $\left \lceil{\theta n}\right \rceil $-UBF for any desired constant $\theta>0$; see Theorem~\ref{th:reform}.
These two results combined imply Theorem~\ref{thm:block-factorization} from the introduction:
\begin{proof}[Proof of Theorem~\ref{thm:block-factorization}]
	In Theorem~\ref{th:reform} we establish that 
	for a spin system that is $\eta$-spectrally independent and $b$-marginally bounded, 
	$\left \lceil{\theta n}\right \rceil $-UBF holds with constant $\Cellubf = (\frac{1}{\theta})^{O(\frac{\eta}{b})}$. Then, taking $\theta=\frac{b^{2(\Delta+1)}}{4e\Delta^2}$, Theorem~\ref{thm:UBF-GBF} implies 
	that GBF holds with constant 
	$$
	\Cgbf=O\left( \frac{4e\Delta^5}{b^{2(\Delta+2)}}\log (1/b) \right)\times \left( \frac{4e\Delta^2}{b^{2(\Delta+1)}} \right) ^{O\(\frac{\eta}{b}\)} = 
	\left(\frac{2}{b}\right)^{O\(\Delta(1+\frac{\eta}{b})\)},
	%\left( \frac{2}{b} \right) ^{O(\Delta(\eta+1)/b)},
	$$
	as claimed.
\end{proof}

We turn to the proof of Theorem \ref{thm:UBF-GBF}. Recall that a graph $G$ of maximum degree $\Delta$ is $k$-partite, with $k\leq \Delta + 1.$ Let $\{V_1,...,V_k\}$ denote the independent sets $V_i \subset V$ corresponding to a $k$-partition of $G$.
A key step in the proof of Theorem~\ref{thm:UBF-GBF} is to establish the following factorization statement. 

\begin{lemma}\label{lem:ubf:boosting}
	Suppose that for an arbitrary $b$-marginally bounded spin system on a graph of maximum degree $\Delta$,  $\left \lceil{\theta n}\right \rceil $-UBF holds with constant $\Cubf$
	and $\theta \le \frac{b^{2(\Delta+1)}}{4e\Delta^2}$. Then,
	%there exists a constant $C=C(\Cubf,\Delta,b)$ such that
	\begin{equation}
	\label{eq:kpart}
	\ent(f)\leq K \Cubf\sum_{i=1}^k\mu[\Ent_{V_i}(f)],
	\end{equation}
	where the constant $K$ satisfies $K=O(\Delta^2(\theta\,b^2)^{-1}\log (1/b) )$.
\end{lemma}

We refer to inequality~\eqref{eq:kpart} as a $k$-partite factorization of entropy with constant $K \Cubf$. 
Once we have Lemma \ref{lem:ubf:boosting}, Theorem \ref{thm:UBF-GBF} is implied by the following lemma.

\begin{lemma}\label{lem:kpart:gen}
	Suppose that for an arbitrary spin system on a graph of maximum degree $\Delta$, $k$-partite factorization of entropy holds with constant $C$.  Then, GBF holds with constant $Ck$.
\end{lemma}	

%\begin{proof}[Proof of Theorem~\ref{thm:block-factorization}]
%	Follows from Theorem~\ref{1.5} from the introduction, and Lemmas~\ref{lem:ubf:boosting} and \ref{lem:kpart:gen}.
%\end{proof}

We provide next the proofs of Lemmas~\ref{lem:kpart:gen} and \ref{lem:ubf:boosting}. %We prove Lemma \ref{lem:kpart:gen} first.

\begin{proof}[Proof of Lemma~\ref{lem:kpart:gen}]
	Let $\alpha=(\alpha_B)_{B \subset V}$ be a probability distribution over the subsets of $V$.
	Observe that for all $j = 1,...,k$ and all $\tau \in \Omega(V \setminus V_j)$,
	$\mu_{V_j}^\tau$ is a product measure on $\Omega_{V_j}^\tau$.
	Therefore, we can apply Lemma \ref{lem:shearer} with $\Lambda=V_j$
	and $\hat{\alpha}=(\hat{\alpha}_U)_{U \subset V_j}$, where $\hat{\alpha}_U=\omega^{-1} \sum_{B \subset V} \alpha_B \IND(V_j \cap B=U)$
	and $\omega = \sum_{B \subset V} \alpha_B \IND(V_j \cap B \neq \emptyset)$
	. We get
	\begin{align}\label{shear1}
	\d(\hat \alpha)\,\ent_{V_j}^\tau (f) \le &\sum_{U\subset V_j}\hat{\a_U} \,\mu_{V_j}^\tau[\ent_{U} (f)] = \omega^{-1}\sum_{B\subset V}\a_B \,\mu_{V_j}^\tau[\ent_{V_j\cap B} (f)].
	%	\d({\a\vert_{V_j}})\,\ent_{V_j}^\tau (f) \le &\sum_{B\subset V}\a_B \,\mu_{V_j}^\tau[\ent_{V_j\cap B} (f)],
	\end{align}
	%	where $\d({\a\vert_{V_j}})=\min_{x\in {V_j}}\sum_{B \subset V: B\ni x}\a_B$.
	%Observe that $\d({\a\vert_{V_j}}) \geq \d({\alpha})$
	Observe that
	$$
	\omega \d(\hat \alpha) = \min_{x\in {V_j}}\sum\nolimits_{U \subset V_j: U\ni x}{\hat \a}_U = \min_{x\in {V_j}}\sum\nolimits_{B \subset V: B\ni x}\a_B \ge \d(\alpha) ,
	$$
	and from~\eqref{eq:prelim:b1} we have $\mu[\ent_{V_j \cap B}(f)]\leq\mu[\ent_B(f)]$. Hence, taking expectation in~\eqref{shear1} with respect to $\mu$ we obtain
	\begin{align*}
	\d(\alpha)\,\mu[\ent_{V_j} (f)] \le &\sum_{B\subset V}\a_B \,\mu[\ent_{B} (f)].
	\end{align*}
	Summing over $j$ we have, for all $f: \Omega \to \R_{+},$
	\begin{align*}
	\d(\alpha)\, \sum_{j=1}^k \mu[\ent_{V_j} (f)] \le &\sum_{j=1}^k \sum_{B\subset V}\a_B \,\mu[\ent_{B} (f)],
	\end{align*}
	and since by assumption $k$-partite factorization of entropy holds with constant $C$, we have
	\begin{align*}
	\d(\alpha)\ent(f) & \leq C \sum_{j=1}^k\sum_{B\subset V}\a_B \,\mu[\ent_{B} (f)] \leq C \, k \sum_{B\subset V}\a_B \,\mu[\ent_{B} (f)].
	\end{align*}
	Hence, GBF holds with constant $C k$.
\end{proof}

The main idea behind the proof of Lemma \ref{lem:ubf:boosting} can be roughly explained as follows. The $\ell$-UBF assumption
with $\ell\sim\theta n$ is the factorization statement \eqref{eq:ubf}. If the set $S$ in \eqref{eq:ubf} were an independent set, then suitable applications of Lemma \ref{lem:telescope} and Lemma \ref{lem:prod:fact} would yield the desired conclusion. Moreover, the same conclusion would continue to hold if $S$ were made of bounded connected components. The delicate part of the argument consists in exploiting the fact that if $\theta$ is sufficiently small then one can effectively reduce the problem to case of bounded connected components.     

\begin{proof}[Proof of Lemma \ref{lem:ubf:boosting}]
	Since $\left \lceil{\theta n}\right \rceil$-UBF holds by assumption, setting $C = \Cubf$ one has
	\begin{equation}
	\label{eq:unif_fact}
	\ent(f)\leq \frac C\theta\,\bbE\left[\mu\left[\ent_S (f)\right]\right],
	\end{equation}
	where $S$ is a random set with uniform distribution over all subsets of $V$
	of cardinality $\left \lceil{\theta n}\right \rceil$, and $\bbE$ denotes the corresponding expectation.
	
	Let $S_1,S_2,\dots$ denote the connected components of $S$ in $G$ (taken in some arbitrary order) and for $i>1$ let $S_{<i}=\cup_{j=1}^{i-1}S_j$. Then $\mu_{S_{<i+1}}$ has the product structure $\mu_{S_{<i+1}}=\otimes_{j=1}^i\mu_{S_{j}}$.
	%	a fixed configuration $\tau$ on $\Omega(V \setminus S)$ and $i>1$,
	%	we let $\nu_i^\tau$ denote the distribution $\nu_i^\tau=\mu_{S_{<i}}^\tau$, where $S_{<i}=\cup_{j=1}^{i-1}S_j$.
	%	As usual, we let $\nu_i(f):\Omega(V \setminus S) \mapsto \bbR_+$  be the function $\nu_i(f)(\tau) = \nu_i^\tau(f)$ and set $\nu_1(f) = f$.\antonio{Do we use a different notation for the set of configurations on a subset of $V$?} 	
	By Lemmas~\ref{lem:telescope} and \ref{lem:prod:fact}, one has the decomposition
	\begin{equation}
	\label{eq:kfact1}
	\mu\left[\ent_S (f)\right] = \sum_{i\ge 1} \mu\left[\ent_{S_{<i+1}} (\mu_{S_{<i}}(f))\right] = \sum_{i\ge 1} \mu[\ent_{S_i} (\mu_{S_{<i}}(f))],
	\end{equation}
	where we have used Eq.\  \eqref{prod} with $A=S_i$ and $B=S_{<i}$.
	For $\tau \in \Omega(V \setminus S_i)$,
	let $\Gamma(S_i,\tau)$ be the optimal constant so that
	$$
	\ent_{S_i}^\tau (\mu_{S_{<i}}(f)) \le \G(S_i,\tau)\,\sum_{j=1}^k  \mu_{S_i}^\tau\left[\ent_{V_j\cap S_i} (\mu_{S_{<i}}(f))\right].
	$$
	Let $\Gamma(S_i) = \max_{\tau \in
		\Omega(V \setminus S_i)} \Gamma(S_i,\tau)$. Then,
	$$
	\mu\left[\ent_S(f)\right] \leq \sum_{i \ge 1} \G(S_i)\, \sum_{j=1}^k \mu\left[\ent_{V_j\cap S_i} (\mu_{S_{<i}}(f))\right].
	$$
	We observe next that for all $j=1,...,k$ one has
	\begin{equation}
	\label{eq:ent:int}
	\mu\left[\ent_{V_j\cap S_i} (\mu_{S_{<i}}(f))\right]\leq \mu\left[\ent_{V_j\cap S_i} (\mu_{V_j\cap S_{<i}}(f))\right].
	\end{equation}
	%where $\nu_{ij}$ denotes the measure $\mu_{V_j\cap S_{<i}}$.
	To see this, we apply Lemma \ref{lem:prod:fact} with $A=V_j\cap S_i$, $B=S_{<i}$ and $U=V_j\cap S_{<i}$.
Since $\mu_{S_{<i+1}}=\otimes_{j=1}^{i}\mu_{S_j}$ the assumptions for that lemma are satisfied and we obtain \eqref{eq:ent:int} from Eq.\ \eqref{prodest}.

	Summarizing, we have obtained
	\begin{equation}
	\label{eq:unif_fact2}
	\ent(f)\leq \frac C\theta \,\sum_{j=1}^k\bbE \left[\sum_{i \ge 1} \G(S_i)\, \mu\left[\ent_{V_j\cap S_i} (\mu_{V_j\cap S_{<i}}(f))\right]
	\right].
	\end{equation}
	We show next that for %a suitable constant $C' > 0$, we have for 
	all $j=1,...,k$
	\begin{equation}
	\label{eq:unif_fact3}
	\bbE\left[\sum_{i \ge 1} \G(S_i)\,\mu\left[\ent_{V_j\cap S_i} (\mu_{V_j\cap S_{<i}}(f))\right]
	\right]\leq C' \mu\left[\ent_{V_j} (f)\right],
	\end{equation}
	with $C'=O\left(\frac{ \log (1/b)}{b^2}\Delta^2\right)$. 
	Combined with \eqref{eq:unif_fact2}, this concludes the proof of the lemma.
	
	Let us fix $j$
	%	and let $v_{i_1},v_{i_2},\dots$ denote an ordering of the sites in $V_j\cap S_i$.
	and let $v_1, v_2, \dots$ denote an ordering of the sites in $V_j\cap S$ such that
	$v_1, . . . , v_{|V_j\cap S_1|}$ is an ordering of $V_j\cap S_1$,
	$v_{|V_j\cap S_1|+1}, . . . , v_{|V_j\cap S_1|+|V_j\cap S_2|}$ is an ordering of $V_j\cap S_2$ and so on.
	Since, for all $i\geq 1$, $\mu_{V_j\cap S_i}$ is a product measure, Lemmas~\ref{lem:telescope} and \ref{lem:prod:fact} (as in \eqref{eq:kfact1})
	imply
	$$
	\mu\left[\ent_{V_j\cap S_i} (\mu_{V_j\cap S_{<i}}(f))\right]
	= \sum_{h=|V_j \cap S_1|+\dots+|V_j \cap S_{i-1}|+1}^{ |V_j \cap S_1|+\dots+|V_j \cap S_{i}|}\mu\left[\ent_{v_{h}} (\rho_{v_{h}}(f))\right],
	$$
	where $\rho_{v_{h}}$ is the conditional distribution obtained from $\mu$ by freezing the spins at all the sites outside $V_j$, together with all the sites $v_{h},v_{h+1},\dots,v_{|V_j \cap S|}$.
	%;  note that $\rho_{v_{i_1}}(f)=\mu_{V_j\cap S_{<i}}(f)$.
	
	Using this decomposition and rearranging one finds
	\begin{align}
	\label{eq:unif_fact4}
	\bbE\left[ \sum_{i \ge 1}\G(S_i)\,\mu\left[\ent_{V_j\cap S_i} (\mu_{V_j\cap S_{<i}}(f))\right]
	\right] & = \bbE\left[\sum_{i \ge 1 } \G(S_i) \sum_{h=|V_j \cap S_1|+\dots+|V_j \cap S_{i-1}|+1}^{ |V_j \cap S_1|+\dots+|V_j \cap S_{i}|} \mu\left[\ent_{v_{h}} (\rho_{v_{h}}(f))\right]\right]
	\nonumber\\ & = \bbE\left[\sum_{h}\mu\left[\ent_{v_{h}} (\rho_{v_{h}}(f))\right]\G(S(v_{h}))\right],
	\end{align}
	where $S(v_{h})$ denotes the (unique) connected component of $S$ containing $v_{h}$.
	Notice that for each realization of $S$, $\mu_{V_j\cap S}$ is a product measure and so one has from Lemmas~\ref{lem:telescope} and \ref{lem:prod:fact} that
	$$
	\sum_{h}\mu\left[\ent_{v_{h}} (\rho_{v_{h}}(f))\right] = \mu\left[\ent_{V_j\cap S}(f))\right]\leq \mu\left[\ent_{V_j}(f)\right];
	$$
	the inequality follows from~\eqref{eq:prelim:b1}.
	
	Observe that each term $\mu[\ent_{v_{h}} (\rho_{v_{h}}(f))]$, as well as the sequence $\{v_h\}$, depends on the realization $S$ only through $V_j\cap S$. Therefore,
	$$
	\bbE\left[\sum_{h}\mu\left[\ent_{v_{h}} (\rho_{v_{h}}(f))\right]\G(S(v_{h}))\right]=
	\bbE\left[\sum_{h}\mu\left[\ent_{v_{h}} (\rho_{v_{h}}(f))\right]\bbE\left[\G(S(v_{h}))\tc V_j\cap S\right]\right],
	$$
	where $\bbE\left[\G(S(v_h))\tc V_j\cap S\right]$ is the conditional expectation of $\G(S(v_{h}))$ given the realization $V_j\cap S$.
	Therefore, \eqref{eq:unif_fact3} follows if we prove that
	
	\begin{equation}
	\label{eq:unif_fact5}
	\max_{W\subset V_j}\max_{v\in W} \,\bbE\left[\G(S(v))\tc V_j\cap S=W\right] \le C'.
	\end{equation}
	Now, for a $b$ marginally bounded spin system,
	it follows from Lemma 4.2 in~\cite{CLV20}
	and~\eqref{eq:prelim:b1}
	that
	$$
	\G(S(v))\leq  \zeta |S(v)|^3 z^{|S(v)|},
	$$
	where $\zeta = \zeta(b) = \frac{3 \log (1/b)}{2b^2}$ and $z = 1/b^2$.
	Thus,
	\begin{equation}
	\label{eq:unif_fact6}
	\max_{W\subset V_j}\max_{v\in W} \,\bbE\left[\G(S(v))\tc V_j\cap S=W\right] \le \zeta \cdot \max_{W\subset V_j}\max_{v\in W} \,\bbE\left[|S(v)|^3 z^{|S(v)|} \,\mid\, V_j\cap S=W\right].
	\end{equation}
	
	To bound the expectation on the right-hand-side of~\eqref{eq:unif_fact6}, we consider
	the graph $G_2$ with vertex set $V$ and edge set $E\cup E_2$, where $E$ is the edge set of $G$ and $E_2$ is the set of all pairs of vertices with a common neighbor in $G$. Note that $G_2$ has maximum degree $\Delta^2$. Let ${\mathcal{A}_v(a)}$ be the collection of subsets of  vertices $U \subset V$ such that $|U|\geq a$, $v \in U$ and the induced subgraph $G_2[U]$ of $U$ in $G_2$ is connected.
%Similarly, let $\mathcal{A}^{=}_v(a)$ be the collection of subsets $U \subset V$ such that $|U|= a$, $v \in U$ and $G_2[U]$ is connected.
	
	Now, let us fix the set $W=V_j \cap S$ and the vertex $v \in W$ and let $S_2 :=( S(v)\cap V_{\neq j} )\subset S$, where $ V_{\neq j}:=\bigcup_{i:i \neq j} V_i$. We claim that when the event $\{|S(v)|=a\}$ occurs for some $a\in \bbN$, then $S_2 \in \mathcal{A}_v(\frac{a}{\D+1})$.
	Indeed,
	$G_2[S_2]$ is connected, since $S(v)$ is connected in $G$ and
	removing the vertices in $V_j \setminus \{v\}$ from $S(v)$ will not disconnect $S_2$ in $G_2$. Moreover, $\Delta|S(v)\cap V_{\neq j} | \ge  |S(v)\cap V_j|$, and so
	$$
	a = |S(v)\cap V_j| + |S(v)\cap V_{\neq j} | \le (\Delta+1)|S(v)\cap V_{\neq j}|,
	$$
	which implies that $|S_2|=|S(v)\cap V_{\neq j}|\geq a/(\Delta+1)$.
	Given $S$, let $T_2(v)$ denote the connected component of $S$  in $G_2$ containing $v$, and note that $S_2 \subset T_2(v)$.
	Then, for any $W \subset V_j$, $v \in W$ and integer $a \ge 1$ we get
	\begin{align}
	\bbP\left( |S(v)|= a \tc V_j\cap S=W\right) 
	&\le \bbP\left(\exists \, S_2 \in {\mathcal{A}_v\left(\frac{a}{\Delta+1}\right)} ;S_2 \subset S \right) \notag\\
	&\le \bbP\left(|T_2(v)| \ge  \frac{a}{\Delta+1} \right).
%	 \notag\\
%	&\leq\sum_{m \ge \frac{a}{\Delta+1}}\bbP\left(|T_2(v)|  = m\right) \notag\\
%	& \leq \lceil \theta \rceil\sum_{m \ge\frac{a}{\Delta+1}}(2e\Delta^2\theta)^{m-1}  \label{eq:lemma4.3} \\  
%	& \leq 2\lceil \theta \rceil (2e\Delta^2\theta)^{\frac{a}{\Delta+1}-1} \ , \label{eq:f1}
	\end{align}
	Next we use Lemma 4.3 from \cite{CLV20}, which implies that for any integer $m\geq 1$,
%	says that if we pick a set $T\subset V$ uniformly at random among all sets of size $\ell=\lceil \theta n\rceil$, then letting $T(v)$ denote the connected component at $v$ in a graph $G'=(V,E')$ with maximum degree $K$ one has
%	  
\begin{align}
	\bbP\left(|T_2(v)|=m\right)\leq  \frac{\ell}n(2e \Delta^2 \theta)^{m-1}.
		\end{align}
		Indeed, the only difference with respect to Lemma 4.3 from \cite{CLV20} is that we have maximum degree $\Delta^2$ here instead of $\Delta$.
		In particular, if $2e \Delta^2 \theta\leq 1/2$,  using $\frac{\ell}n\leq 2\theta$,
		\begin{align}
	\bbP\left(|T_2(v)|\geq \frac{a}{\Delta+1}\right)\leq  4\theta(2e \Delta^2 \theta)^{\lfloor\frac{a}{\Delta+1}\rfloor-1}\leq\Delta^{-2}(2e \Delta^2 \theta)^{\lfloor\frac{a}{\Delta+1}\rfloor}.
		\end{align}
It follows that 
	\begin{align}
	\bbE\left[|S(v)|^3 z^{|S(v)|}\tc V_j\cap S=W\right] &= \sum_{a \geq 1} a^3 z^a \cdot \bbP\left(|S(v)|=a \tc V_j\cap S=W\right) \\
%	&\le 3\theta z^{\Delta+1}
%	  \sum_{a \geq 1} a^3  (2e\Delta^2\theta z^{\Delta+1})^{\lfloor\frac{a}{\Delta+1}\rfloor-1}\\
	&\le \Delta^{-2} \sum_{a \geq 1} a^3  (2e\Delta^2\theta z^{\Delta+1})^{\lfloor\frac{a}{\Delta+1}\rfloor} \leq C_1\Delta^2,
	\end{align}
for some absolute constant $C_1$ provided that $2e\Delta^2\theta z^{\Delta+1}\leq 1/2$.
	%	\begin{align}
	%	\bbE\left[z^{|S(v)|}\tc V_j\cap S=W\right] & = \sum_{a \geq 1}\bbP\left(z^{|S(v)|}>a \tc V_j\cap S=W\right) \notag\\
	%	& = \sum_{a \geq 1}\bbP\left( |S(v)|> \log_z a \tc V_j\cap S=W\right) \notag\\
	%	&\leq\sum_{a \geq 1}\bbP\left(\exists \, S_2 \in {\mathcal{A}_v\left(\frac{\log_z a}{\Delta+1}\right)} :S_2 \subset S \right) \notag\\
	%	&\leq\sum_{a \geq 1}\sum_{m > \frac{\log_z a}{\Delta+1}}\bbP\left(\exists \, S_2 \in \mathcal{A}^{2,=}_v(m) :S_2 \subset S \right) \notag\\
	%	& \leq \lceil \theta \rceil\sum_{a \geq 1}\sum_{m > \frac{\log_z a}{\Delta+1}}(2e\Delta^2\theta)^{m-1}  \label{eq:lemma4.31} \\
	%	& \leq C'' \cdot \sum_{a \geq 1} (2e\Delta^2\theta)^{\frac{\log_z a}{\Delta+1}} \ \label{eq:f12} \\
	%	&=O(1)  \label{eq:f2},
	%	\end{align}
	%	where \eqref{eq:lemma4.3} follows by Lemma 4.3 from \cite{CLV20},
	%	and \eqref{eq:f1} holds provided $\theta \leq\frac{1}{12\Delta^2}$ for a suitable constant $C'' = C''(\Delta)$, and \eqref{eq:f2} follows when $\theta$ is such that
	%	$(2e\Delta^2\theta)^{1/(\Delta+1)} \le 1/z^2$.
	This implies that
	$$
	\max_{W\subset V_j}\max_{v\in W} \,\bbE\left[|S(v)|^3z^{|S(v)|} \,\tc\, V_j\cap S=W\right] \le C_1\Delta^2.
	$$
	Hence,~\eqref{eq:unif_fact5} and~\eqref{eq:unif_fact3} hold with $C' = C_1 \zeta \Delta^2$,
	and so $k$-partite factorization holds with constant $\Cubf C_1 \zeta \Delta^2/\theta$.
\end{proof}

\section{Spectral independence for contractive distributions}
\label{sec:contraction-SI}
In this section we establish our main results that a contractive distribution is spectrally independent.
These results in particular connect classic probabilistic approach for establishing fast mixing of Markov chains such as coupling with recent developments utilizing spectral independence.
%We will first consider a simple case in Section~\ref{subsec:glauber-hamming} concerned with contraction w.r.t.\ Glauber dynamics and Hamming metric, and establish a special case of Theorem~\ref{thm:glauber} serving as a concrete example for illustrating our proof approach.
We first consider a special case of Theorem~\ref{thm:glauber} concerned with Glauber dynamics and Hamming metric in Section~\ref{subsec:glauber-hamming}; this will serve as a concrete example to illustrate our approach for establishing spectral independence.
In Section~\ref{subsec:glauber-wHamming}, we consider arbitrary metric and prove Theorem~\ref{thm:glauber}.
Finally, we consider general Markov chains and metrics in Section~\ref{subsec:general-general} and prove Theorem \ref{thm:general}.

%We begin with a simple technical lemma which is the heart of the proofs.

%One way of establishing spectral independence is by considering weighted sum of absolute influences under arbitrary vertex weights as is done in~\cite{CLV1}.
%More specifically, if there exist vertex weights $\rho:V\rightarrow \R_+$ and a constant $\eta$ such that,
%for all $x,a$, we have
%$$\sum_{y\in V,\si_y} \rho_v \cdot |J^\t(x, \si_x; y, \si_y)|\leq \eta$$
%then $\lambda_1(J^\t)  \leq \eta.$.

\subsection{Warm-up: contraction for Glauber dynamics and Hamming metric}
\label{subsec:glauber-hamming}
In this section, we prove a simpler version of Theorem~\ref{thm:glauber}, which already gives the main idea of our proof approach for establishing spectral independence.
We show that, if the distribution $\mu$ is contractive w.r.t.\ the Glauber dynamics and the Hamming metric, then it is spectrally independent.

\begin{theorem}\label{thm:glauber-hamming}
If $\mu$ is $\ctrct$-contractive w.r.t.\ the Glauber dynamics and the Hamming metric for some $\ctrct\in(0,1)$,
then $\mu$ is spectrally independent with constant $\eta =\frac{2}{(1-\ctrct)n}$.
In particular, if $\ctrct\le 1-\eps / n$, then $\eta \leq 2/\eps$.
\end{theorem}

\begin{remark}\label{rmk:glauber-pinning}
In this paper, the Glauber dynamics $P^\tau_{\textsc{gl}}$ for the conditional distribution $\mu^\tau$ with a pinning $\tau$ on $U\subset V$ is defined as follows: in each step the chain picks a vertex $\vx \in V$ u.a.r.\ and updates its spin conditioned on all other vertices and $\tau$.
In particular, all pinned vertices in $U$ are allowed to be selected and when this happens the configuration will remain the same (no updates will be made).
This setting can make our theorem statements and proofs easier to understand, and will not harm our results since we only consider these chains for the purpose of analysis rather than actually running them.
Alternatively, we can define the Glauber dynamics $\tilde{P}^\tau_{\textsc{gl}}$ for $\mu^\tau$ in the following way: in each step an \emph{unpinned} vertex $\vx \in V \setminus U$ is selected u.a.r.\ and updated accordingly.
Note that $\tilde{P}^\tau_{\textsc{gl}}$ is faster than $P^\tau_{\textsc{gl}}$ and the contraction rate of $\tilde{P}^\tau_{\textsc{gl}}$ depends on the number of unpinned vertices.
If we assume $\mu^\tau$ is $\ctrct_\ell$-contractive w.r.t.\ $\tilde{P}^\tau_{\textsc{gl}}$ and $d_{\mathrm{H}}$ where $\ell = |V \setminus U|$, then an analog of Theorem~\ref{thm:glauber-hamming} can show that $\mu$ is spectrally independent with
$$ \eta = \max_{\ell = 1,\dots, n} \left\{\frac{2}{(1-\ctrct_\ell) \ell}\right\}. $$
However, in actual applications such as under the Dobrushin uniqueness condition in Section~\ref{subsec:dob-si}, the contraction rate satisfies $\ctrct_\ell \le 1-\eps/ \ell$, so we eventually get $\eta \le 2/\eps$ just as from Theorem~\ref{thm:glauber-hamming}.
%\zongchen{New!}
\end{remark}

Recall that for any pinning $\tau \in \Pinning$ we let $\mu^\tau$ be the conditional distribution over $\Omega^\tau$ given $\tau$,
and the ALO influence matrix $\ALOi^\tau$ is a square matrix indexed by $\VSpairs^\tau$ and defined as $\ALOi(\vx, \spa; \vx, \spb) = 0$ and
\[
\ALOi^\tau(\vx, \spa; \vy, \spb) = \mu^\tau(\sigma_\vy = \spb \mid \sigma_\vx = \spa) - \mu^\tau(\sigma_\vy = \spb) ~\text{for}~ \vx \neq \vy.
\]
The distribution $\mu$ is said to be $\eta$-spectrally independent if $\lambda_1(\ALOi^\tau) \le \eta$ for all pinning $\tau$.

Our goal is to upper bound the maximum eigenvalue of the ALO influence matrix $\ALOi^\tau$ for a given pinning $\tau$.
In fact, to make notations simpler we will only consider the case where there is no pinning; the proof is identical by replacing $\Omega, \mu, J$ with $\Omega^\tau, \mu^\tau, J^\tau$ when an arbitrary pinning $\tau$ is given.
To upper bound $\lambda_1(\ALOi)$, a standard approach that has been applied in previous works \cite{ALO20,CLV1,CGSV21,FGYZ21,CLV20} is to upper bound the infinity norm of $J$.
More specifically, for each $(\vx,\spa) \in \VSpairs$ we define
\begin{equation}\label{eq:S_def}
S(\vx, \spa) = \sum_{(\vy, \spb) \in \VSpairs} | \ALOi(\vx, \spa; \vy, \spb) |
%= \sum_{(\vy, \spb) \in \VSpairs} | \mu(\sigma_\vy = \spb \mid \sigma_\vx = \spa) - \mu(\sigma_\vy = \spb) |
\end{equation}
to be the sum of absolute influences of a given pair $(\vx, \spa)$.
The quantity $S(\vx, \spa)$ can be thought of as the total influence of $(\vx, \spa)$ on all other vertex-spin pairs.
If one can show $S(\vx, \spa) \le \eta$ for all $(\vx,\spa) \in \VSpairs$, then it immediately follows that
\[
\lambda_1(J) \le \|J\|_\infty = \max_{(\vx, \spa) \in \VSpairs} S(\vx, \spa) \le \eta.
\]
Hence, it suffices to show that $S(\vx, \spa) = O(1)$.

Fix $(\vx,\spa) \in \VSpairs$, and define the distribution $\nu = \mu(\cdot \mid \sigma_\vx = \spa)$; namely, $\nu$ is the conditional distribution of $\mu$ with the pinning $\sigma_\vx = \spa$.
The key observation we make here is that the quantity $S(\vx, \spa)$ can be viewed as the difference of the expectation of some function $f$ under the two measures $\mu$ and $\nu$.
More specifically, we define
\begin{equation}\label{eq:f_def}
f(\si) = \sum_{(\vy, \spb)\in \VSpairs} t(\vx, \spa; \vy, \spb) \, \mathbf{1}_{\{\sigma_\vy = \spb\}},
\end{equation}
where
$$
t(\vx, \spa; \vy, \spb) = \sgn(J(\vx, \spa; \vy, \spb)) =
\begin{cases}+1, & J(\vx, \spa; \vy, \spb) > 0;\\
-1, & J(\vx, \spa; \vy, \spb) < 0;\\
0, & J(\vx, \spa; \vy, \spb) = 0.
\end{cases}
$$
With this definition it follows that
\begin{align*}
%\label{eq:sum_inf}
S(\vx, \spa) &= \sum_{(\vy, \spb) \in \VSpairs} t(\vx, \spa; \vy, \spb) \ALOi(\vx, \spa; \vy, \spb) \\
&= \sum_{(\vy, \spb) \in \VSpairs} t(\vx, \spa; \vy, \spb) \mu(\sigma_\vy = \spb \mid \sigma_\vx = \spa) - t(\vx, \spa; \vy, \spb) \mu(\sigma_\vy = \spb) \\
&=\bbE_\nu f - \bbE_\mu f.
\end{align*}
Therefore, the absolute sum of influences $S(\vx, \spa)$ describes, in some sense, the ``distance'' of the two distributions $\nu$ and $\mu$ measured by $f$.

To be more precise about our last statement, we review some standard definitions about the Wasserstein distance.
Let $(\Omega, d)$ be a finite metric space.
We say a function $f: \Omega \to \R$ is \emph{$L$-Lipschitz} w.r.t.\ the metric $d$ if for all $\sigma,\tau \in \Omega$ we have
\[
|f(\sigma) - f(\tau)| \le L d(\si,\tau).
\]
For every function $f: \Omega \to \R$, we let $L_d(f)$ be the optimal Lipschitz constant of $f$ w.r.t.\ the metric $d$; i.e., $L_d(f) = \inf \{L\ge0: \text{$f$ is $L$-Lipschitz w.r.t.\ $d$} \} $.
For a pair of distributions $\mu$ and $\nu$ on $\Omega$, the \emph{$1$-Wasserstein distance} w.r.t.\ the metric $d$ between $\mu$ and $\nu$ is defined as
\[
W_{1,d}(\mu, \nu) = \inf \left\{ \bbE_\pi [ d(\si,\tau) ] \mid \pi \in \cC(\mu,\nu) \right\}
\]
where $\cC(\mu,\nu)$ denotes the set of all couplings of $\mu,\nu$ (i.e., $\pi(\cdot,\cdot) \in \cC(\mu,\nu)$ is a joint distribution over $\Omega \times \Omega$ with the marginals on the first and second coordinates being $\mu$ and $\nu$ respectively) and $(\sigma,\tau)$ is distributed as $\pi$;
equivalently, the $1$-Wasserstein distance can be represented as
\begin{equation}
\label{eq:W1-dual}
W_{1,d}(\mu,\nu) = \sup \left\{ \bbE_\mu f - \bbE_\nu f \mid f: \Omega \to \R,\, L_d(f) \le 1 \right\}.
\end{equation}

Observe that, the function $f$ defined by \eqref{eq:f_def} is $2$-Lipschitz w.r.t.\ the Hamming metric $d_{\mathrm{H}}$;
%(recall that $\hamming{\sigma}{\tau} = | \{x \in V: \sigma_x \neq \tau_x\} |$ for $\sigma, \tau \in \Omega \subset [q]^V$).
to see this, if $\sigma, \tau \in \Omega$ and $\hamming{\sigma}{\tau} = k$ then by the definition of $f$ we have $|f(\si) - f(\tau)| \le 2k$.
Therefore, we deduce from \eqref{eq:W1-dual} that
\[
S(\vx, \spa) = \bbE_\nu f - \bbE_\mu f \le L_{d_{\mathrm{H}}}(f) \, W_{1,d_{\mathrm{H}}}(\nu, \mu) \le 2 W_{1,d_{\mathrm{H}}}(\nu, \mu).
\]
That means, if one can show $W_{1,d_{\mathrm{H}}}(\nu, \mu) = O(1)$ for $\mu$ and $\nu = \mu(\cdot \mid \sigma_\vx = \spa)$ for any pair $(\vx, \spa)$, then $\lambda_1(J) = O(1)$ and spectral independence would follow.

%
%For two distributions $\mu,\nu$ over $\Omega$, the \emph{1-Wasserstein distance w.r.t.\ the metric $d$} between $\mu$ and $\nu$ is defined as
%\begin{equation}
%\label{eq:W1-def}
%W_{1,d}(\mu, \nu) = \inf \left\{ \bbE_\rho [ d(\si,\tau) ] \mid \rho \in \cC(\mu,\nu) \right\}
%\end{equation}
%where $\cC(\mu,\nu)$ denotes the set of all couplings of $\mu$ and $\nu$, and $(\sigma,\tau)$ is distributed as $\rho$.
%Equivalently, the 1-Wasserstein distance can be represented as
%\begin{equation}
%\label{eq:W1-dual}
%W_{1,d}(\mu,\nu) = \sup \left\{ \bbE_\mu f - \bbE_\nu f \mid f: \Omega \to \R, L_d(f) \le 1 \right\}.
%\end{equation}
%
%Let $P$ be the transition matrix of a Markov chain on $\Omega$ with stationary distribution $\mu$.
%We say $P$ is \emph{$(1-\alpha)$-contractive w.r.t.\ the metric $d$} if for all $\sigma,\tau \in \Omega$ one has
%\[
%W_{1,d}(P(\sigma,\cdot), P(\tau,\cdot)) \le (1-\alpha) d(\si,\tau);
%\]
%this further implies that for every integer $t \ge 1$,
%\[
%W_{1,d}(P^t(\sigma,\cdot), P^t(\tau,\cdot)) \le (1-\alpha)^t d(\si,\tau).
%\]

%To establish our main results, we first give a result comparing the stationary distributions of two Markov chains in a very general setting.
%We begin with some standard definitions.

The following lemma, which generalizes previous works \cite{BN,RR19}, will be used to bound the Wasserstein distance of two distributions and may be interesting of its own.
Roughly speaking, it claims that if $\mu,\nu$ are the stationary distributions of two Markov chains $P,Q$ (e.g., Glauber dynamics) respectively, and if $\mu$ is contractive w.r.t.\ $P$ and the two chains $P,Q$ are ``close'' to each other in one step, then the Wasserstein distance between $\nu$ and $\mu$ is small.
%The following lemma is the main tool in proving that a contractive coupling implies spectral independence;
The special case where $\Omega = \{+,-\}^n$ and $P,Q$ are both the Glauber dynamics appeared in \cite[Theorem 3.1]{BN} and~\cite[Theorem 2.1]{RR19}, but here we do not make any assumption on the state space or the chains, which is crucial to our applications in Section~\ref{subsec:coloring}.

\begin{lemma}\label{lem:general-stein}
Let $(\Omega, d)$ be a finite metric space.
Let $\mu,\nu$ be two distributions over $\Omega$,
and $P,Q$ be two Markov chains on $\Omega$ with stationary distributions $\mu,\nu$ respectively.
%Assume that $P$ is irreducible.
If $\mu$ is $\ctrct$-contractive w.r.t.\ the chain $P$ and the metric $d$,
then for every $f: \Omega \to \R$ we have
\[
|\bbE_\mu f - \bbE_\nu f| \le \frac{L_d(f)}{1-\ctrct} \, \bbE_\nu \left[ W_{1,d}(P(\sigma,\cdot), Q(\sigma,\cdot)) \right]
\]
where $P(\sigma,\cdot)$ is the distribution after one step of the chain $P$ when starting from $\sigma$ and similarly for $Q(\sigma,\cdot)$.
As a consequence,
\[
%\tv{\mu}{\nu} \le
W_{1,d}(\mu, \nu) \le \frac{1}{1-\ctrct} \, \bbE_\nu \left[ W_{1,d}(P(\sigma,\cdot), Q(\sigma,\cdot)) \right].
\]
\end{lemma}

We remark that Lemma~\ref{lem:general-stein} holds in a very general setting, and $(\Omega, d)$ can be any finite metric space.
It shows that if two Markov chains are close to each other, then their stationary distributions must be close to each other, under the assumption that one of the chains is contractive.
%says that, if two Markov chains are close to each other, then their stationary distributions must be close to each other, under the assumption that one of the chains is contractive.

\begin{proof}[Proof of Lemma~\ref{lem:general-stein}]
The proof imitates the arguments from \cite{BN,RR19}.
Assume for now that $P$ is irreducible; this is a conceptually easier case and we will consider general $P$ later.
Since $P$ is irreducible, let $h$ be the principal solution to the Poisson equation $(I-P)h=\bar f$ where $\bar{f} = f-\bbE_\mu f$; that is,
\begin{equation}\label{eqn:h}
h = \sum_{t=0}^\infty P^t\bar{f}.
\end{equation}
See Lemma 2.1 in \cite{BN} and the references in that paper for backgrounds on the Poisson equation.
We then have
\[
\bbE_\nu f - \bbE_\mu f = \bbE_\nu \bar{f} = \bbE_\nu [(I-P) h]=\bbE_\nu [(Q-P) h]
\]
where the last equality is due to $\nu = \nu Q$.
For each $\sigma \in \supp(\nu) \subset \Omega$, we deduce from \eqref{eq:W1-dual} that
\[
((Q - P) h) (\sigma) = \bbE_{Q(\sigma,\cdot)} h - \bbE_{P(\sigma,\cdot)} h \le L_d(h) \, W_{1,d}(Q(\sigma,\cdot), P(\sigma,\cdot)).
\]
It remains to bound the Lipschitz constant of $h$.
For $\sigma, \tau \in \Omega$,
\begin{align*}
|h(\sigma) - h(\tau)|
&\le \sum_{t=0}^\infty \left| (P^t\bar{f})(\si) - (P^t\bar{f})(\tau) \right| \\
&= \sum_{t=0}^\infty \left| \bbE_{P^t(\si,\cdot)} \bar{f} - \bbE_{P^t(\tau,\cdot)} \bar{f} \right| \\
&\le L_d(f) \sum_{t=0}^\infty W_{1,d}(P^t(\si,\cdot), P^t(\tau, \cdot))
\end{align*}
where the last inequality again follows from \eqref{eq:W1-dual}.
Since $\mu$ is $\ctrct$-contractive w.r.t.\ $P$ and $d$, for all $\sigma,\tau \in \Omega$ and every integer $t \ge 1$ we have
\[
W_{1,d}(P^t(\sigma,\cdot), P^t(\tau,\cdot)) \le \ctrct^t d(\si,\tau).
\]
We then deduce that
\[
|h(\sigma) - h(\tau)|
\le L_d(f) \sum_{t=0}^\infty \ctrct^t d(\si,\tau) = \frac{L_d(f)}{1-\ctrct} \, d(\si,\tau).
\]
This implies that $L_d(h) \le L_d(f)/(1-\ctrct)$ and the lemma then follows.

%\zongchen{New!}
Next, we show how to remove the assumption that $P$ is irreducible.
Observe that in the proof above we only need the irreducibility of $P$ to guarantee that the function $h$ given by \eqref{eqn:h} is well-defined; i.e., the series on the right-hand side of \eqref{eqn:h} is convergent. The rest of the proof does not require the irreducibility of $P$.
In fact, one can deduce the convergence of \eqref{eqn:h} solely from the contraction of $P$.
Note that for all $\sigma \in \Omega$,
\[
\left| P^t\bar{f}(\sigma) \right|
= \left| P^t\bar{f}(\sigma) - \bbE_\mu P^t\bar{f} \right|
= \left| P^t\bar{f}(\sigma) - \sum_{\tau \in \Omega} \mu(\tau) P^t\bar{f}(\tau) \right|
\le \sum_{\tau \in \Omega} \mu(\tau) \left| P^t\bar{f}(\sigma) - P^t\bar{f}(\tau) \right|
\]
where the first equality follows from $\bbE_\mu P^t\bar{f} = \bbE_\mu \bar{f} = 0$.
Since $\Omega$ is finite, to show that \eqref{eqn:h} is convergent for all $\sigma \in \Omega$, it suffices to show that for all $\sigma,\tau \in \Omega$ the series $\sum_{t=0}^\infty \left| P^t\bar{f}(\sigma) - P^t\bar{f}(\tau) \right|$ is convergent.
Actually, our proof before has already showed that
\[
\sum_{t=0}^\infty \left| P^t\bar{f}(\sigma) - P^t\bar{f}(\tau) \right| \le \frac{L_d(f)}{1-\ctrct} \, d(\si,\tau) < \infty
\]
using only the contraction of $P$, where we have $L_d(f) < \infty$ and $\sup_{\sigma,\tau \in \Omega} d(\si,\tau) < \infty $ because $\Omega$ is finite.
Therefore, the lemma remains true without the assumption of irreducibility of $P$.
\end{proof}

Given Lemma~\ref{lem:general-stein}, we can now complete the proof of Theorem~\ref{thm:glauber-hamming}.

\begin{proof}[Proof of Theorem~\ref{thm:glauber-hamming}]
For every $(\vx, \spa) \in \VSpairs$, we deduce from Lemma~\ref{lem:general-stein} that
\begin{equation}\label{eq:S_proof}
S(\vx,\spa) = \bbE_\nu f - \bbE_\mu f \le \frac{L_{d_{\mathrm{H}}}(f)}{1-\ctrct} \, \bbE_\nu \left[ W_{1,d_{\mathrm{H}}}(P(\sigma,\cdot), Q(\sigma,\cdot)) \right]
\end{equation}
where $S(\vx,\spa)$ is given by \eqref{eq:S_def}, $f$ is given by \eqref{eq:f_def},
$P$ is the Glauber dynamics for $\mu$,
%and $Q$ is the (modified) Glauber dynamics for $\nu = \mu(\cdot \mid \sigma_\vx = \spa)$ where the pinned vertex $\vx$ is allowed to be picked but no updates will be made when this happens.
and $Q$ is the Glauber dynamics for $\nu = \mu^{(\vx,\spa)} = \mu(\cdot \mid \sigma_\vx = \spa)$ (we use $(\vx,\spa)$ to denote the pinning $\sigma_\vx = \spa$).
We claim that for every $\sigma \in \Omega^{(\vx,\spa)}$,
\begin{equation}\label{eq:W_proof}
W_{1,d_{\mathrm{H}}}(P(\si,\cdot), Q(\si,\cdot)) \le \frac{1}{n}.
\end{equation}
To see this, let $\sigma_1$ and $\sigma_2$ be the configurations after one step of $P$ and $Q$ respectively when starting from $\sigma$.
We can couple $\sigma_1$ and $\sigma_2$ by picking the same vertex to update in the Glauber dynamics.
If the picked vertex is not $\vx$, then we can make $\sigma_1 = \sigma_2$; meanwhile, if $\vx$ is picked, which happens with probability $1/n$, then $d_{\mathrm{H}}(\si_1,\si_2) \le 1$ where the discrepancy is caused by the pinning $\sigma_\vx = \spa$.
Therefore, the $1$-Wasserstein distance between $\sigma_1$ and $\sigma_2$ is upper bounded by $1/n$; this justifies our claim.
%since we can couple the two configurations after one step of $P$ and $Q$ by picking the same vertex, and the Hamming distance becomes $1$ only when the site $x$ is picked because of the pinning, which happens with probability $1/n$.
Combining $L_{d_{\mathrm{H}}}(f) \le 2$ and \eqref{eq:W_proof}, we obtain from \eqref{eq:S_proof} that $S(\vx,\spa) \le \frac{2}{(1-\ctrct)n}$ for each $(\vx,\spa)$; consequently, $\lambda_1(J) \le \frac{2}{(1-\ctrct)n}$.
The same argument holds for $\mu^\tau$ under any pinning $\tau$ as well,
%(recall that the Glauber dynamics for $\mu^\tau$ is irreducible for any $\tau$ since we assume $\mu$ is totally connected),
and spectral independence then follows.
\end{proof}

\subsection{Contraction for Glauber dynamics and general metrics}
\label{subsec:glauber-wHamming}
In this section, we generalize the Hamming metric assumption in Theorem~\ref{thm:glauber-hamming} to any weighted Hamming metric or any metric equivalent to Hamming, which establishes Theorem~\ref{thm:glauber}.
We restate it here for convenience.

\glauber*

We prove the two cases of Theorem~\ref{thm:glauber} separately.
We first consider the weighted Hamming metric.
Recall that for a positive weight function $w : V \to \R_+$,
the $w$-weighted Hamming metric $d = d_w$ is given by
\[
d_w(\sigma, \tau) = \sum_{\vx \in V} w(\vx) \mathbf{1}\{\si_\vx \neq \tau_\vx\} ~\text{for}~ \sigma,\tau \in \Omega.
\]
%for two configurations $\sigma,\tau \in \Omega$.
In particular, if $w(\vx) = 1$ for all $\vx$ then $d$ is the usual Hamming metric.

Unfortunately, the proof of Theorem~\ref{thm:glauber-hamming} does not work directly in this scenario.
The reason is that the right-hand side of \eqref{eq:S_proof}, with $d_{\mathrm{H}}$ replaced by $d = d_w$ now, can be as large as $O(w_{\max}/w_{\min})$ (more specifically, $L_d(f) = O(1/w_{\min})$ and $W_{1,d}(P(\sigma,\cdot), Q(\sigma,\cdot)) = O(w_{\max})$), which can be unbounded since we are not making any assumption on $w$.
%The proof of Theorem~\ref{thm:glauber-hamming} does not work here since the function $f$ defined by \eqref{eq:f_def} is not $O(1)$-Lipschitz w.r.t.\ the $w$-weighted Hamming metric denoted by $d = d_w$.
To deal with this, we need to take the vertex weights into account when defining the function $f$ and, more importantly, when defining the absolute sum of influences $S(\vx,\spa)$.

\begin{proof}[Proof of Theorem~\ref{thm:glauber}(1)]
For ease of notation we may assume that there is no pinning; the proof remains the same with an arbitrary pinning $\tau$.
For fixed $(\vx,\spa) \in \VSpairs$, we define the $w$-weighted sum of absolute influences given by
\begin{equation}\label{eq:Sw_def}
S_w(\vx, \spa) = \sum_{(\vy, \spb) \in \VSpairs} w(\vy) \,| \ALOi(\vx, \spa; \vy, \spb) |.
\end{equation}
Such weighted sums were considered in \cite[Lemma 22]{CLV1} to deduce spectral independence.
We claim that
%\begin{lemma}\label{fact:w_sum_inf}
if $S_w(\vx,\spa) \le \eta\, w(\vx)$ for all $(\vx,\spa)\in\VSpairs$ for some $\eta > 0$, then $\lambda_1(J) \le \eta$.
%\end{lemma}
%\begin{proof}
To see this, let $\tilde{w} \in \R_+^{|\VSpairs|}$ with $\tilde{w}(\vx,\spa) = w(\vx)$ and let $W = \diag(\tilde{w})$;
the assumption of the claim then implies that $\|W^{-1} J W\|_\infty \le \eta$ and thus $\lambda_1(J) = \lambda_1(W^{-1} J W) \le \eta$. %\zongchen{cite CLV}
%\end{proof}
Therefore, it suffices to upper bound the ratio $S_w(\vx,\spa) / w(\vx)$.

Let $\nu = \mu^{(\vx,\spa)} = \mu(\cdot \mid \sigma_\vx = \spa)$ be the conditional distribution with pinning $\sigma_\vx = \spa$, and define
\begin{equation}\label{eq:fw_def}
f_w(\si) = \sum_{(\vy, \spb)\in \VSpairs} w(\vy) \, t(\vx, \spa; \vy, \spb) \, \mathbf{1}_{\{\sigma_\vy = \spb\}}
\end{equation}
where $t(\vx, \spa; \vy, \spb) = \sgn(J(\vx, \spa; \vy, \spb))$.
%$$
%t(\vx, \spa; \vy, \spb) = \sgn(J(\vx, \spa; \vy, \spb)) =
%\begin{cases}+1, & J(\vx, \spa; \vy, \spb) > 0;\\
%-1, & J(\vx, \spa; \vy, \spb) < 0;\\
%0, & J(\vx, \spa; \vy, \spb) = 0.
%\end{cases}
%$$
Observe that $L_d(f_w) \le 2$ and
\[
S_w(\vx, \spa) = \bbE_\nu f_w - \bbE_\mu f_w.
\]
It then follows from Lemma~\ref{lem:general-stein} that
\[
S_w(\vx,\spa) \le \frac{2}{1-\ctrct} \, \bbE_\nu \left[ W_{1,d}(P(\sigma,\cdot), Q(\sigma,\cdot)) \right]
\]
where $P,Q$ are the Glauber dynamics for $\mu,\nu$ respectively.
For every $\sigma \in \Omega^{(\vx,\spa)}$ we have
\[
W_{1,d}(P(\si,\cdot), Q(\si,\cdot)) \le \frac{w(\vx)}{n},
\]
since if we couple the configurations $\sigma_1,\sigma_2$ after one step of $P,Q$ respectively by picking the same vertex to update, then $d(\sigma_1,\sigma_2) = w(x)$ only when the site $\vx$ is picked, and $\sigma_1 = \sigma_2$ otherwise.
Therefore, we get $S_w(\vx,\spa) \le \frac{2 w(\vx)}{(1-\ctrct)n}$ for every $(\vx,\spa) \in \VSpairs$, implying that $\lambda_1(J) \le \frac{2}{(1-\ctrct)n}$.
The same argument works for $\mu^\tau$ under any pinning $\tau$ as well, which establishes spectral independence.
\end{proof}

Next we consider the second part of Theorem~\ref{thm:glauber}.
Recall that a metric $d$ on $\Omega$ is said to be $\gamma$-equivalent (to the Hamming metric) for some $\gamma > 1$ if for all $\sigma,\tau \in \Omega$
\[
\frac{1}{\gamma} \hamming{\sigma}{\tau} \le d(\sigma, \tau) \le \gamma \hamming{\sigma}{\tau}.
\]
To prove the second part, we follow the proof approach for Theorem~\ref{thm:glauber-hamming}, and in particular the right-hand side of \eqref{eq:S_2proof} below (analogous to \eqref{eq:S_proof}) can be upper bounded using the $\gamma$-equivalence.

\begin{proof}[Proof of Theorem~\ref{thm:glauber}(2)]
For every $(\vx, \spa) \in \VSpairs$, we deduce from Lemma~\ref{lem:general-stein} that
\begin{equation}\label{eq:S_2proof}
S(\vx,\spa) = \bbE_\nu f - \bbE_\mu f \le \frac{L_d(f)}{1-\ctrct} \, \bbE_\nu \left[ W_{1,d}(P(\sigma,\cdot), Q(\sigma,\cdot)) \right]
\end{equation}
where $S(\vx,\spa)$ and $f$ are defined by \eqref{eq:S_def}, \eqref{eq:f_def} respectively,
and $P, Q$ are the Glauber dynamics for $\mu$ and $\nu = \mu^{(\vx,\spa)} = \mu(\cdot \mid \sigma_\vx = \spa)$ respectively.
Since $d$ is $\gamma$-equivalent, for all $\sigma,\tau \in \Omega$ we have
\[
|f(\sigma) - f(\tau)| \le 2 \hamming{\sigma}{\tau} \le 2\gamma d(\sigma,\tau);
\]
this shows $L_d(f) \le 2\gamma$.
Meanwhile, by the definition of $1$-Wasserstein distance for every $\sigma \in \Omega^{(\vx,\spa)}$ we have
\begin{multline*}
W_{1,d}(P(\sigma,\cdot), Q(\sigma,\cdot))
= \inf \left\{ \bbE_\pi [ d(\si,\tau) ] \mid \pi \in \cC(P(\sigma,\cdot), Q(\sigma,\cdot)) \right\} \\
\le \gamma \inf \left\{ \bbE_\pi [ d_{\mathrm{H}}(\si,\tau) ] \mid \pi \in \cC(P(\sigma,\cdot), Q(\sigma,\cdot)) \right\}
= \gamma W_{1,d_{\mathrm{H}}}(P(\sigma,\cdot), Q(\sigma,\cdot)) \le \frac{\gamma}{n}
\end{multline*}
where the last inequality is \eqref{eq:W_proof}.
Thus, we obtain from \eqref{eq:S_2proof} that $S(\vx,\spa) \le \frac{2\gamma^2}{(1-\ctrct)n}$. The rest of the proof is the same as Theorem~\ref{thm:glauber-hamming}.
\end{proof}

\subsubsection{Application: Dobrushin uniqueness condition}
\label{subsec:dob-si}

%We can relax the Dobrushin uniqueness condition to Hayes' spectral condition.

%\begin{definition}[Dobrushin dependency matrix]
%The \emph{Dobrushin dependency/influence matrix} $R \in \R_{+}^{n \times n}$ is defined by $R(x,x) = 0$ and
%\[
%R(x,y) =
%\max \left\{ \tv{\mu_y(\cdot \mid \sigma)}{\mu_y(\cdot \mid \tau)}: (\si,\tau) \in \cS_{x,y} \right\} \text{~for~} x \neq y
%\]
%where $\cS_{x,y}$ is the set of all pairs of configurations on $V \setminus \{y\}$ that differ only at $x$.
%\end{definition}
%
%\begin{definition}[ALO influence matrix]
%The \emph{ALO influence matrix} $J \in \R^{qn \times qn}$ is defined by $J(x, \si_x; x, \si'_x) = 0$ and
%\[
%J(x, \si_x; y, \si_y) = \mu(\si_y\;{\rm at}\; y \mid \si_x\;{\rm at}\; x) - \mu(\si_y\;{\rm at}\; y) \text{~for~} x \neq y.
%\]
%\end{definition}

%For a square matrix $M$ we denote its maximum eigenvalue by $\lambda_1(M)$.
%The following theorem compares the maximum eigenvalues of these two matrices, and thus establishes spectral independence when $\lambda_1(R) < 1$.

As an application of Theorem~\ref{thm:glauber}(1), we show that the Dobrushin uniqueness condition, as well as its generalizations \cite{Hayes,DGJ09}, implies spectral independence.
Recall that the Dobrushin dependency matrix $R$ is a $|V| \times |V|$ matrix defined as $R(\vx,\vx) = 0$ and
\[
R(\vx,\vy) =
\max \left\{ \tv{\mu_\vy(\cdot \mid \sigma)}{\mu_\vy(\cdot \mid \tau)}: (\si,\tau) \in \cS_{\vx,\vy} \right\} \text{~for~} x \neq y
\]
where $\cS_{\vx,\vy}$ is the set of pairs of configurations on $V \setminus \{\vy\}$ that differ at most at $\vx$.
Denote the spectral radius of a square matrix $M$ by $\rho(M)$.
If $M$ is nonnegative, then $\rho(M)$ is an eigenvalue of $M$ by the Perron-Frobenius theorem.
%Utilizing Theorem~\ref{thm:glauber-general}, we can show that Hayes' spectral condition on the Dobrushin dependency matrix implies spectral independence.
We prove Theorem~\ref{thm:dep-inf} from the introduction. 

\dep*
%\begin{theorem}\label{thm:dep-inf}
%If the Dobrushin dependency matrix $R$ of $\mu$ satisfies $\rho(R) \le 1-\eps$ for some $\eps > 0$,
%then $\mu$ is spectrally independent with constant $\eta =2/\eps$.
%\end{theorem}
%\begingroup
%\def\thetheorem{\ref{thm:dep-inf}}
%\begin{theorem}
%If the Dobrushin dependency matrix $R$ satisfies $\rho(R) \le 1-\eps$ for some $\eps > 0$,
%then $\mu$ is spectrally independent with constant $\eta =2/\eps$.
%\end{theorem}
%\addtocounter{theorem}{-1}
%\endgroup

\begin{remark}\label{rmk:dob}
If $\|R\|_\infty < 1$, then the Glauber dynamics mixes rapidly by a simple application of the path coupling method of Bubley and Dyer \cite{BubleyDyer}.
The same is true under the Dobrushin uniqueness condition, i.e., when $\|R\|_1 < 1$.
Hayes \cite{Hayes} generalized the condition to the spectral norm $\|R\|_2 < 1$.
Dyer, Goldberg, and Jerrum \cite{DGJ09} further improved it to $\|R\| < 1$ for any matrix norm (where the mixing time depends logarithmly on the norm of the all-one matrix).
Our condition $\rho(R) < 1$ in Theorem~\ref{thm:dep-inf} is technically better than previous works since for a nonnegative matrix $R$ one has $\rho(R) \le \|R\|$ for any matrix norm, and the inequality can be strict for all norms when $R$ is not irreducible; see \cite{DGJ09} for related discussions.
Finally, we point out that if $R$ is symmetric then $\rho(R) = \|R\|_2$.
\end{remark}

It is known that the Glauber dynamics is contractive for some weighted Hamming metric if the weight vector satisfies a spectral condition related to $R$.

\begin{lemma}[{\cite[Lemma 20]{DGJ09}}]
\label{lem:hayes}
%Let $w$ be the right eigenvector associated with the eigenvalue $\lambda_1(R)$; namely, $R w = \lambda_1(R) w$.
%If $R$ is irreducible and $\lambda_1(R) \le 1-\eps$ for some $\eps > 0$,
If $w \in \R_+^V$ is a positive vector such that $R w \le (1 -\eps) w$ entrywisely,
then $\mu$ is $(1-\eps/n)$-contractive w.r.t.\ the Glauber dynamics and the $w$-weighted Hamming metric $d = d_w$.
\end{lemma}

%For a fixed pair $(x,a)\in\cX$, define
%\begin{equation}
%%\label{eq:infl1}
%S_w(x,a)=\sum_{(y,a')\in\cX} w_y |J(x,a;y,a')|.
%\end{equation}

%Define $\nu=\mu(\cdot\tc\si_x \;{\rm at}\; x)$, and
%$$
%f(\si)=\sum_{(y,a')\in\cX} a(y,a') w_y {\bf 1}(\si_y\;{\rm at}\; y)\,,
%$$
%where $$
%a(y,a')=\begin{cases}
%+1 & J(x,a;y,a')>0\\
%-1 & J(x,a;y,a')<0\\
%0 & J(x,a;y,a')=0
%\end{cases}
%$$
%With these definitions it follows that
%\begin{equation}
%\label{eq:w_sum_inf}
%S_w(x,a)=\bbE_\nu f - \bbE_\mu f.
%\end{equation}
%Observe that $L_{d_w}(f) \le 2$.

The following fact about nonnegative matrices is helpful.
%Note that the lemma is stated for spectral radius in \cite{Meyer}, but for nonnegative matrices $\rho(M) = \lambda_1(M)$ due to the Perron-Frobenius theorem.
\begin{lemma}[{\cite[Example 7.10.2]{Meyer}}]
\label{lem:Meyer}
If $M,N \in \R_{+}^{n \times n}$ are two nonnegative square matrices such that $M \le N$ entrywisely, then $\rho(M) \le \rho(N)$.
\end{lemma}

We give below the proof of Theorem~\ref{thm:dep-inf}.

\begin{proof}[Proof of Theorem~\ref{thm:dep-inf}]
Consider first the case that there is no pinning.
If the Dobrushin dependency matrix $R$ is irreducible, then the right principal eigenvector $w$ associated with the eigenvalue $\rho(R)$ satisfies $R w = \rho(R) w \le (1-\eps) w$ and $w > 0$ by the Perron-Frobenius theorem.
Hence, Lemma~\ref{lem:hayes} and (the proof of) Theorem~\ref{thm:glauber}(1) immediately yield $\lambda_1(J) \le 2/\eps$.
However, if $R$ is reducible, we cannot use the principal eigenvector directly since it may have zero entries.
We instead consider the matrix $R_\delta = R + \delta O$ where $O$ is the all-one matrix and $\delta>0$ is a tiny constant.
Let $w_\delta$ be the right principal eigenvector of $R_\delta$ associated with the eigenvalue $\rho(R_\delta)$.
Since $R_\delta$ is irreducible, $w_\delta > 0$ by the Perron-Frobenius theorem.
Moreover, $R w_\delta \le R_\delta w_\delta = \rho(R_\delta) w_\delta$.
Since $\lim_{\delta \to 0} R_\delta = R$, we have $\lim_{\delta \to 0} \rho(R_\delta) = \rho(R)$; see, e.g., Remark 3.4 in \cite{Alen}.
%Observe that $\lambda_1(R_\delta) \to \lambda_1(R)$ as $\delta \to 0$, and
Thus, $\rho(R_\delta) < 1$ for sufficiently small $\delta$.
Then by Lemma~\ref{lem:hayes} and Theorem~\ref{thm:glauber}(1), for $\delta$ small enough, we have $\lambda_1(J) \le 2/(1-\rho(R_\delta))$.
Taking $\delta \to 0$ and using the assumption that $\rho(R) \le 1- \eps$, we obtain $\lambda_1(J) \le 2/\eps$.

Next, consider the conditional measure $\mu^\tau$ with a pinning $\tau$ on a subset $U \subset V$.
Let $R^\tau$ be the Dobrushin dependency matrix for $\mu^\tau$;
note that by definition $R^\tau(\vx, \vy) = 0$ if $\vx \in U$ or $\vy \in U$, and $R^\tau(\vx,\vy) \le R(\vx, \vy)$ for all $\vx, \vy \in V$.
We deduce from Lemma~\ref{lem:Meyer} that $\rho(R^\tau) \le \rho(R) \le 1- \eps$ and thus this is reduced to the no-pinning case.
Therefore, we get $\lambda_1(J^\tau) \le 2/\eps$ for all $\tau$ and spectral independence then follows.
\end{proof}
%note that $R^\tau$ is a $|V \setminus U| \times |V \setminus U|$ matrix.
%Let $R_0$ be the submatrix of $R$ indexed by $V \setminus U$; so $R^\tau$ and $R_0$ have the same dimensions.
%If $R$ is irreducible then the theorem follows immediately from Theorem~\ref{thm:glauber-general} and Lemma~\ref{lem:hayes}.
%If $R$ is reducible, we may consider $\tilde{R}(t) = R + t O$ where $O$ is the all-one matrix and $t>0$ is a small constant.
%Applying Theorem~\ref{thm:glauber-general} and Lemma~\ref{lem:hayes}, we get $\eta \le \frac{2}{1-\lambda_1(\tilde{R}(t))}$ for all sufficiently small $t$.
%The theorem then follows by taking $t \to 0$.
%We apply Theorem~\ref{lem:general-stein}.
%The metric is $d = d_w$.
%The chain $P$ is the Glauber dynamics for $\mu$.
%We define $Q$ to be the (modified) Glauber dynamics for $\nu = \mu(\cdot\tc\si_x = i)$, where the pinned vertex $x$ is allowed to be picked but no updates will be made when this happens.
%Then, for every $\sigma \in \supp(\nu) \subset \Omega$ we have
%\[
%W_{1,d_w}(P(\si,\cdot), Q(\si,\cdot)) \le \frac{w_x}{n}
%\]
%since the weighted Hamming distance increases to $w_x$ only when the site $x$ is picked.
%Combining $L_{d_w}(f) \le 2$, we deduce from \eqref{eq:w_sum_inf} and Lemma~\ref{lem:general-stein} that
%\[
%S_w(x,a) \le \frac{2}{\alpha} \cdot \frac{w_x}{n}.
%%\le \frac{2w_x}{a}.
%\]
%The theorem then follows by Lemmas~\ref{lem:hayes} and \ref{lem:w_sum_inf}.
%\TODO

\subsection{Contraction for general Markov chains and general metrics}
\label{subsec:general-general}

In this section, we generalize Theorem~\ref{thm:glauber-hamming} to arbitrary ``local'' Markov chains and arbitrary metrics close to the Hamming metric.
In particular, we prove Theorem~\ref{thm:general}.

%This is given by the following theorem, which implies the second part of Theorem~\ref{thm:glauber} as well as Theorem~\ref{thm:general}.
%Recall that a metric $d$ on $\Omega$ is said to be $\gamma$-equivalent for some $\gamma \in (0,1]$ if $\gamma \hamming{\sigma}{\tau} \le d(\sigma, \tau) \le \hamming{\sigma}{\tau}$ for all $\sigma,\tau \in \Omega$.
%\[
%\gamma \hamming{\sigma}{\tau} \le d(\sigma, \tau) \le \hamming{\sigma}{\tau}.
%\]

%Given a distribution $\mu$, let $\Pinning = \Pinning(\tau)$ denote the collection of all (feasible) pinnings of $\mu$.
%We will often consider a collection $\mathcal{P} = \mathcal{P}(\mu) = \{P^\tau: \tau \in \Pinning\}$ of Markov chains associated with $\mu$, where each $P^\tau$ is a Markov chain with stationary distribution $\mu^\tau$; for example, $\mathcal{P}$ can be the collection of Glauber dynamics for all $\mu^\tau$'s.
Consider a collection of Markov chains $\mathcal{P} = \{P^\tau: \tau \in \Pinning\}$ associated with $\mu$, where each $P^\tau$ is a Markov chain on $\Omega^\tau$ with stationary distribution $\mu^\tau$.
Intuitively, one can think of $\mathcal{P}$ as the same dynamics applied to all conditional distributions $\mu^\tau$; for example, $\mathcal{P}$ can be the collection of Glauber dynamics for all $\mu^\tau$'s.
We are particularly interested in local dynamics; these are Markov chains that make local updates on the configuration in each step, e.g., Glauber dynamics for spin systems or flip dynamics for colorings.
Alternatively, we can describe local dynamics as those insensitive to pinnings; that is, if the dynamics is applied to both $\mu$ and $\mu^{(\vx,\spa)}$ with a pinning $\sigma_\vx = \spa$, then with high probability there is no difference in the two chains or the discrepancy caused by the pinning will not propagate.
This motivates the following definition. %\zongchen{New!}
%We say a collection $\mathcal{P}$ of Markov chains associated with $\mu$

\begin{definition}\label{def:local-chains}
We say a collection $\mathcal{P}$ of Markov chains associated with $\mu$ is \emph{$\loc$-local} if for any two adjacent pinnings $\tau \in \Pinning$ and $\tau' = \tau \cup (\vx,\spa)$ where $(\vx,\spa) \in \VSpairs^\tau$ (i.e., $\tau'$ combines $\tau$ and the pinning $\sigma_\vx = \spa$), and for all $\sigma \in \Omega^{\tau'}$, we have
\[
W_{1,d_{\mathrm{H}}} (P^\tau(\sigma,\cdot), P^{\tau'} (\sigma,\cdot)) \le \loc.
\]
\end{definition}

We show that for such local dynamics contraction implies spectral independence.
\begin{theorem}\label{thm:general-general}
If $\mu$ is $\ctrct$-contractive w.r.t.\ a $\loc$-local collection $\mathcal{P}$ of Markov chains and a $\gamma$-equivalent metric $d$ for some $\ctrct\in(0,1)$,
then $\mu$ is spectrally independent with constant $\eta =\frac{2\gamma^2\loc}{1-\ctrct}$.
%In particular, if $\a\le 1- \eps/n$, then $\eta \leq 2/\eps$.
\end{theorem}

\begin{proof}%[Proof of Theorem~\ref{thm:general-general}]
The proof is similar to that of Theorems~\ref{thm:glauber-hamming} and \ref{thm:glauber}(2).
For an arbitrary pinning $\tau$ and $(\vx,\spa) \in \VSpairs^\tau$, we define
\begin{equation}%\label{eq:S_def}
S^\tau(\vx, \spa) = \sum_{(\vy, \spb) \in \VSpairs^\tau} | \ALOi^\tau(\vx, \spa; \vy, \spb) |
%= \sum_{(\vy, \spb) \in \VSpairs} | \mu(\sigma_\vy = \spb \mid \sigma_\vx = \spa) - \mu(\sigma_\vy = \spb) |
\end{equation}
and
\begin{equation}%\label{eq:f_def}
f^\tau(\si) = \sum_{(\vy, \spb)\in \VSpairs^\tau} t^\tau(\vx, \spa; \vy, \spb) \, \mathbf{1}_{\{\sigma_\vy = \spb\}}
\end{equation}
where $t^\tau(\vx, \spa; \vy, \spb) = \sgn(J^\tau(\vx, \spa; \vy, \spb))$;
these definitions are analogous to \eqref{eq:S_def} and \eqref{eq:f_def} with pinning $\tau$.
Let $\tau' = \tau \cup (\vx,\spa)$.
Then we deduce from Lemma~\ref{lem:general-stein} that
\[
S^\tau(\vx,\spa) = \bbE_{\mu^{\tau'}} f^\tau - \bbE_{\mu^\tau} f^\tau \le \frac{L_d(f^\tau)}{1-\ctrct} \, \bbE_{\mu^{\tau'}} \left[ W_{1,d}(P^\tau(\sigma,\cdot), P^{\tau'}(\sigma,\cdot)) \right].
\]
As shown in the proof of Theorem~\ref{thm:glauber}(2), since $d$ is $\gamma$-equivalent to the Hamming metric we have $L_d(f^\tau) \le \gamma L_{d_{\mathrm{H}}}(f^\tau) \le 2\gamma$ and
%For any $\sigma,\sigma' \in \Omega^\tau$, we have
%\[
%|f^\tau(\sigma) - f^\tau(\sigma')| \le 2 d_{\mathrm{H}}(\sigma,\sigma') \le \frac{2}{\gamma} d(\sigma,\sigma');
%\]
%hence, $L_d(f^\tau) \le 2/\gamma$.
for all $\sigma \in \Omega^{\tau'}$ we have
%Also, by $\loc$-locality of $\mathcal{P}$ we have for each $\sigma \in \Omega^\tau$
\[
W_{1,d}(P^\tau(\sigma,\cdot), P^{\tau'}(\sigma,\cdot)) \le \gamma W_{1,d_{\mathrm{H}}}(P^\tau(\sigma,\cdot), P^{\tau'}(\sigma,\cdot)) \le \gamma \loc
\]
using the $\loc$-locality of $\mathcal{P}$.
Therefore, we obtain that $S^\tau(\vx,\spa) \le \frac{2\gamma^2\loc}{1-\ctrct}$ for all $(\vx,\spa) \in \VSpairs^\tau$.
This yields $\lambda_1(J^\tau) \le \frac{2\gamma^2\loc}{1-\ctrct}$ and spectral independence follows.
\end{proof}

%With a similar proof approach we can prove Theorem~\ref{thm:block} for the analog for block-type dynamics.
%\eric{State the result for the block-type dynamics here?}
To better understand local dynamics, we consider a very general type of Markov chains which we call \emph{select-update dynamics};
examples include the Glauber dynamics, heat-bath block dynamics, and flip dynamics.
Let $\mathcal{B}$ be a collection of blocks associated with the select-update dynamics
and fix some pinning $\tau$.
Given the current configuration $\sigma^t \in \Omega^\tau$, the next configuration $\sigma^{t+1}$ is generated as follows:
\begin{enumerate}
\item \textsc{Select}: Select a block $B \in \mathcal{B}$ from some distribution $\select_t$ over $\mathcal{B}$;
\item \textsc{Update}: Resample the configuration on $B$ from some distribution $\nu^t_B$.
\end{enumerate}
We try to make weakest assumptions on the selection rule $\select_t$ and the update rule $\nu^t_B$: the selection distribution $\select_t$ is allowed to depend on the current configuration $\sigma^t$ but is independent of the pinning $\tau$, and the update distribution $\nu^t_B$ is allowed to depend on the whole current configuration $\sigma^t$ and the part of the pinning $\tau$ contained in $B$.
In particular, the heat-bath block dynamics is a special case of the select-update dynamics: the selection rule $\select_t = \alpha$ is a fixed distribution over $\mathcal{B}$ and the update rule $\nu^t_B$ is the marginal distribution on $B$ conditioned on $\sigma^t$ outside $B$ and the pinning $\tau$ in $B$.

\begin{remark}\label{rmk:pinning-SU}
The assumption that the selection rule $\select_t$ is independent of the pinning $\tau$ is not necessary, but it is helpful for stating and proving our theorems and does not weaken our results.
Roughly speaking, we only require that the collection of the select-update dynamics is the same dynamics applied to all $\mu^\tau$'s, and the selection rule $\select_t$ can be conditioned on containing at least one unpinned vertex.
See the discussions in Remark~\ref{rmk:glauber-pinning} for the Glauber dynamics.
\end{remark}

%For a given pinning $\tau$, the block-type dynamics in each step samples a block $B \in \mathcal{B}$ from the distribution $\alpha$, and updates every connected component of the induced subgraph $G[B]$ independently conditioned on the current configuration outside $B$ and the pinning $\tau$.
%If the chain update each component by the corresponding conditional marginal distribution on this component, then it is the usual heat-bath block dynamics.
%Here we allow the chain to make arbitrary moves in each component; we only require that the updates among different components are independent.
%An important example of such block-type dynamics is the flip dynamics for colorings.
We write $\mathcal{P}_{\mathcal{B}}$ for a collection of select-update dynamics associated with $\mu$.
Denote the maximum block size of $\mathcal{B}$ by
\[
M %= M(\mathcal{B})
= \max_{B \in \mathcal{B}} |B|,
\]
and the maximum probability of a vertex being selected in Step 1 by
\[
D %= D(\mathcal{B})
= \max_{\select_t} \max_{\vx} \sum_{B \in \mathcal{B}: \vx \in B} \select_t(B)
\]
where we maximize over all selection rules $\select_t$ that can occur. 
We can show that the select-update dynamics $\mathcal{P}_{\mathcal{B}}$ is $\loc$-local with $\loc = DM$; using this and Theorem~\ref{thm:general-general} we establish Theorem~\ref{thm:general}, which we restate here for convenience.

\general*
%\begingroup
%\def\thetheorem{\ref{thm:general}}
%\begin{theorem}%\label{coro:block-general}
%If $\mu$ is $\ctrct$-contractive w.r.t.\ arbitrary select-update dynamics and an arbitrary $\gamma$-equivalent metric,
%then $\mu$ is spectrally independent with constant $\eta =\frac{2\gamma^2DM}{1-\ctrct}$.
%\end{theorem}
%\addtocounter{theorem}{-1}
%\endgroup

\begin{proof}%[Proof of Theorem~\ref{thm:general}]
It suffices to show that the select-update dynamics $\mathcal{P}_{\mathcal{B}}$ is $\loc$-local with $\loc = DM$; the theorem would then follows immediately from Theorem~\ref{thm:general-general}.
%We claim that the select-update dynamics $\mathcal{P}_{\mathcal{B}}$ is $\loc$-local with $\loc = DM$.
Consider two adjacent pinnings $\tau$ and $\tau' = \tau \cup (\vx,\spa)$ where $(\vx,\spa) \in \VSpairs^\tau$.
For $\sigma \in \Omega^{\tau'}$, let $\sigma_1$ and $\sigma_2$ be the two configurations obtained from $\sigma$ after one step of $P^\tau$ and $P^{\tau'}$ respectively.
We couple $\sigma_1$ and $\sigma_2$ by picking the same block $B \in \mathcal{B}$ in Step 1 of the select-update dynamics.
If $\vx \notin B$, then we have $\sigma_1 = \sigma_2$.
Meanwhile, if $\vx \in B$, which happens with probability at most $D$, we have $\hamming{\sigma_1}{\sigma_2} \le |B| \le M$.
Therefore,
\[
W_{1,d_{\mathrm{H}}} (P^\tau(\sigma,\cdot), P^{\tau'} (\sigma,\cdot)) \le DM.
\]
%Indeed, for two adjacent pinning $\tau = (U,\tau_U)$ and $\tau' = (U\cup \vx, \tau_U \cup \tau_\vx)$ and $\sigma \in \Omega^{\tau'}$,
%the two configurations obtained from $\sigma$ after one step of $P^\tau$ and $P^{\tau'}$ respectively do not couple only when the block $B$ selected contains $\vx$, which happens with probability at most $D$, and when this happens the discrepancies between the two configurations is at most $|B| \le M$.
This establishes the $(DM)$-locality for $\mathcal{P}_{\mathcal{B}}$.
%, and thus proves Theorem~\ref{thm:general} which we restate here for convenience.
%the Hamming metric increases to at most $M$ when a block containing $\vx$ is selected, which happens with probability at most $D$.
%We arrive at the following corollary.
\end{proof}

\begin{remark}\label{rmk:block-general}
If we further assume that in Step 2 the select-update dynamics resamples a block independently for each of its components (i.e., the update rule $\nu^t_B$ is a product distribution over all components of the induced subgraph $G[B]$), then in Theorem~\ref{thm:general} the maximum block size $M$ can be replaced by the maximum component size of all blocks.
\end{remark}

%We also present here the proof of Theorem~\ref{thm:glauber}.
%
%\begin{proof}[Proof of Theorem~\ref{thm:glauber}]
%The first part is given by Theorem~\ref{thm:glauber-general}.
%The second part follows from Theorem~\ref{thm:general} once we notice that $M = 1$ and $D = 1/n$ for the Glauber dynamics.
%\end{proof}

%We note that Theorems~\ref{thm:glauber} and \ref{thm:general} follows immediately from Theorem~\ref{thm:glauber-general} and Corollary~\ref{coro:block-general}.

\subsubsection{Application: flip dynamics for colorings}
\label{subsec:coloring}
%\subsubsection{Application: colorings and flip dynamics \TODO}
%\label{sec:flip}
In this section we establish spectral independence for colorings utilizing Theorem~\ref{thm:general}.
%and hence obtain Theorem~\ref{thm:colorings}.

\begin{theorem}\label{thm:coloring-SI}
Let $\eps_0\approx 10^{-5}>0$ be a fixed constant.
Let $\Delta,q \ge 3$ be integers and $q > (\frac{11}{6}-\eps_0) \Delta$.
Then there exists $\eta = \eta(\Delta,q)>0$ such that the following holds.

Let $\mu$ be the uniform distribution over all proper $q$-colorings of a graph $G=(V,E)$ of maximum degree at most $\Delta$.
Then $\mu$ is spectrally independent with constant $\eta$.
\end{theorem}

%Recall that the Hamming metric of two configurations $\sigma,\tau \in \Omega$ is $\hamming{\sigma}{\tau} = | \{x \in V: \sigma_x \neq \tau_x\} |$.
%For $0 < \gamma \le 1$, a metric $d$ on $\Omega$ is said to be $\gamma$-Hamming-equivalent if for all $\sigma,\tau \in \Omega$,
%\[
%\gamma \hamming{\sigma}{\tau} \le d(\sigma, \tau) \le \hamming{\sigma}{\tau}.
%\]

%Let $\Lambda \subset V$ and $\si_\Lambda \in [q]^\Lambda$ such that $\mu(\si_\Lambda) > 0$.
%We consider the conditional Gibbs distribution $\mu(\cdot \tc \si_\Lambda)$.

To apply Theorem~\ref{thm:general}, we need a contractive Markov chain for sampling colorings of a graph.
Vigoda considered the \emph{flip dynamics} \cite{Vig00} and showed that it is contractive for the Hamming metric when the number of colors $q > \frac{11}{6}\Delta$.
Recently, \cite{CDMPP19} improved the bound to $q > (\frac{11}{6} - \eps_0) \Delta$ for a fixed tiny constant $\eps_0 \approx 10^{-5}$, using variable-length coupling or an alternative metric.
Our result on spectral independence builds upon contraction results for the flip dynamics.

%We consider the \emph{flip dynamics} from \cite{Vig00}, which is described as follows.
We first describe the flip dynamics.
Let $\Omega$ be the set of all proper $q$-colorings of $G$.
Fix a pinning $\tau$ on $U \subset V$.
For a coloring $\si \in \Omega$, a vertex $\vx \in V$, and a color $\spa \in [q]$, denote by $L_\si(\vx,\spa)$ the bicolored component containing $\vx$ with colors $\spa$ and $\si_\vx$; that is, the set of all vertices which can be reached from $\vx$ through an alternating $(\si_\vx,\spa)$-colored path.
Given the coloring $\si^t$ at time $t$, the flip dynamics with pinning $\tau$ generates the next coloring $\si^{t+1}$ as follows:
\begin{enumerate}
\item Pick a vertex $\vx \in V$ u.a.r.\ and a color $\spa \in [q]$ u.a.r.;
\item If $L_{\si^t}(\vx, \spa)$ contains a pinned vertex (i.e., $L_{\si^t}(\vx, \spa) \cap U \neq \emptyset$), then $\si^{t+1} = \si^t$;
\item If all vertices in $L_{\si^t}(\vx, \spa)$ are free (i.e., $L_{\si^t}(\vx, \spa) \cap U = \emptyset$), then flip the two colors of $L_{\si^t}(\vx, \spa)$ with probability $p_s/s$ where $s = |L_{\si^t}(\vx, \spa)|$.
\end{enumerate}

The flip dynamics is specified by the flip parameters $\{p_s\}_{s=1}^\infty$. In \cite{Vig00} and the recent improvement \cite{CDMPP19}, the flip parameters are chosen in such a way that $p_s = 0$ for all $s \ge 7$; i.e., in each step at most six vertices change their colors.
We set the flip parameters as in Observation 5.1 from \cite{CDMPP19}, where the authors established contraction of the flip dynamics using the path coupling method.

\begin{lemma}[\cite{CDMPP19}]
\label{lem:flip_dyn}
%The flip dynamics is $(1-\Omega(1/n))$-contractive w.r.t.\ some metric $d$ that is $\frac{1}{2}$-equivalent to the Hamming metric.
%Let $\Delta,q \ge 3$ be integers and $q > (\frac{11}{6}-\eps_0) \Delta$.
%Let $\mu$ be the uniform distribution over all proper $q$-colorings of a graph $G=(V,E)$ of maximum degree at most $\Delta$.
%Then
Under the assumptions of Theorem~\ref{thm:coloring-SI},
there exists a constant $\eps = \eps(\Delta,q) > 0$ and a $2$-equivalent metric $d$ such that $\mu$ is $(1-\eps/n)$-contractive w.r.t.\ the flip dynamics and the metric $d$.
\end{lemma}

%Let $P$ denote the flip dynamics for $\mu(\cdot \tc \si_\Lambda)$.
%Let $\nu=\mu(\cdot\tc \si_\Lambda, \si_x)$.
%We define $Q$ to be the (modified) flip dynamics with stationary distribution $\nu$.
%Given the coloring $\si^t$ at time $t$, the next coloring $\si^{t+1}$ is generated by:
%\begin{enumerate}
%\item Pick a vertex $z \in V \setminus \Lambda$ u.a.r.\ (the pinned vertex $x$ is allowed to be picked) and a color $c \in [q]$ u.a.r.;
%\item If $S_{\si^t}(z, c) \cap (\Lambda \cup \{x\}) \neq \emptyset$ (i.e., $S_{\si^t}(z, c)$ contains pinned vertices), then $\si^{t+1} = \si^t$;
%\item If $S_{\si^t}(z, c) \cap (\Lambda \cup \{x\}) = \emptyset$ (i.e., all vertices in $S_{\si^t}(z, c)$ are free), then flip the two colors of $S_{\si^t}(z, c)$ with probability $p_s/s$ where $s = |S_{\si^t}(z, c)|$.
%\end{enumerate}

We remark that the pinning $\tau$ induces a list coloring instance where each unpinned vertex has a color list to choose its color from, and the results of \cite{CDMPP19} generalize naturally to list colorings.
Also, in this paper we assume that the flip dynamics may pick a pinned vertex and stay at the current coloring.
This does not weaken our results since we only consider the flip dynamics for analysis rather than actually running it;
see Remark~\ref{rmk:glauber-pinning} addressing the same issue for the Glauber dynamics and also Remark~\ref{rmk:pinning-SU} for general select-update dynamics.

We give below the proof of Theorem~\ref{thm:coloring-SI}.

\begin{proof}[Proof of Theorem~\ref{thm:coloring-SI}]
Observe that the flip dynamics belongs to the class of select-update dynamics, where the associated $\mathcal{B}$ is the collection of connected subsets of vertices.
Since the flip parameters satisfy $p_s > 0$ only for $s \le 6$, we have $M \le 6$.
Moreover, we have $D \le \Delta^6/n$ since a vertex $\vx$ is in the selected bicolored component $L_{\si^t}(\vy, \spa)$ only if $\dist(\vx,\vy) \le 5$, which happens with probability at most $\Delta^6/n$.
The theorem then follows from Lemma~\ref{lem:flip_dyn} and Theorem~\ref{thm:general}.
%Then, for every $\sigma \in \supp(\nu) \subset \Omega$ we have
%\[
%W_{1,d}(P(\si,\cdot), Q(\si,\cdot)) \le W_{1,d_{\mathrm{H}}}(P(\si,\cdot), Q(\si,\cdot)) = O\left( \frac{1}{n} \right)
%\]
%since the Hamming distance increases only when $x \in S_\si(z, c)$ and $|S_\si(z, c)| \le 6$, which happens with probability $O(1/n)$, and it increases to at most $6$.
%Also, we have
%\[
%L_d(f) \le \frac{L_{d_{\mathrm{H}}}(f)}{\gamma} \le \frac{2}{\gamma} \le 4.
%\]
%We then obtain the theorem from Lemmas~\ref{lem:general-stein} and \ref{lem:flip_dyn}.
\end{proof}

We conclude here with the proof of Theorem~\ref{thm:colorings}.

\begin{proof}[Proof of Theorem~\ref{thm:colorings}]
By Theorem~\ref{thm:coloring-SI} the uniform distribution $\mu$ of proper colorings is spectrally independent.
Then the results follows immediately from Theorem~\ref{thm:block-factorization}.
\end{proof}

\subsubsection{Application: block dynamics for Potts model}

Here we apply Theorems~\ref{thm:glauber-hamming} and \ref{thm:general} to the ferromagnetic Potts model to establish spectral independence. 
%We emphasize that this is only one example of application of our approach. Further examples can be obtained by using spatial mixing conditions for spin systems on the integer; see, e.g.\ \cite{DS85,MO94}. 

\begin{theorem}
	\label{thm:Potts-SI}
	Let $\Delta \ge 3$ and $q \ge 2$ be integers. 
	Let $\mu$ be the Gibbs distribution of the $q$-state ferromagnetic Potts model with inverse temperature parameter $\beta$ on a graph $G=(V,E)$ of maximum degree at most $\Delta$. Then, the following holds:
	\begin{enumerate}
		\item If $\beta < \max\left\{ \frac 2\Delta,\frac{1}{\Delta}\ln(\frac{q-1}{\Delta}) \right\}$, then $\mu$ is spectrally independent with constant $\eta = \eta(\beta, \Delta)$.
		\item For any $\delta > 0$ there exists $c = c(\delta,\Delta) > 0$ such that, if $\beta \le \frac{\ln q - c}{\Delta-1+\delta}$ then $\mu$ is spectrally independent with constant $\eta = \eta(\delta,\beta, \Delta)$.	
	\end{enumerate} 
\end{theorem}

To prove this theorem, we need the following results from \cite{Ullrich} and \cite{BGP} regarding the contraction of the Glauber dynamics and of the heat-bath block dynamics with a specific choice of blocks. 
%~\cite[Corollary 2.14]{Ullrich} (see \cite[Observation 11]{Hayes}
\begin{lemma}[{\cite[Corollary 2.14]{Ullrich} \& \cite[Proposition 2.2]{BGP}}]
	\label{lem:potts-gd}
	Under the assumptions in Part 1 of Theorem~\ref{thm:Potts-SI}, 
	there exists a constant $\eps = \eps(\beta,\Delta)$ such that
	$\mu$ is $(1-\frac{\eps}{n})$-contractive w.r.t.\ the Glauber dynamics and the Hamming metric.
\end{lemma}

\begin{lemma}[{\cite[Theorem 2.7]{BGP}}]
	\label{lem:potts-block}
	%Under the assumptions of Theorem~\ref{thm:Potts-SI}, there exists a (multi)set $\mathcal{B} = \{B_\vx: \vx \in V\}$ with uniform distribution satisfying $\vx \in B_\vx$, $|B_\vx| = O(1/\delta)$, and $G[B_\vx]$ is connected for all $\vx$, such that $\mu$ is $(1-\frac{1}{2n})$-contractive w.r.t.\ the heat-bath block dynamics with $\mathcal{B}$ and the Hamming metric. 
	Under the assumptions in Part 2 of Theorem~\ref{thm:Potts-SI}, there exists a collection of blocks $\mathcal{B} = \{B_\vx\}_{\vx \in V}$ satisfying $\vx \in B_\vx$, $|B_\vx| = O(1/\delta)$ and $G[B_\vx]$ connected for all $\vx$, such that
	$\mu$ is $(1-\frac{1}{2n})$-contractive w.r.t.\ the $\alpha$-weighted heat-bath block dynamics for $\mathcal{B}$ and the Hamming metric, where $\alpha$ is the uniform distribution over $\mathcal B$.
\end{lemma}

\begin{remark}\label{rmk:BGP}
To be more precise, \cite{BGP} shows that the conclusion of Lemma~\ref{lem:potts-block} is true when $\beta$, $q$, and the maximum block size $M = \max_{\vx \in V} |B_\vx|$ satisfy
\begin{equation}\label{eq:cond_BGP}
\beta \left( \Delta-1+\frac{1}{M} \right) + 3M (\ln \Delta + \ln M) \le \ln q. 
\end{equation}
Thus, for any $\delta > 0$, by taking $M = \lceil\delta^{-1}\rceil$ and $c = 3M (\ln \Delta + \ln M)$, our assumption $\beta \le \frac{\ln q - c}{\Delta-1+\delta}$ in Part 2 of Theorem~\ref{thm:Potts-SI} implies \eqref{eq:cond_BGP}. 
Moreover, if we take, say, $M \approx \sqrt{\ln q}$ (namely, $\delta \approx 1/\sqrt{\ln q}$), then $c = o(\ln q)$ and our assumption becomes $\beta \le (1-o(1))\frac{\ln q}{\Delta -1}$ where $o(1)$ tends to $0$ as $q \to \infty$; this gives the bound $\beta_1$ in Theorem~\ref{thm:Ising-Potts-mixing} from the introduction. 
\end{remark}

Theorem~\ref{thm:Potts-SI} is an immediate consequence of Lemmas~\ref{lem:potts-gd}, \ref{lem:potts-block} and the results proved in this section.

\begin{proof}[Proof of Theorem~\ref{thm:Potts-SI}]
	Part 1 follows directly from Lemma~\ref{lem:potts-gd} and Theorem~\ref{thm:glauber-hamming}. 
	For Part 2, we note that the block dynamics from Lemma~\ref{lem:potts-block} corresponds to a select-update dynamics with $M = O(1/\delta)$ and $D = \Delta^{O(1/\delta)} /n$; the reason of the latter is that each $\vx$ is in at most $\Delta^{O(M)}$ blocks. 
	The theorem then follows from Lemma~\ref{lem:potts-block} and Theorem~\ref{thm:general}.
\end{proof}

We end this section with the proof of Theorem~\ref{thm:Ising-Potts-mixing}.

\begin{proof}[Proof of Theorem~\ref{thm:Ising-Potts-mixing}]
For Ising model, spectral independence is known in the whole uniqueness region \cite{CLV1}. 
For Potts model, Theorem~\ref{thm:Potts-SI} establishes spectral independence in the corresponding parameter regimes. 
The theorem then follows from Theorems~\ref{thm:block-factorization} and \ref{thm:sw:general}. 
\end{proof}

\section{Spectral independence and entropy factorization}\label{sec:spectral-ind}
The goal of this section is to reformulate in the setting of spin systems some of  the key facts that were derived in \cite{CLV20} and the references therein in the more general  framework of simplicial complexes. This specialization yields some minor simplification in the main proofs, and may be of use for later reference. The approach consists in exploiting a recursive scheme which allows one to derive a global contraction estimate by analysing the spectral norm of a local operator. This % which is the analogue of the local walks in the simplicial complex framework. As a side remark, the translation
is reminiscent of the recursive approach developed in \cite{CCL,CapMar03,Cap},
where similar ideas were used to derive spectral gap estimates for a class of conservative spin systems.
The argument here is more robust and, unlike the one in \cite{CCL,CapMar03,Cap}, it does not rely on symmetries of the underlying measures.

We first introduce some notation. %rewrite some of the notation in    \cite{CLV20}.
Let $f$ be a function of the full spin configuration $\si$, and $U\subset V=[n]$ a subset of vertices. We use
the  notation $\mu^U=\mu_{V\setminus U}$  for the conditional distribution given the spins in $U$, and write $\Av_{|U|=\ell}$ for the uniform average over all sets $U\subset[n]$ with $\ell$ vertices.
We are going to prove the following result that was established in \cite{CLV20}.
\begin{theorem}\label{th:reform}
If the spin system is $\eta$-spectrally independent and $b$-marginally bounded then there exists a constant $C=O(1+\frac{\eta}{b})$ such that for any $\ell=\{1,\dots,n-1\}$ and for all $f\geq 0$: %local inequality \eqref{eq:locfact} holds with
\begin{align}\label{eq:ubfsub11}
\frac{n}{\ell}\,\Av_{|U|=\ell}\,\ent(\mu^U f)\leq C\,\ent f.
\end{align}
Moreover, for any $\theta \in(0,1]$, there exists $C=\left(\frac{1}{\theta}\right)^{O(\frac{\eta}{b})}$ such that for $\ell=\lceil \theta n\rceil$:
\begin{align}\label{eq:ubfsub121}
\frac{\ell}{n}\,\ent f\leq C \,\Av_{|\L|=\ell}\,\mu\left[ \ent_{\L}f\right].
\end{align}
\end{theorem}
We remark that \eqref{eq:ubfsub11} is an approximate subadditivity statement, which coincides with \eqref{eqn:sub-add} when $\ell=1$. On the other hand \eqref{eq:ubfsub121} is the uniform block factorization statement $\ell$-UBF with $\ell=\lceil \theta n\rceil$; see Definition \ref{def:ubf}. %the size of the blocks growing linearly with $n$. 
We articulate the proof in two steps. The first is a recursive scheme which allows one to go from a local inequality to a global one; see Lemma \ref{localtoglobal}. The second step is a control of the local inequality; see Lemma \ref{lem:alphak}.

\subsection{Setting up  the recursion}
If $U\subset V$, and $\t=\t_U$ a configuration of spins on $U$, recall that  we use notation $\mu^\t=\mu(\cdot\tc\t)$ for the conditional distribution $\mu^U$ when the spins on $U$ are given by $\t$.
Moreover, we write $\mu^{\t,x}=\mu(\cdot\tc\t\cup \si_x)$
if we additionally condition on the spin $\si_x$ at vertex $x\notin U$ and similarly for $\mu^{\t,x,y}=\mu(\cdot\tc\t\cup \si_x\cup \si_y)$ for $x,y\notin U$, so that e.g.\ the expression
$\mu^\t\left[ \ent_{\mu^{\t,x,y}}f\right]$ indicates the entropy of $f$ with respect to  $\mu(\cdot\tc\t\cup \si_x\cup \si_y)$,
$$
 \ent_{\mu^{\t,x,y}}f =  \mu^{\t,x,y}[f\log(f/ \mu^{\t,x,y}(f))]\,
$$
averaged over the two spins $\si_x,\si_y$ sampled according to $\mu^\t$.
Define the constants $\a_k$, $k=0,\dots,n-2$, as the largest numbers such that the inequalities
\begin{equation}\label{eq:locfact}
(1+\a_k)\Av_{x\notin U}\,\ent_{\mu^\t}(\mu^{\t,x} (f))\leq \Av_{x,y\notin U}\,\ent_{\mu^\t}(\mu^{\t,x,y} (f))\,,
\end{equation}
hold for all $k=0,\dots,n-2$, for all $U\subset [n]$ with $|U|=k$, for all configurations $\t$ on $U$ and for all functions $f\geq 0$.
The symbol $\Av_{x\notin U}$ denotes the uniform average over all $n-k$ vertices $x\notin U$, and $\Av_{x,y\notin U}$ stands for
 the uniform average over all $(n-k)(n-k-1)$ pairs $(x,y)$ with $x,y\notin U$ and $x\neq y$. We refer to \eqref{eq:locfact} as the local inequality, since for each choice of $x,y$, the distributions involved are concerned with the spins at two vertices only.

\begin{remark}
Fix $x,y\notin U$. Using $\mu^{\t,x}f = \mu^{\t,x}\mu^{\t,x,y}f$, from Lemma \ref{lem:telescope} we have the decomposition
$$
\ent_{\mu^\t}(\mu^{\t,x,y} (f) )= \ent_{\mu^\t}(\mu^{\t,x} (f)) +  \mu^\t\left[\ent_{\mu^{\t,x}}(\mu^{\t,x,y} (f)
\right].
$$
In particular, $\ent_{\mu^\t}(\mu^{\t,x,y} (f) )\geq \ent_{\mu^\t}(\mu^{\t,x} (f))$ and therefore \eqref{eq:locfact} is always true with $\a_k=0$.
If $\mu$ is a product measure then the subadditivity of entropy for product measures gives
$$
\ent_{\mu^\t}(\mu^{\t,x,y} (f))\geq \ent_{\mu^\t}(\mu^{\t,x} (f)) + \ent_{\mu^\t}(\mu^{\t,y} (f)),
$$
which implies the validity of  \eqref{eq:locfact} with $\a_k=1$ for all $k=0,\dots,n-2$.
%
%the Shearer inequality for the ``all but one" cover implies % block factorization  with the
% \begin{equation}\label{eq:locfact2}
%\ent_{\mu^\t}(f) \leq  \frac{n-k}{n-k-1} \Av_x\, \mu^\t\left[ \ent_{\mu^{\t,x}}f\right],
%\end{equation}
\end{remark}
The recursion is based on the following statement, which rephrases \cite[Theorem 5.4]{CLV20}.
\begin{lemma}\label{localtoglobal}
Let $\a_k$, $k=0,\dots,n-2$, be defined by \eqref{eq:locfact}. Then, for all functions $f\geq 0$,
\begin{equation}\label{eq:locfact5}
\Av_{|U|=j} \ent(\mu^Uf)
\leq (1-\k_j)\ent(f),\qquad j=1,\dots,n-1,
\end{equation}
where $$
\kappa_j = \frac{\sum_{i=j}^{n-1}\G_i}{\sum_{i=0}^{n-1}\G_i}\,,\qquad \G_i = \prod_{k=0}^{i-1}\a_k\,,\quad \G_0=1.
$$
\end{lemma}
\begin{proof}
The claim \eqref{eq:locfact5}
follows from the fact that for all $k=1,\dots,n-1$:
\begin{equation}\label{eq:locfact4}
\Av_{|U|=k} \,\ent(\mu^Uf)\leq  \d_k\Av_{|U|=k+1} \,\ent(\mu^Uf)\,,\qquad \d_k= \frac{\sum_{i=0}^{k-1}\G_i}{\sum_{i=0}^{k}\G_i},
\end{equation}
since $\Av_{|U|=n} \,\ent(\mu^Uf)=\ent(f)$, and $\d_j\d_{j+1}\cdots \d_{n-1}= (1-\kappa_j)$.
%This is Eq. (10) in  \cite{CLV20}.

To prove \eqref{eq:locfact4}, note that it holds for $k=1$ with $\d_1=1/(1+\a_0)=\G_0/(\G_0+\G_1)$ by the assumption \eqref{eq:locfact} at $\t=\emptyset$. Next, we suppose it holds for $0<k-1<n-1$ and show it for $k$. For any $|U|=k+1$ and $U'\subset U$ with $|U'|=k-1$, setting $\{x,y\}=U\setminus U'$ and letting $\t=\t_{U'}$ be the configuration on $U'$, as in Lemma \ref{lem:telescope} we have the decomposition
\begin{align}\label{eq:locfact6}
\ent(\mu^Uf) &= \ent(\mu(\mu^Uf\tc\t_{U'})) + \mu\left[
\ent(\m^Uf\tc \t_{U'})\right]\nonumber\\
&=\ent(\mu^{U'}f) + \mu\left[
\ent_{\mu^\t}(\mu^{\t,x,y}f)\tc \t_{U'}\right].
\end{align}
Averaging we obtain
\begin{align}\label{eq:locfact7}
\Av_{|U|=k+1}\ent(\mu^Uf) &=
\Av_{|U'|=k-1}\ent(\mu^{U'}f) \\ &\qquad + \Av_{|U'|=k-1}\Av_{x,y\notin U'}\mu\left[
\ent_{\mu^\t}(\mu^{\t,x,y}f)\tc \t_{U'}\right].
\end{align}
In the same way
\begin{align}\label{eq:locfact8}
\Av_{|U|=k}\ent(\mu^Uf) &=
\Av_{|U'|=k-1}\ent(\mu^{U'}f) \\ &\qquad + \Av_{|U'|=k-1}\Av_{x\notin U'}\mu\left[
\ent_{\mu^\t}(\mu^{\t,x}f)\tc \t_{U'}\right].
\end{align}
From \eqref{eq:locfact},
\begin{align}\label{eq:locfact9}
\Av_{|U|=k+1}\ent(\mu^Uf) &-
\Av_{|U'|=k-1}\ent(\mu^{U'}f) \\ &\geq  (1+\a_{k-1})\Av_{|U'|=k-1}\Av_{x\notin U'}\mu\left[
\ent_{\mu^\t}(\mu^{\t,x}f)\tc \t_{U'}\right]\\
& = (1+\a_{k-1})\left[\Av_{|U|=k}\ent(\mu^Uf) -
\Av_{|U'|=k-1}\ent(\mu^{U'}f)\right].
\end{align}
Therefore,
\begin{align}\label{eq:locfact10}
\Av_{|U|=k+1}\ent(\mu^Uf) \geq
(1+\a_{k-1})\Av_{|U|=k}\ent(\mu^Uf)- \a_{k-1} \Av_{|U'|=k-1}\ent(\mu^{U'}f) .
\end{align}
By the inductive assumption \eqref{eq:locfact4}  at $k-1$ we have
\begin{align}%\label{eq:locfact11}
\Av_{|U|=k+1}\ent(\mu^Uf) &\geq
(1+\a_{k-1}- \a_{k-1}\d_{k-1})\Av_{|U|=k}\ent(\mu^Uf)\nonumber \\&= \d_{k}^{-1}\Av_{|U|=k}\ent(\mu^Uf). \qedhere
\end{align}
\end{proof}
\subsection{Estimating the local coefficients} %$\texorpdfstring{$\a_k$}{alpha\_k}'s}

The next step is an estimate on the coefficients $\a_k$ appearing in \eqref{eq:locfact}.
%We follow  \cite{CLV20} and references therein, but things are reformulated in the spin setting, which perhaps yields some simplification
%(note that there is no reference to local walks in the proof and one goes straight to the estimate of the maximum eigenvalue of the influence matrix).
%%allows one to avoid the introduction of local walks or any other auxiliary Markov chain.
%
%
%Our aim here is to prove:
\begin{lemma}\label{lem:alphak}
If the spin system is $\eta$-spectrally independent and $b$-marginally bounded then the local inequality \eqref{eq:locfact} holds with
\begin{align}\label{eq:facto35}
\a_k \geq 1- \frac{2\eta}{b(n-k-1)}.
\end{align}
\end{lemma}
%This corresponds to the first bound in Eq.(11) of  \cite{CLV20}.
\begin{proof}
Fix $U\subset V$, $|U|=k\leq n-2$ and $\t=\t_U$. We may assume $\mu^\t(f)=1$, which implies $\mu^\t(\mu^{\t,x,y} (f))=\mu^\t(\mu^{\t,x} (f))=1$ for all $x,y\notin U$. For simplicity, we write $\Av_{x,y}$ and $\Av_{x}$ for the averages $\Av_{x,y\notin U}$ and $\Av_{x\notin U}$. Observe that
\begin{align}\label{eq:facto2}
&\Av_{x,y}\,\ent_{\mu^\t}(\mu^{\t,x,y} (f))-2 \Av_x\,\ent_{\mu^\t}(\mu^{\t,x} (f))
\nonumber \\&\qquad =\Av_{x,y}\,\mu^\t\left[\mu^{\t,x,y} (f)\log \mu^{\t,x,y} (f) - \mu^{\t,x} (f)\log \mu^{\t,x} (f)-\mu^{\t,y} (f)\log \mu^{\t,y} (f)\right]
\\&\qquad =\Av_{x,y}\,\mu^\t\left[\mu^{\t,x,y} (f)\log \frac{\mu^{\t,x,y} (f)}{\mu^{\t,x} (f)\mu^{\t,y} (f)} \right].
\end{align}
Using $a\log(a/b)\geq a-b $ for all $a,b\geq 0$,
\begin{align}\label{eq:facto3}
&\Av_{x,y}\,\ent_{\mu^\t}(\mu^{\t,x,y} (f))-2 \Av_x\,\ent_{\mu^\t}(\mu^{\t,x} (f))
\nonumber \\&\qquad \geq
1- \Av_{x,y}\,\mu^\t\left[\mu^{\t,x} (f)\mu^{\t,y} (f) \right]
\nonumber \\&\qquad =
-\Av_{x,y}\,\mu^\t\left[(\mu^{\t,x} (f)-1)(\mu^{\t,y} (f)-1) \right].
\end{align}
We may rewrite
\begin{align}\label{eq:facto31}
&\Av_{x,y}\,\mu^\t\left[(\mu^{\t,x} (f)-1)(\mu^{\t,y} (f)-1) \right] \nonumber \\ & \qquad \qquad = \frac1{n-k-1}\sum_{(x,a)\in \VSpairs}%\sum_{(y,a')\in\VSpairs}
\nu(x,a)\varphi(x,a)[J^\t\varphi](x,a),
\end{align}
where $$\varphi(x,a)=\mu^{\t} (f\tc \si_x=a)-1=[\mu^{\t,x} (f)](a)-1,$$
 $\VSpairs$ is the set of all pairs $(x,a)$ where $x\in V\setminus U$ (if $U$ is the set where $\t=\t_U$ is specified) and $a\in[q]$, $\nu$ denotes the probability measure on $\VSpairs$ obtained by setting $$\nu(x,a)=\frac1{n-k}\,\mu^\t(\si_x=a),$$  and
$J^\t:\VSpairs\times\VSpairs\mapsto \bbR$ denotes the influence matrix from Definition \ref{def:J}.
Note that in the derivation of \eqref{eq:facto31} we have used the fact that
for each fixed $y\notin U$ one has
$$\sum_{a'\in[q]} \nu(y,a') \varphi(y,a') = \frac1{n-k}\,\mu^\t(\mu^{\t,y}(f)-1) = 0.
 $$
Observe that $J^\t$ is self-adjoint in $L^2(\VSpairs,\nu)$:
\begin{equation}
\label{self-adjoint}
\nu(x,a)J^\t(x,a;y,a')=\nu(y,a')J^\t(y,a';x,a).
\end{equation}
In particular, its eigenvalues are real. Let $\g\geq 0$ denote its largest eigenvalue (the eigenvalue zero always exists since all row sums of $J^\t$ vanish).  Letting $\scalar{\cdot}{\cdot}$ denote the scalar product in $L^2(\VSpairs,\nu)$ we have $\scalar{\psi}{J^\t\psi}\leq \eta\scalar{\psi}{\psi}$ for all $\psi \in L^2(\VSpairs,\nu)$. Therefore,
 \begin{align}\label{eq:facto32}
&\Av_{x,y}\,\mu^\t\left[(\mu^{\t,x} (f)-1)(\mu^{\t,y} (f)-1) \right]\nonumber\\ &\qquad = \frac1{n-k-1}\scalar{\varphi}{J^\t\varphi}
\leq  \frac{\g}{n-k-1}\scalar{\varphi}{\varphi} \nonumber \\&\qquad = \frac{\g}{n-k-1}\,\Av_{x}\,\mu^\t\left[(\mu^{\t,x} (f)-1)^2 \right] =\frac{\g}{n-k-1}\,\Av_{x}\var_{\mu^\t}(\mu^{\t,x} (f)).
\end{align}
%where the estimate follows by the spectral decomposition of $J$.
Recalling \eqref{eq:facto3}  we have shown
\begin{align}\label{eq:facto33}
\Av_{x,y}\,&\ent_{\mu^\t}(\mu^{\t,x,y} (f))-2 \Av_x\,\ent_{\mu^\t}(\mu^{\t,x} (f))
\nonumber \\ & \qquad \quad\geq
- \frac{\g}{n-k-1}\,\Av_{x}\var_{\mu^\t}(\mu^{\t,x} (f)) .
\end{align}
%A priori $\g$ could depend on $n,k$ and $\t$, but we are going to assume that $\g>0$ is bounded  uniformly in all parameters, that is the system is $\g$-spectrally independent in the sense of \cite{ALO20}.
Next, observe that for every fixed $x\notin U$, setting $h^\t(\si_x)=[\mu^{\t,x} (f)](\si_x)$:
%$$
%1 = \sum_{\si_x} \mu^\t[\mu^{\t,x} (f)] = \sum_{\si_x} \mu^\t(\si_x=a) [\mu^{\t,x} (f)](\si_x) \geq b\,|\mu^{\t,x} (f)|_\infty\,,
%$$
\begin{align*}
\var_{\mu^\t}(\mu^{\t,x} (f))& = \sum_{a} \mu^\t(\si_x=a)(h^\t(a) - 1)^2\\
& \leq \frac1b \left(\sum_{a} \mu^\t(\si_x=a)|h^\t(a) - 1|\right)^2
\end{align*}
where $b=\min_{x\notin U}\min_{a} \mu^\t(\si_x=a)$,  as in Definition \ref{def:margb}, with the minimum over $a$ restricted to spin values that are allowed at $x$, that is such that $\mu^\t(\si_x=a)>0$, and we have used $\sum_i a_i^2\leq (\sum_i a_i)^2$ for all $a_i\geq 0$.
Pinsker's inequality shows that
$$
\sum_{a} \mu^\t(\si_x=a)|h^\t(a) - 1|\leq  \sqrt{2\,\ent_{\mu^\t}(\mu^{\t,x} (f))}.
$$
 It follows that
\begin{align}\label{eq:var_ent}
\var_{\mu^\t}(\mu^{\t,x} (f))\leq \frac2b\,\ent_{\mu^\t}(\mu^{\t,x} (f)).
\end{align}
Inserting \eqref{eq:var_ent} into  \eqref{eq:facto33}
concludes the proof. \end{proof}

%We note that the notion of $b$-marginally bounded measure  is the same appearing in \cite{Marton15}, which uses a sort of reverse Pinsker inequality to compare total variation distance and entropy.

%The above discussion may be summarized as follows.

\subsection{Proof of Theorem \ref{th:reform}}
From  Lemma \ref{localtoglobal}, we see that \eqref{eq:ubfsub11}  holds with $C=\frac{n}\ell(1-\kappa_\ell)$.
From Lemma \ref{lem:alphak} if follows that $$\a_k\geq \max\{1-R/(n-k-1),0\},\qquad R=\lceil 2\eta/b\rceil.$$
Using this bound in the definition of the coefficients $\kappa_\ell$ and rearranging, see Section 2.2 of  \cite{CLV20}, it is not hard to see that for any $1\leq \ell\leq n-1$:
\begin{align}\label{eq:Asj}
\kappa_{\ell} \geq \frac{(n-\ell-1)\cdots(n-\ell-R)}{(n-1)\cdots(n-R)}.
\end{align}
%so that the uniform block-factorization \eqref{eq:ubf} holds with constant $$\frac{\ell}{(n\kappa_{n-\ell})}\leq \frac{(n-1)\cdots(n-R)}{(\ell-1)\cdots(\ell-R)}.$$
%This bound is $O(1)$ as long as $\ell=\O(n)$ but it gets worse as $\ell$ decreases and it becomes of order $n^R$ for $\ell=O(1)$.
%Thus it implies an approximate Shearer inequality for uniform covers with blocks of size $\ell=\O(n)$; see \eqref{eq:ubfsub121}. On the other hand if we look at the approximate subadditive version of Shearer inequality \eqref{eq:ubfsub}, then we are interested in an upper bound on
In particular, $$\frac{n}{\ell}(1-\kappa_{\ell})\leq \frac{n}{\ell}\left(1-\frac{(n-\ell-1)\cdots(n-\ell-R)}{(n-1)\cdots(n-R)}\right).$$
Remarkably, the expression in the right hand side above is  decreasing with $\ell$, and therefore it is always less than $R+1$, its value at $\ell=1$. This shows that \eqref{eq:ubfsub11}  holds with $C\leq R+1=O(1+\frac{\eta}{b})$.

To prove \eqref{eq:ubfsub121}, we start with the decomposition
$$
\Av_{|\L|=\ell}\,\mu\left[ \ent_{\L}f\right] =\ent(f)-\Av_{|U|=n-\ell}\,\ent\left[ \mu^U f\right] ,
$$
which follows from Lemma \ref{lem:telescope}.
Therefore Lemma \ref{localtoglobal} implies that \eqref{eq:ubfsub121} holds with $C=\frac{\ell}{n\,\kappa_{n-\ell}}$. Using \eqref{eq:Asj} we see that
 $$\frac{\ell}{n\,\kappa_{n-\ell}}\leq \frac{(n-1)\cdots(n-R)}{(\ell-1)\cdots(\ell-R)}.$$
In particular, if $\ell=\lceil \theta n\rceil$ with $\theta \in(0,1]$ fixed, then for all sufficiently large $n$ one has $\frac{\ell}{n\,\kappa_{n-\ell}}\leq (\frac{1}{\theta})^{O(R)}$. This ends the proof of Theorem \ref{th:reform}.

\section{Optimal mixing of the SW dynamics}
\label{sec:SW}

In this section, we show that for ferromagnetic Potts models, the $k$-partite factorization of entropy, as defined in~\eqref{eq:kpart}, implies optimal mixing of the Swendsen-Wang (SW) dynamics.
Since we have already established that, for any spin system,
$k$-partite factorization is implied by
spectral independence, we then deduce Theorem~\ref{thm:sw:general} from the introduction.

We again take $G=(V,E)$ to be an $n$-vertex graph of maximum degree $\Delta$ %independent of $n$ 
and $\mu$ to be the Potts distribution on $G$ with configuration space $\Omega = [q]^V$.
The SW dynamics takes a spin configuration, transforms it into a ``joint'' spin-edge configuration, performs a step in the joint space, and then drops the edges to obtain a new  Potts configuration. Formally, from a Potts configuration $\sigma_t\in [q]^V$,
a transition $\sigma_t\rightarrow\sigma_{t+1}$ of the SW dynamics is defined as follows:
\begin{enumerate}
	\item Let $M_t=M(\sigma_t)$ %=E\setminus D(\sigma_t)$ 
	denote the set of monochromatic edges in $\sigma_t$.
	\item Independently for each edge $e\in M_t$, keep $e$ with probability $p=1-\exp(-\beta)$
	and remove $e$ with probability $1-p$.  Let $A_t \subset M_t$ denote the resulting subset.
	\item In the subgraph $(V,A_t)$, independently for each connected component $C$ (including isolated vertices),
	choose a spin $s_C$ uniformly at random from $[q]$ and assign to each vertex in $C$ the spin $s_C$.
	This spin assignment defines $\sigma_{t+1}$.
\end{enumerate}

It will be useful for us to consider the ``joint'' Edwards-Sokal distribution for $G$ with parameters $p\in[0,1]$ and integer $q \ge 2$.
Let $\Joint =\Omega \times \{0,1\}^E$
be the set of ``joint'' spin-edge configurations $(\sigma,A)$ consisting of a spin assignment to the vertices $\sigma \in \Omega$ and a subset of edges $A \subset E$.	
The Edwards-Sokal measure assigns to each $(\sigma,A) \in \Joint$ a probability given by
\begin{equation}\label{jointnu-intro}
\nu(\si,A) = \frac1{Z_\textsc{j}}\,p^{|A|}(1-p)^{|E|-|A|}\IND(\si\sim A),
\end{equation}
where $\si\sim A$ means that
$A \subset M(\si)$
(i.e., every edge in $A$ is monochromatic in $\sigma$)
and $Z_\textsc{j}$ is the corresponding normalizing constant or partition function.
When $p = 1 - e^{-\b}$, the ``spin marginal'' of $\nu$ is precisely the Potts distribution $\mu$ and $Z_G = Z_\textsc{j}$, and the ``edge marginal'' of $\nu$ corresponds to the random-cluster measure; see, e.g., \cite{FK, Grimmett} for extensive background on these measures.

%Our goal in this section is to establish Theorem~\ref{th:main-intro} from the introduction, which states that when $\mu$ is $\eta$-spectrally independent the SW dynamics mixes in $O(\log n)$ steps. This bound for the mixing time is optimal, since on trees of bounded degree the SW dynamics is now known to require $\Omega(\log n)$ steps for all $\beta > 0$~\cite{BCSVtree}.

% \begin{theorem}\label{th:main-intro}
%For the ferromagnetic Ising and Potts models on a graph of maximum degree $\Delta$, if the system is $\eta$-spectrally independent and $b$-marginally bounded,
%then
%%for all $\Delta$,
%there exists a constant $C=C(b,\eta,\Delta)$ such that
%%for any graph of maximum degree $\Delta$,
%the mixing time of the Swendsen-Wang dynamics is $\leq C\log{n}$ and the modified log-Sobolev constant
%is $\geq C^{-1}$.
%\end{theorem}

%From Section~\ref{sec:ubf:to:gbf}, we know that when $\mu$ is $\eta$-spectrally independent,
%$k$-partite factorization holds with constant $C = C(\Delta,\beta)$.

A key concept in our strategy to prove optimal mixing results for the SW dynamics is the \emph{spin/edge factorization} of entropy in the joint space~$\Joint$.
This spin/edge factorization was shown in \cite[Lemma 1.8]{BCPSV20} to
imply $O(\log n)$ mixing of the SW dynamics on any graph.
Moreover, in~\cite{BCPSV20} it was proved that for bipartite graphs, even/odd factorization of entropy for $\mu$ implies
the desired spin/edge factorization of entropy in the joint space~$\Joint$.  We will generalize
the argument from~\cite{BCPSV20} to general graphs, and show that a $k$-partite factorization of entropy for~$\mu$ implies spin/edge factorization of entropy
in the joint space~$\Joint$, and thus combined with~\cite[Lemma 1.8]{BCPSV20} this will complete the proof of
$O(\log{n})$ mixing of the SW dynamics.
%\eric{I found the discussion in this paragraph about \cite{BCPSV20} a bit confusing.  It wasn't clear what was shown there for general vs. bipartite graphs.
%So I changed this paragraph a bit.}
Note, this $O(\log{n})$ bound is optimal as there are graphs of bounded degree where the SW dynamics requires $\Omega(\log n)$ steps to mix.

Before stating our results, we stipulate some notation. We write $\ent_\nu(f)=\nu[f\log(f/\nu(f))]$ for the entropy of the function $f: \Joint \to \R_{+}$ with respect to $\nu$.
For a fixed configuration $\s \in \Omega$ and subset of edges $A \subset E$,
$\ent_\nu (f\tc\si)$ and $\ent_\nu (f\tc A)$  denote
the entropy of $f$ with respect to the conditional measures $\nu(\cdot\tc\si)$ and $\nu(\cdot\tc A)$, respectively. More precisely, for a given $\si\in\Omega$, $\nu(\cdot\tc\si)$ is the measure $\nu$ conditioned on the event that the spin configuration is equal to $\sigma$, and for a given $A\subset E$, $\nu(\cdot\tc A)$ is the measure $\nu$ conditioned on the event that the edge configuration is equal to $A$.
In this way, $\ent_\nu (f\tc\si)$ and $\ent_\nu (f\tc A)$ are functions of $\si$ and $A$, respectively, and $\nu\left[\ent_\nu (f\tc\si)\right]$, $\nu\left[\ent_\nu (f\tc A)\right]$ denote the corresponding expectations with respect to $\nu$.

\begin{theorem}\label{th:main-intro2}
	Suppose $\mu$ satisfies the $k$-partite factorization of entropy with constant $C_{\rm par}$; see Eq.~\eqref{eq:kpart}.
	Then, there exists a constant $C=C(C_{\rm par},\beta,\Delta)$ such that for all $f: \Joint \mapsto \bbR_+$
	\begin{align}\label{entfact2o-intro}
	\ent_\nu (f)\leq C \,\left(\nu\left[\ent_\nu (f\tc\si)] +\nu[\ent_\nu (f\tc A)\right]\right).
	\end{align}
	The constant $C$ satisfies $C=C_{\rm par}\times O(\beta\Delta^2e^{\beta\Delta})$.
\end{theorem}

We call \eqref{entfact2o-intro} the {\em spin/edge factorization of entropy} with constant $C$ for the joint measure $\nu$. The main motivation for this inequality is the result established in \cite[Lemma 1.8]{BCPSV20} that  on any $n$-vertex graph, approximate spin/edge factorization with constant $C$ implies that the SW dynamics has discrete time entropy decay with rate $\delta=1/C$, and therefore, by Lemma \ref{lem:mix}, satisfies $\tmix = O(\log n)$. 
Theorem \ref{thm:sw:general} from the introduction now follows immediately.

\begin{proof}[Proof of Theorem \ref{thm:sw:general}]
	For the Potts model one has $e^{\beta\Delta}=O(1/b)$. Therefore, the results follows from Theorem~\ref{th:reform},
	Lemma~\ref{lem:ubf:boosting},
	Theorem~\ref{th:main-intro2}
	and  \cite[Lemma 1.8]{BCPSV20}. 
\end{proof}

Let $\{V_1,...,V_k\}$ be the $k$-partition of $G$, where $k \leq \Delta+1$, as in Section \ref{sec:ubf:to:gbf}.
For all $j\in[k]$ let $\nu(\cdot\tc\overline{\si_{V_j}},A)$ denote the measure $\nu$ conditioned on $\overline{\si_{V_j}}=\{\si_v,\,v\notin V_j\}$ and $A \subset E$.
We use $\ent_\nu(f\tc \overline{\si_{V_j}},A)$ to denote the corresponding conditional entropy and $\nu\left[\ent_\nu (f\tc\overline{\si_{V_j}},A)\right]$ for its expectation with respect to $\nu$.
Theorem~\ref{th:main-intro2} will follow from the following lemmas.

\begin{lemma}
	\label{lem:conv}
	For  all $f: \Joint \mapsto \R_+$ and all $j \in [k]$ we have
	\begin{align*}
	%\label{swent6}
	&\nu\left[\ent_\nu (f\tc A)\right]\geq
	\nu\left[\ent_\nu (f\tc \overline{\sigma_{V_j}},A)\right].
	\end{align*}
\end{lemma}

\begin{lemma}\label{lem:tensor}
	There exists a constant $\d_1>0$ such that, for  all $f: \Joint \mapsto \R_+$ and all $j \in [k]$,
	\begin{align*}
	&\nu\left[\ent_\nu (f\tc \si)\right] + \nu\left[\ent_\nu (f\tc \overline{\si_{V_j}},A)\right]
	\geq \d_1\,\nu\left[\ent_\nu (f\tc \overline{\si_{V_j}})\right].	%\label{swent8a}
	\end{align*}
	The constant $\d_1$ satisfies $1/\d_1=O(\beta\Delta e^{\beta\Delta})$.
\end{lemma}

\begin{lemma}\label{lem:conc}
	If $\mu$ satisfies the $k$-partite factorization  with constant $C_{\rm par}$, then  for all $f: \Joint \mapsto \R_+$,
	\begin{align*}	
	\sum_{j=1}^k\nu\left[\ent_\nu (f\tc \overline{\si_{V_j}})\right]\geq \d_2 \ent_\nu(f),
	\end{align*}
	where $\d_2=\frac1{C_{\rm par}}$.
\end{lemma}

\begin{proof}[Proof of Theorem \ref{th:main-intro2}]
	By combining the bounds from Lemmas \ref{lem:conv}, \ref{lem:tensor} and \ref{lem:conc} we get
	\begin{align}\label{swent9}
	\nu\left[\ent_\nu (f\tc \si) + \ent_\nu (f\tc A)\right] &\geq  \frac{\d_1\d_2}{k}\ent_\nu(f),
	\end{align}
	and so, using also $k\leq \Delta+1$, the spin/edge factorization holds with constant 
	\[
	C=\frac{k}{\d_1\d_2}=C_{\rm par}\times O(\beta\Delta^2e^{\beta\Delta}). \qedhere
	\]
\end{proof}

We briefly discuss next the proof of Lemmas~\ref{lem:conv}, \ref{lem:tensor} and \ref{lem:conc}, which are
the respective counterparts of  Lemmas 4.3, 4.4 and 4.5 in \cite{BCPSV20} for the bipartite setting.
%\antonio{We should discuss whether to omit the following proofs or include them in an appendix instead.}
\begin{proof}[Proof of Lemma~\ref{lem:conv}]	
	This is an instance of the same monotonicity already seen in Lemma \ref{lem:mono}. In this particular case, it follows from the argument in the proof of Lemma 4.3 in \cite{BCPSV20} by simply substituting $\sigma_O$ with $\overline{\sigma_{V_j}}$ in that proof.
\end{proof}

\begin{proof}[Proof of Lemma~\ref{lem:tensor}]	
	Let us fix $j \in [k]$. To simplify the notation, we shall use $xy$ to denote the edge $\{x,y\}$,
	and view the edge configuration $A$ as a vector in $\{0,1\}^E$. For any fixed configuration $\overline{\si_{V_j}}$ of spins, the conditional probability $\nu(\cdot\tc\overline{\si_{V_j}})$ is a product measure.
	That is,
	\begin{align}\label{eo1}
	\nu(\cdot\tc\overline{\si_{V_j}})
	= \bigotimes_{x\in V_j}\nu_x(\cdot\tc\overline{\si_{V_j}}),
	\end{align}
	where,
	for each $x\in V_j$, $\nu_x(\cdot\tc\overline{\si_{V_j}})$ is the probability measure on $\{1,\dots,q\}\times \{0,1\}^{\deg(x)}$
	with $\deg(x)$ denoting the degree of $x$.
	The distribution $\nu_x(\cdot\tc\overline{\si_{V_j}})$ can be described s follows: pick the spin of site $x$ according to the Potts measure on $x$ conditioned on the spin of its neighbors in $V \setminus V_j$; then, independently for every edge $xy \in E$ incident to the vertex $x$, if $\si_x=\si_y$ set $A_{xy}=1$
	with probability $p$ and set $A_{xy}=0$ otherwise; if $\si_x\neq\si_y$, set $A_{xy}=0$.

	The measure $\nu(\cdot\tc\overline{\si_{V_j}},A)$, obtained by further conditioning on a valid configuration of all edge variables $A$ compatible with the fixed spins $\overline{\si_{V_j}}$, is again a product measure:
	\begin{align}\label{eo11}
	\nu(\cdot\tc \overline{\si_{V_j}},A)
	= \bigotimes_{x\in V_j}\nu_x(\cdot\tc\overline{\si_{V_j}},A),
	\end{align}
	where $\nu_x(\cdot\tc\overline{\si_{V_j}},A)$ is the probability measure on $\{1,\dots,q\}$ that is uniform if $x$ has no incident edges in~$A$, and is concentrated on the unique admissible value given $\overline{\si_{V_j}}$ and $A$ otherwise.
	
	Next, we note that $\nu(\cdot\tc\si)$ is a product of Bernoulli($p$) random variables over all monochromatic edges in $\si$, while it is concentrated on $A_e=0$ on all remaining edges. Therefore we may write
	\begin{align}\label{eo12}
	\nu(\cdot\tc\si)
	= \bigotimes_{x\in V}\nu_x(\cdot\tc\si),
	\end{align}
	where $\nu_x(\cdot\tc\si)$ is the probability measure on $\{0,1\}^{\deg(x)}$ given by  the product of Bernoulli($p$) variables on all edges $xy$ incident to $x$ such that $\si_x=\si_y$ and is concentrated on $A_{xy}=0$ if $\si_x\neq\si_y$.
	
	%For a detailed description of these distributions see Lemma 4.4 in \cite{BCPSV20}.
	
	We write
	$\ent_x(\cdot\tc \overline{\si_{V_j}})$, $\ent_x(\cdot\tc \overline{\si_{V_j}},A)$, $\ent_x(\cdot\tc \si)$
	for the entropies with respect to the distributions $\nu_x(\cdot\tc \overline{\si_{V_j}})$, $\nu_x(\cdot\tc \overline{\si_{V_j}},A)$, $\nu_x(\cdot\tc \si)$ respectively.
	The first observation is that, for every site $x$, there is a local factorization of entropies in the following sense. There exists a constant $\d_1>0$ such that $1/\d_1=O(\beta\Delta e^{\beta\Delta})$, and such that for all functions $f\geq 0$ and all $x\in V_j$,
	\begin{align}\label{swent18}
	\nu_x\left[\ent_x (f\tc \si)\tc \overline{\si_{V_j}}\right] + \nu_x\left[\ent_x (f\tc \overline{\si_{V_j}},A)\tc \overline{\si_{V_j}}\right]
	\geq \d_1\,\ent_x (f\tc \overline{\si_{V_j}});
	\end{align}
	this follows from a direct generalization of Lemma 4.7 from \cite{BCPSV20} for bipartite graphs; the proof of such generalization to the $k$-partite setting is  the same as that of Lemma 4.7 and is thus omitted.

	%Note that the proof of \eqref{swent18} is stated in bipartite graph setting, but can be trivially generalized to the $k-$partite setting.
	
	Next, we want to lift inequality~\eqref{swent18} to the product measure
	$\nu(\cdot\tc\overline{\si_{V_j}})=\otimes_{x\in V_j}\nu_x(\cdot\tc\overline{\si_{V_j}})$.
	Let $x=1,\dots, n$ denote an arbitrary ordering of the sites $x\in V_j$.
	Let $A_x \in \{0,1\}^{\deg(x)}$ be the random variable corresponding to the state of the edges incident to $x$. We write $\xi_x=(\si_x,A_x)$ for the pair of variables corresponding to $x$.
	We first observe that
	\begin{align}\label{swent21}
	\ent_\nu(f\tc\overline{\si_{V_j}}) = \sum_{x=1}^n
	\nu\left[ \ent_x (g_{x-1}\tc\overline{\si_{V_j}})\tc \overline{\si_{V_j}}\right],
	\end{align}
	where $g_x=\nu\left[f\tc \overline{\si_{V_j}},\xi_{x+1},\dots,\xi_{n}\right]$,
	$g_0=f$ and $g_{n}=\nu\left[f\tc \overline{\si_{V_j}}\right]$. This identity is an instance of the decomposition in Lemma \ref{lem:telescope}.
	%The equality is the same as eq. (4.15) in \cite{BCPSV20} for $k$-partite graphs.
	
	%To prove \eqref{swent21}, we note that since $\nu(\cdot\tc\overline{\si_{V_j}})=\otimes_{x\in E}~\nu_x(\cdot\tc\overline{\si_{V_j}})$, one has
	%$g_x=\nu\left(f\tc \si_O,\si_{A_x}\right)$ and
	%$\nu_x[g_{x-1}\tc \overline{\si_{V_j}}]=g_{x}.$
	%Therefore,
	%To prove \eqref{tele}, set %$g=\mu_{\L_0}f$,
	%$g_i=\mu_{\L_i}f$, and note that $g_{i}=\mu_{\L_{i}}g_{i-1}$ by \eqref{DLR}. Therefore,
	%\begin{align}
	%\ent_\nu(f\tc\overline{\si_{V_j}}) &= \nu\left[g_0
	%\log\left(g_0/g_{n}\right)\tc\overline{\si_{V_j}}\right]
	% = \mu\left[g_0\log g_0\right] - \mu\left[g_k\log g_k\right]
	%\notag\\
	%&
	%=\sum_{x=1}^{n}  \nu\left[g_{0}\log \left(g_{x-1}/g_{x}\right)\tc\overline{\si_{V_j}}\right]\notag.
	%\end{align}
	%Since the $g_x$ are (conditional) expectations, we deduce
	%\begin{align}
	%\ent_\nu(f\tc\overline{\si_{V_j}})
	%&=\sum_{x=1}^{n}  \nu\left[g_{x-1}\log \left(g_{x-1}/g_{x}\right)\tc\overline{\si_{V_j}}\right] \notag\\
	%&=\sum_{x=1}^{n}  \nu\left[\nu_x\left[g_{x-1}\log \left(g_{x-1}/g_{x}\right)\tc\overline{\si_{V_j}}\right]\tc\overline{\si_{V_j}}\right]\notag\\
	%&=\sum_{x=1}^{n}\nu\left[\ent_{x} (g_{x-1}\tc\overline{\si_{V_j}})\tc\overline{\si_{V_j}}\right]. \label{tele1}
	%\end{align}
	
	Putting together \eqref{swent18} and \eqref{swent21} yields
	\begin{align}
	\d_1\,\ent_\nu (f\tc \overline{\si_{V_j}})&\leq
	\sum_{x=1}^n
	\nu\left[\nu_x\left[\ent_x (g_{x-1}\tc \si)\tc \overline{\si_{V_j}}\right] + \nu_x\left[\ent_x (g_{x-1}\tc \overline{\si_{V_j}},A)\tc \overline{\si_{V_j}}\right]
	\tc \overline{\si_{V_j}}\right]\notag\\&
	= \sum_{x=1}^n
	\nu\left[\ent_x (g_{x-1}\tc \si) + \ent_x (g_{x-1}\tc \overline{\si_{V_j}},A)
	\tc \overline{\si_{V_j}}\right].
	\label{swent22}
	\end{align}
	Proceeding as in the proof of Lemma 4.8  from \cite{BCPSV20}, we obtain the following two inequalities:
	\begin{align*}
	\sum_{x=1}^n
	\nu\left[\ent_x (g_{x-1}\tc \si)\tc \overline{\si_{V_j}}\right] &\leq \nu\left[\ent_\nu (f\tc \si)\tc \overline{\si_{V_j}}\right], \\
	\sum_{x=1}^n
	\nu\left[ \ent_x (g_{x-1}\tc \overline{\si_{V_j}},A)
	\tc \overline{\si_{V_j}}\right]
	&\leq \nu\left[\ent_\nu (f\tc \overline{\si_{V_j}},A)\tc \overline{\si_{V_j}}\right].
	\end{align*}
	These two inequalities combined with~\eqref{swent22} yields that
	\begin{align}\label{swentop}
	\d_1\,\ent_\nu (f\tc \overline{\si_{V_j}})
	\leq  \nu\left[\ent_\nu (f\tc \si)\tc \overline{\si_{V_j}}\right] + \nu\left[\ent_\nu (f\tc \overline{\si_{V_j}},A)\tc \overline{\si_{V_j}}\right].
	\end{align}
	The results follows by taking expectations with respect to $\nu$ in \eqref{swentop}.
\end{proof}

\begin{proof}[Proof of lemma \ref{lem:conc}]
	From the definition of conditional entropy and the fact that $\nu(\cdot\tc \si_{V_j},\overline{\si_{V_j}})=\nu(\cdot\tc \si)$ we get
	\begin{align}
	\label{eq1}
	\ent_\nu (f\tc \overline{\si_{V_j}})
	= \ent_\nu\left(\nu\left[f\tc \si\right]\tc \overline{\si_{V_j}}\right) + \nu\left[\ent_\nu (f\tc \si)\tc \overline{\si_{V_j}}\right].
	\end{align}
	(see eq. (4.5), (4.6) from Lemma 4.5 in \cite{BCPSV20}).	
	Now, since the function 	
	$\nu\left[f\tc \si\right]$ depends only on the spin configuration $\si$,  one has the identity
	\begin{align}
	\label{eq2}
	\sum_{j=1}^k\nu\left[\ent_\nu (\nu[f\tc \si]\tc \overline{\si_{V_j}})\right]
	=
	\sum_{j=1}^k\mu\left[\ent (\nu[f\tc \si]\tc \overline{\si_{V_j}})\right],
	\end{align}
	where the entropy in the right hand side is with respect to $\mu$ and not with respect to $\nu$. 
	%to the function $\nu\left[f\tc \si\right]$.
	Since $k$-partite factorization holds by assumption, 
	\begin{align}
	\label{eq3}
	\sum_{j=1}^k\mu\left[\ent (\nu[f\tc \si]\tc \overline{\si_{V_j}})\right]
	\geq
	\d_2\, \ent \left(\nu\left[f\tc \si\right]\right),
	\end{align}
	where $\d_2=1/C_{\rm par}$. By taking functions depending only on $\sigma_{V_j}$ for a single $V_j$ one easily sees that $C_{\rm par}$ must be at least $1$. 
	Then, taking expectation and summing over $j$ in \eqref{eq1},
	and combining with \eqref{eq2} and \eqref{eq3},  we get
	\begin{align*}
	\sum_{j=1}^k\nu\left[\ent_\nu (f\tc \overline{\si_{V_j}})\right]
	\geq
	\d_2\, \ent_\nu \left(\nu\left[f\tc \si\right]\right) + k\,\nu\left[\ent_\nu (f\tc \si)\right].
	\end{align*}
Using the simple decomposition $ \ent_\nu (f)=\ent_\nu \left(\nu\left[f\tc \si\right]\right) + \nu\left[\ent_\nu (f\tc \si)\right]$, and the fact that $\d_2\leq 1\leq k$, 
we conclude that 
	\begin{align*}
	\sum_{j=1}^k\nu\left[\ent_\nu (f\tc \overline{\si_{V_j}})\right]
	\geq
	\d_2\, \ent_\nu (f). &\qedhere
	\end{align*}
\end{proof}

%\bibliographystyle{halpha-abbrv}
%\bibliography{CouplingEntropy}
\printbibliography

\end{document}